\documentclass[12pt]{article}

\usepackage{amssymb}

\usepackage{amsfonts}

\usepackage{amsmath}

\usepackage{amsthm}

\usepackage{graphicx}

\usepackage[boxruled]{algorithm2e}

\def\Rs{{\mathbb{R} ^3}}
\def\Rp{{{\rm I\! R}^2}}
\def\N{{\rm I\! N}}

\def\o{{\omega}} 
\def\O{{\Omega}} 
\def\a{{\alpha}}
\def\b{{\beta}}
\def\g{{\gamma}}
\def\G{{\Gamma}}
\def\d{{\delta}}
\def\D{{\Delta}}
\def\q{{\theta}}
\def\r{{\rho}}
\def\s{{\sigma}}
\def\S{{\Sigma}}
\def\t{{\tau}}

\def\R{{\mathbb{R}}}

\newtheorem{thm}{Theorem}[section]

\newtheorem{lm}[thm]{Lemma}

\newtheorem{co}[thm]{Corollary}

\newtheorem{rmk}[thm]{Remark}

\newtheorem{ex}[thm]{Example}

\usepackage{xcolor}

\newcommand{\squeezelist}{\setlength{\itemsep}{0pt}}
\newtheorem{innercustomthm}{Theorem}
\newenvironment{customthm}[1]
  {\renewcommand\theinnercustomthm{#1}\innercustomthm}
  {\endinnercustomthm}
\usepackage{enumitem}
\begin{document}

\title{\bf Tailoring for Every Body:\\
Reshaping Convex Polyhedra}

\author{Joseph O'Rourke and Costin V\^\i lcu}

\date{\today}

\maketitle

%%%%%%%%%%%%%%%%%%%%%%%%%%%%%%%%%%%%%%%%%%%%%%%%%%%%%%%%%%%%%%%%%

\begin{abstract}
Given any two convex polyhedra $P$ and $Q$, we prove as one of our main results
that the surface
of $P$ can be reshaped to a homothet of $Q$ by a finite sequence of ``tailoring" steps.
Each tailoring excises a digon surrounding a single vertex and sutures the digon closed.
One phrasing of this result is that, if $Q$ can be ``sculpted" from $P$ by 
a series of slices with planes, then $Q$ can be tailored from $P$.
And there is a sense in which tailoring is finer than sculpting in that $P$
may be tailored to polyhedra that are not achievable by sculpting $P$.
It is an easy corollary 
that, if $S$ is the surface of any convex body,
then any convex polyhedron $P$ may be tailored to approximate a homothet of $S$ as closely as desired.
So $P$ can be ``whittled" to e.g., a sphere $S$.

Another main result achieves the same reshaping, but by excising more complicated shapes we call ``crests," still each enclosing one vertex.
Reversing either digon-tailoring or crest-tailoring leads to proofs that any $Q$ inside $P$ can be enlarged to $P$ by cutting $Q$ and inserting and sealing surface patches.

One surprising corollary of these results is that, for $Q \subset P$, we can
cut-up $Q$ into pieces and paste them non-overlapping onto
an isometric subset of $P$. This can be viewed as a form of ``unfolding" $Q$ onto $P$.

All our proofs are constructive, and lead to polynomial-time algorithms.

\end{abstract}

%%%%%%%%%%%%%%%%%%%%%%%%%%%%%%%%%%%%%%%%%%%%%%%%%%%%%%%%%%%%%%%%%

\section{Introduction}
\label{secIntroduction}

Let $P$ and $Q$ be convex polyhedra, each the convex hull of finitely many points in $\Rs$.
If $Q \subset P$, it is easy to see that $Q$ can be \emph{sculpted} from $P$ by ``slicing $Q$ with planes." 
By this we mean intersecting $Q$ with half-spaces each of whose plane boundary contains a face of $Q$.
If $Q \not\subset P$, we can shrink $Q$ until it fits inside.
So a \emph{homothet} of any given $Q$ can be sculpted from any given $P$, where a homothet is a copy possibly scaled, rotated, and translated.
Main results of this paper (Theorems~\ref{thmMainTailoring}, \ref{thmTailorNoSculpting}, \ref{thmCrestTailoring})
are similar claims but via ``tailorings": a homothet of any given $Q$ can be tailored from any given $P$.

With some abuse of notation, we will use the same symbol $P$ for a polyhedral hull and its boundary.
We define two types of tailoring.
A \emph{digon-tailoring} cuts off a single vertex of $P$ along a digon, and then sutures the digon closed.
A \emph{digon} is a subset of $P$ bounded by two equal-length geodesic segments that share endpoints; see Fig.~\ref{Tetra_3D}.
A \emph{geodesic segment} is a shortest geodesic between its endpoints.
A \emph{crest-tailoring} cuts off a single vertex of $P$ but via a more complicated shape we call a ``crest."
Again the hole is sutured closed.
We defer discussion of crests to Section~\ref{secCrests}.
Meanwhile, we shorten ``digon-tailoring" to simply \emph{tailoring}.
%Fig.~\ref{Tetra_3D} illustrates a digon.
%==================Figure================================
\begin{figure}
\centering
 \includegraphics[width=0.4\textwidth]{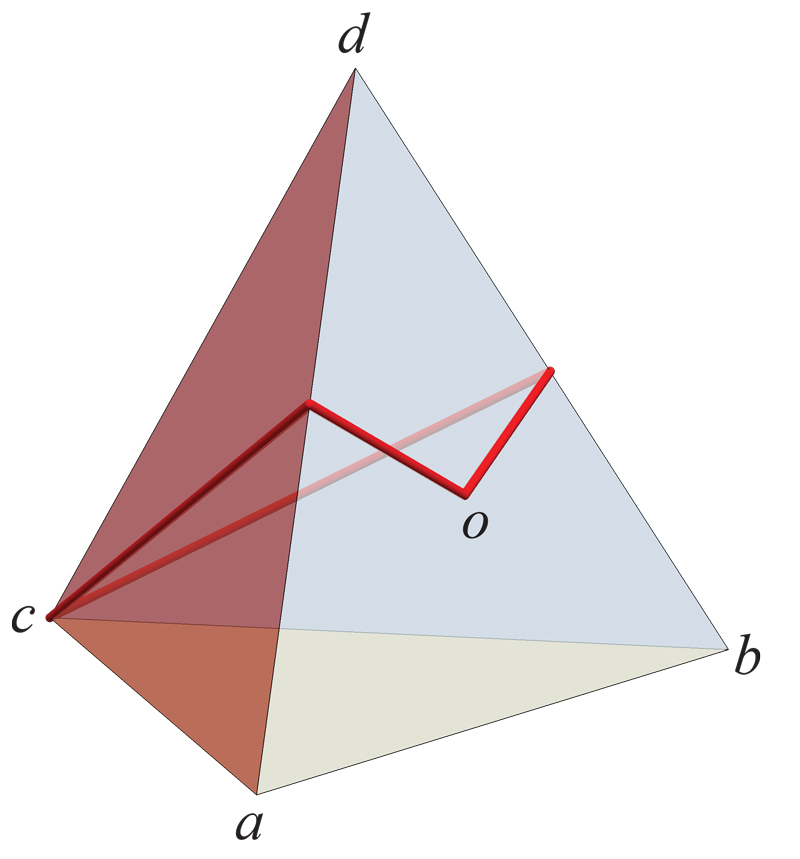}
\caption{A digon $D(c,o)$ on the regular tetrahedron $R=abcd$, surrounding vertex $d$.}
\label{fig1}\label{Tetra_3D}
\end{figure}
%==================Figure================================

Cutting out a digon means excising the portion of the surface between the geodesics, including the vertex they surround.\footnote{
An informal view (due to Anna Lubiw) is that one could pinch the surface flat in a neighborhood of the vertex, and then snip-off the flattened vertex with scissors.}
Once removed, the digon hole is closed by naturally identifying the two geodesics along their lengths.
This identification is often called ``gluing" in the literature, although we also call it ``suturing" or ``sealing." 

Throughout, we make extensive use of Alexandrov's Gluing Theorem~\cite[p.100]{code2},
which guarantees that the surface obtained after a tailoring of $P$
corresponds uniquely to a convex polyhedron $P'$.
A precise statement of this theorem, which we will abbreviate to AGT, is as follows.

\begin{customthm}{AGT}
\label{AGT} 
Let $S$ be a topological sphere obtained by gluing planar polygons (i.e., naturally identifying pairs of sides of the same length) such that 
at most $2\pi$ surface angle is glued at any point.
Then $S$, endowed with the intrinsic metric induced by the distance in $\Rp$, is isometric to a convex polyhedron $P\subset\Rs$, 
possibly degenerated to a doubly covered convex polygon. 
Moreover, $P$ is unique up to rigid motion and reflection in $\Rs$. 
\end{customthm}

Because the sides of the digon are geodesics, gluing them together to seal the hole leaves $2\pi$ angle at all but the digon endpoints.
The endpoints lose surface angle with the excision, and so have strictly less than $2\pi$ angle surrounding them.
So AGT applies and yields a new convex polyhedron.

This shows that tailoring is possible and alters the given $P$ to another convex polyhedron.
How to ``aim" the tailoring to a given target $Q$ is a long story, told in subsequent sections.

AGT is a fundamental tool in the geometry of convex surfaces and, at a theoretical level, our paper helps to elucidate its implications.
While AGT, %and its particular form of digon-tailoring, proved useful many times before, 
and its particular form of ``vertex merging" (discussed in Section~\ref{secVertexMerging}) has proved useful in several investigations,
the inverse problem we treat here has, to our knowledge, never been considered before as the central object of study.

One remark concerning AGT.
Alexandrov's proof of his celebrated theorem is a difficult existence proof and gives little hint of the structure of the polyhedron guaranteed by the theorem. 
And as-yet there is no effective procedure to construct the three-dimensional shape of the polyhedron guaranteed by his theorem.
There are numerical approximations (see~\cite{JOR07}), but realistically, only small or highly symmetric examples can be reconstructed, e.g., Fig.~\ref{figPentahedron} (ahead).

%%%%%%%%%%%%%%%%%%%%%%%%%%%%%%%%%%%%%%%%%%%%%%%%%%%%%%%%%%%%

\subsection{Examples}

Before discussing background context, we present several examples.
Throughout we let $xy$ denote the line segment between points $x$ and $y$, $x,y \in \Rs$.
Also we make extensive use of vertex curvature.
The {\it discrete} (or {\it singular}) {\it curvature} $\o(v)$ at a vertex $v \in P$ is the angle deficit: $2\pi$ minus the sum of the face angles incident to $v$.
%By the Gauss-Bonnet theorem, $\sum_v \o(v) = 4\pi$. 

\begin{ex}
\label{tetra-deg}
Let $R=abcd$ be a regular tetrahedron, and let $o$ be the center of the face $abd$.
Cut out the digon on $R$ between $c$ and $o$ ``encircling" $d$, and zip it closed.

The unfolding $T$ of $R$ with respect to $c$ is a planar regular triangle $c_a c_b c_d$ with center $o$.
Cutting out that digon from $R$ is equivalent to removing from $T$ the isosceles triangle $o c_a c_b$. See Fig.~\ref{figTetraTailoring}(a).
We zip it closed by identifying the digon-segments $c_a o$ and $c_b o$, and refolding the remainder of $T$ by re-identifying $c_a b$ and $b c_d$,  and $c_b a$ and $a c_d$.
One can easily see that the result is the doubly covered kite $K=aob c_d$, shown in Fig.~\ref{figTetraTailoring}(b).
\end{ex}

%==================Figure================================
\begin{figure}
\centering
 \includegraphics[width=\textwidth]{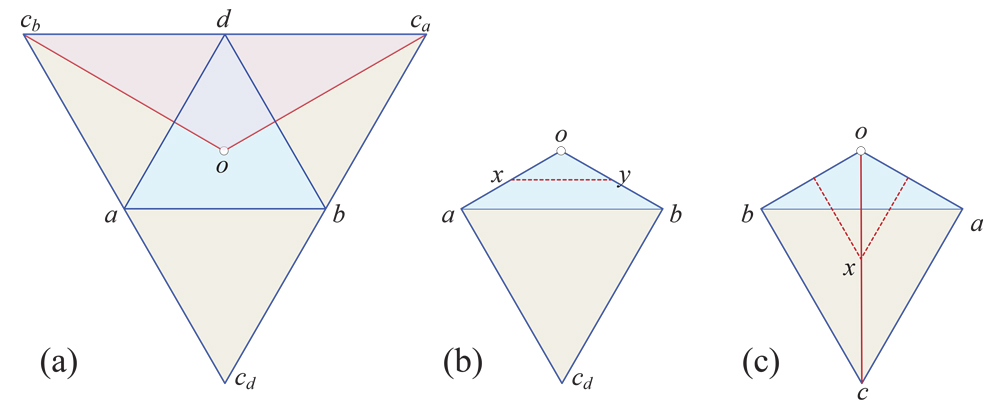}
\caption{Illustrations for Examples~\ref{tetra-deg}--\ref{deg-vol}.}
\label{fig2}\label{figTetraTailoring}
\end{figure}
%==================Figure================================

\begin{ex}
\label{deg-vol}
We further tailor the doubly covered kite $K$ obtained in Example~\ref{tetra-deg}.
Excising a digon encircling $o$, between points $x \in oa$ and $y \in ob$, and zipping closed, 
yields a doubly covered pentagon,  see Fig.~\ref{figTetraTailoring} (b).

On the other hand, excising a digon encircling $o$, between corresponding points $x,y \in oc$ on different sides of $K$ (see Fig.~\ref{figTetraTailoring} (c)), and zipping closed, 
provides a non-degenerate pentahedron, illustrated in Fig.~\ref{figPentahedron}.
\end{ex}

%==================Figure================================
\begin{figure}
\centering
 \includegraphics[width=0.9\textwidth]{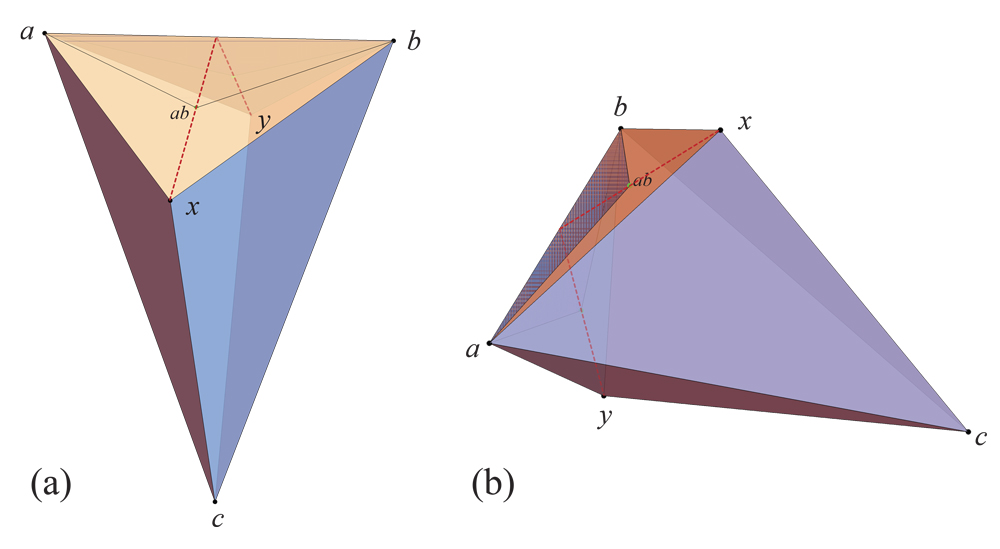}
\caption{The non-degenerate pentahedron obtained in Example \ref{deg-vol}.}
\label{fig3}\label{figPentahedron}
\end{figure}
%==================Figure================================

\begin{ex}
\label{iso_tetra}
A digon may well contain several vertices, but for our digon-tailoring we only consider those with at most one vertex.
The limit case of a digon, containing no vertex, is an edge between two vertices.
(This will play a role in Section~\ref{secTailor_2}.)
In this case, gluing back along the cut would produce the original polyhedron, but we can as well zip closed from another starting point.
For example, cutting along an edge of an isosceles tetrahedron and carefully choosing the gluing provides a doubly covered rectangle.
See Fig.~\ref{figIsoscTetra}.
\end{ex}

%\red{[This is a digon, but not a digon-tailoring, according to the definition at page 2. We could call this operation a ``special (digon) tailoring''.
%Or we could alter the definition of ``digon-tailoring'' to say ```at most one vertex'', not ``a single vertex'' as it is now (I'd choose this).
%Perhaps we could also add a phrase saying that ``a digon may well contain several vertices, but for our digon-tailoring we only consider those with at most one vertex.
%As we prove later on, this suffices for our purposes.'']}

%==================Figure================================
\begin{figure}
\centering
 \includegraphics[width=1.0\textwidth]{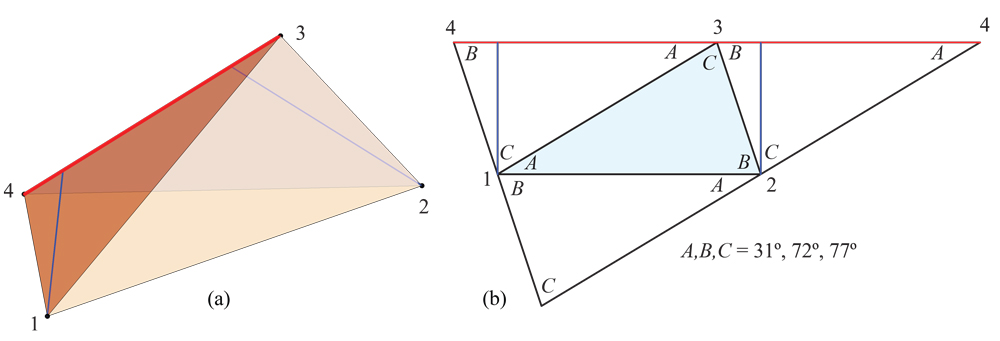}
\caption{(a)~An isosceles tetrahedron: All four faces are congruent.
Cutting along $v_3 v_4$ and regluing the two halves of that
slit differently, creasing at the blue segments,
yields a doubly-covered rectangle, as shown in~(b).}
\label{fig44}\label{figIsoscTetra}
\end{figure}
%==================Figure================================

By Alexandrov's Gluing Theorem (AGT), this limit-case tailoring only works between vertices of curvatures $\geq \pi$.

In view of Fig.~\ref{figTetraTailoring}(c) and Fig.~\ref{figPentahedron},
it is clear that, even though tailoring is area decreasing, it is not necessarily volume decreasing.

%%%%%%%%%%%%%%%%%%%%%%%%%%%%%%%%%%%%%%%%%%%%%%%%%%%%%%%

\subsection{Vertex Merging}
\label{secVertexMerging}

Digon-tailoring is, in some sense, the opposite of {\it vertex merging}, a technique introduced by A.~D.~Alexandrov~\cite[p.~240]{code2},
and subsequently used later by others, see e.g.~\cite{Zal}, \cite{OV}, \cite{o-vtcp-2020}.
The technique we introduce for enlarging surfaces, in Section~\ref{enlarging}, is a generalization of vertex merging.

Consider two vertices $v_1, v_2$ of $P$ of curvatures $\o_1, \o_2$, with $\o_1 + \o_2 < 2 \pi$, and cut $P$ along a 
geodesic segment $\g$ joining $v_1$ to $v_2$. 
Construct a planar triangle $T = \bar{v}' \bar{v}_1 \bar{v}_2$ of base length $|\bar{v}_1 - \bar{v}_2|=|\g|$
and the base angles equal to $\o_1 /2$ and $\o_2 /2$ respectively. 
Glue two copies of $T$ along the corresponding lateral sides, and further glue the two bases of the copies to the two ``banks" of the cut of $P$ along $\g$. 
By Alexandrov's Gluing Theorem (AGT), the result is a convex polyhedral surface $P'$. 
On $P'$, the points (corresponding to) $v_1$ and $v_2$ are no longer vertices because exactly the angle deficit at each has been
sutured in; they have been replaced by a new vertex $v'$ of curvature $\o' = \o_1 +\o_2$.
See Fig.~\ref{figVertMerge}.
\begin{figure}
\centering
 \includegraphics[width=1.0\textwidth]{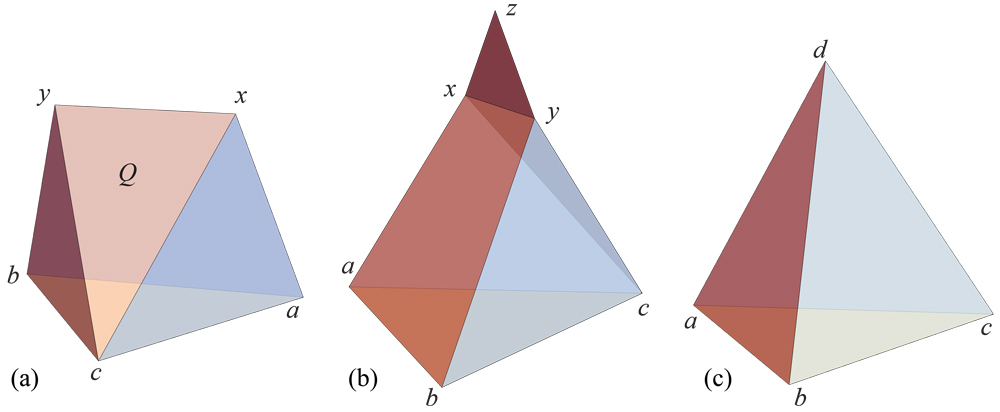}
\caption{(a)~$Q$ is a $5$-vertex polyhedron. Its base $abc$ is an equilateral triangle.
(b)~Vertex merging $x$ and $y$ by gluing two $\triangle$s $xyz$.
(c)~The merging reduces $Q$ to a regular tetrahedron.
}
\label{figPent}\label{figVertMerge}
\end{figure}

Tailoring a vertex $v$ identifies a digon $D=(x,y)$ enclosing $v$, with two geodesics from $x$ to $y$. 
In general, neither $x$ nor $y$ is a vertex before tailoring, but they become vertices after removing $D$, thereby increasing the number of vertices of $P$ by $1$.
If instead both $x$ and $y$ are already vertices, then the number of vertices of $P$ is decreased by $1$ from tailoring.
The challenge answered in our work is to direct tailoring to ``aim" from one polyhedon $P$ to the target $Q$.

%%%%%%%%%%%%%%%%%%%%%%%%%%%%%%%%%%%%%%%%%%%%%%%%%%%%%%%%%%

\subsection{Summary of Results}

Here we list our
main theorems, each with a succinct (and at this stage, quite approximate) summary of their claims.
\begin{itemize}
\squeezelist
\item Theorem~\ref{thmMainTailoring}: 
$Q$ may be digon-tailored from $P$, tracking a sculpting of $P$ to $Q$.
\item Theorem~\ref{thmTailorNoSculpting}: 
A different proof of a similar
result, that $P$ may be digon-tailored to a homothet of $Q$, but this time without sculpting.
\item Theorem~\ref{thmCrestTailoring}: 
$P$ may be crest-tailored to $Q$.
\item Theorems~\ref{SliceAlgorithm} and~\ref{thmAlgorithmNoSculpt} and \ref{thmCrestTailoring}: 
Tailoring algorithms have time-complexity $O(n^4)$.
\item Theorem~\ref{Enlarge_Algorithm1}: Reversing tailoring yields procedures for enlarging $Q$ inside $P$ to match $P$.
\item Theorem~\ref{reshape}: Tailoring and enlarging allow reshaping convex polyhedra.
\item Theorem~\ref{thmP-unf}: $Q$ may be cut up and ``unfolded" isometrically onto $P$.

\end{itemize}

\noindent
Along the way to our central theorems, we obtain results not directly related to AGT:
\begin{itemize}
\item Theorem~\ref{rigid}: If two convex polyhedra with the same number of vertices match on all but 
the neighborhoods of one vertex, then they are congruent.
\item Theorem~\ref{thmDomePyr}: Every ``g-dome" can be partitioned into a finite sequence of pyramids by planes through its base edges.
\end{itemize}

The above results raise several open problems of various natures, either scattered along the text or presented in the last section of the paper.

Finally, we sketch the logic behind the first result listed above, Theorem~\ref{thmMainTailoring}.
Start with $Q$ inside $P$, and imagine a sequence of slices by planes that sculpt $P$ to $Q$. 
Theorem~\ref{general_slice} shows how to digon-tailor one such slice, which then establishes the claim that we can tailor $P$ to $Q$.
Theorem~\ref{general_slice} is achieved by first slicing off shapes we call ``g-domes," 
and then showing in Theorem~\ref{thmDomePyr} that every g-dome can be reduced to its base by slicing off pyramids, i.e., by vertex truncations.
Lemma~\ref{VertexTruncation} shows that such vertex truncations can be achieved by tailoring.
And the proof of Lemma~\ref{VertexTruncation} relies on the rigidity established by Theorem~\ref{rigid}.
So the path  of logic is:
\begin{align*}
\textrm{plane slice} \;\to\; \textrm{g-domes} \;\to\; \textrm{pyramids}  \;\to\; & \textrm{digon removals} \;.\\
& \uparrow\\
& \textrm{rigidity}
\end{align*}

%%%%%%%%%%%%%%%%%%%%%%%%%%%%%%%%%%%%%%%%%%%%%%%%%%%%%%%%%%%%%%%%%

\section{Preliminaries}
\label{secPreliminaries}

In this section we present some basic properties of cut loci on convex polyhedra, and other geometric tools needed subsequently.
The reader might skim this section and return to it as the tools are deployed.

The \emph{cut locus} $C(x)$ {\it of the point $x$} on a convex polyhedron $P$ is the closure of the set of points to which there are more than one shortest path from $x$.
This concept goes back to Poincar\'e. 
It has been studied algorithmically since~\cite{SS} (there, the cut locus is called the ``ridge tree'').

\begin{lm}
\label{basic}
(i) $C(x)$ has the structure of a finite $1$-dimensional simplicial complex which is a tree.
Its leaves (endpoints) are vertices of $P$, and all vertices of $P$, excepting $x$ (if it is a vertex) are included in $C(x)$. 
All points interior to $C(x)$ of tree-degree $3$ or more are known as \emph{ramification points} of $C(x)$.\footnote{%
In some literature, these points are called ``branch points" or ``junctions" of $C(x)$.}
All vertices of $P$ interior to $C(x)$ are also considered as ramification points.

(ii) Each point $y$ in $C(x)$ is joined to $x$ by as many geodesic segments as the number of connected
components of $C(x) \setminus {y}$.
For ramification points in $C(x)$, this is precisely their degree in the tree.

(iii) The edges of $C(x)$ are geodesic segments on $P$.

(iv) Assume the geodesic segments $\g$ and $\g'$ (possibly $\g = \g'$) from $x$ to $y \in C(x)$ are bounding a domain $D$ of $P$, 
which intersects no other geodesic segment from $x$ to $y$.
Then there is an arc of $C(x)$ at $y$ which intersects $D$ and it bisects the angle of $D$ at $y$.
\end{lm}

\begin{proof} 
The statements (i)-(ii) and (iv) are well known.
The statement (iii) is Lemma 2.4 in \cite{aaos-supa-97}.
\end{proof}

The following is Lemma~4 in~\cite{INV}.

\begin{lm}
\label{path}
If $C(x)$ is a path, the polyhedron is a doubly-covered (flat) convex polygon, with $x$ on the rim.
\end{lm}

Next we introduce a general method for unfolding any convex polyhedron $P$ to a simple (non-overlapping) polygon in the plane. 
We use this subsequently largely because of its connection to the cut locus.

To form the {\it star unfolding} of a $P$ with respect to $x$, one cuts $P$ along the geodesic segments (supposed unique) from $x$ to every vertex of $P$.
The idea goes back to Alexandrov \cite{code2}; the  non-overlapping of the unfolding was established in \cite{ao}, where the next result was also proved.
See Fig. \ref{StarUnfCube}.

\begin{lm}
\label{*unfCL}
Let $S_P=S_P(x)$ denote the star unfolding of $P$ with respect to $x \in P$.
Then the image of $C(x)$ in $S_P$ is the restriction to $S_P$ of the Voronoi diagram of the images of $x$.
\end{lm}

\begin{figure}
\centering
 \includegraphics[width=1.0\textwidth]{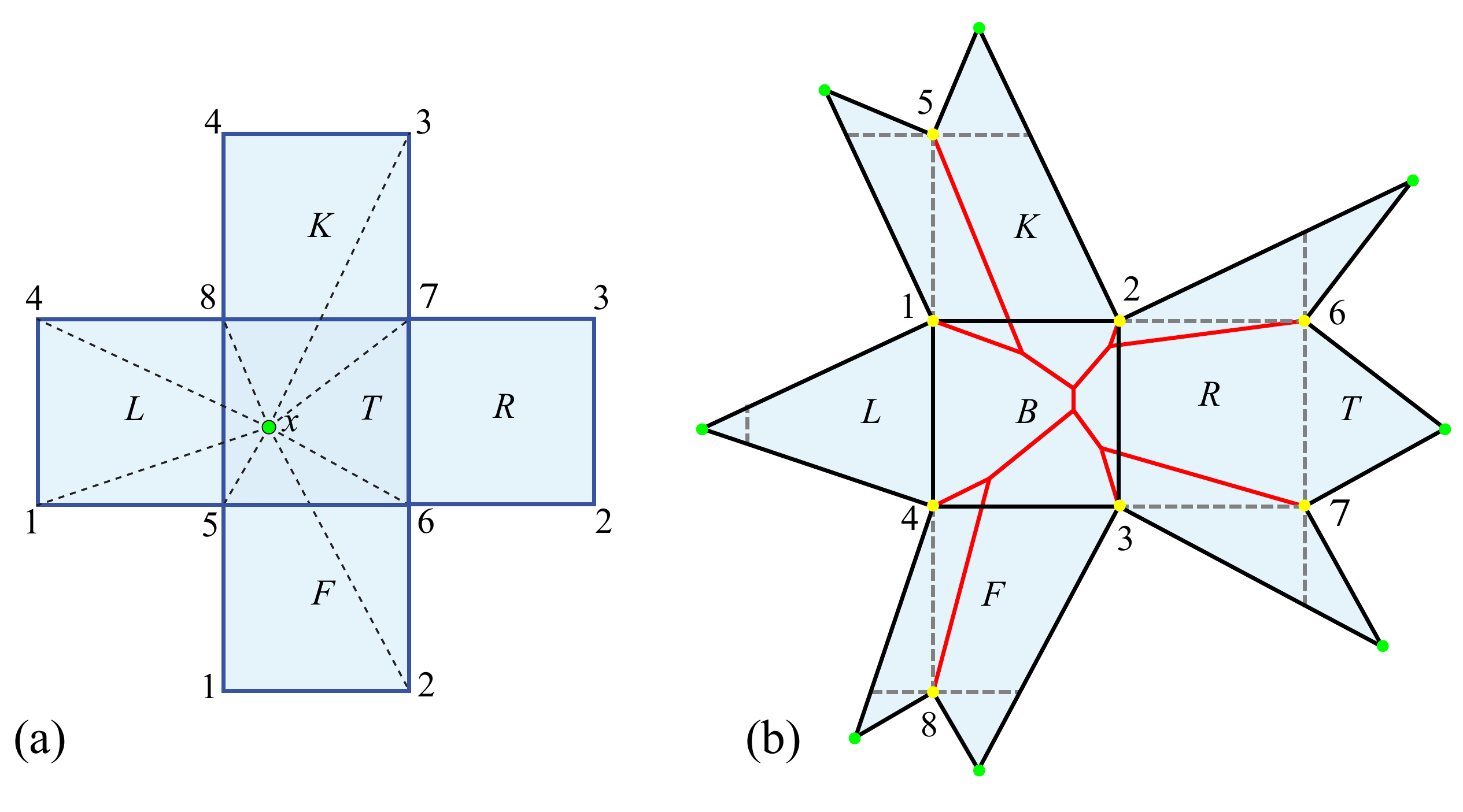}
\caption{(a)~Cut segments to the $8$ vertices of a cube from a point $x$
on the top face. T, F, R, K, L, B $=$ Top, Front, Right, Back, Left, Bottom.
(b)~The star-unfolding from $x$. The cut locus $C(x)$ (red) is the Voronoi
diagram of the $8$ images of $x$ (green).}
\label{StarUnfCube}
\end{figure}

The following result is very similar to, but more general than, Lemmas~2 and~3 in \cite{INV}. 
Its proof is a straightforward application of Lemmas \ref{basic} and \ref{*unfCL}, and will be omitted.

\begin{lm}
\label{Truncation}
Consider two segments $\g, \g'$ between points $x \in P$ and $y \in C(x)$. 
Cut along $\g \cup \g'$ and zip closed the two parts, starting from $x$.
Let $Q$ be one of the two resulting convex polyhedra.
Then the cut locus $C(x,Q)$ of $x$ on $Q$ is (isometric to) the truncation of the cut locus $C(x, P)$.
\end{lm}

In several proofs we invoke \emph{Cauchy's Arm Lemma}.
One form of the lemma says that if we have an open convex chain in the plane, $(x_1,\ldots,x_n)$, and the angles at the vertices $x_i$ are opened but not beyond $\pi$, 
maintaining all edge lengths $|x_i x_{i+1}|$ fixed, then the the distance between the chain endpoints $|x_1 x_n|$ lengthens.
In Lemma~\ref{pyramid} we will need an extension of Cauchy's lemma to reflex (greater than $\pi$) openings, described in~\cite{o-ecala-01}. This will be detailed in that proof.

The next elementary result assures the angle increase, 
in the frameworks in which we will apply Cauchy's Arm Lemma.

\begin{lm}
\label{lemAngles}
Consider three rays $r_1,r_2,r_3$ in $\R^3$, emanating from the point $w$, and put $\t_i=\angle(h_i, h_{i+1})$, with $3+1 \equiv 1$ mod $3$.
Then $\theta_1 \leq \theta_2 + \theta_3$.
\end{lm}

\begin{proof}
Imagine a unit sphere $S$ centered on $w$ and let $\{s_i \} = r_i \cap S$, and use $\r$ to indicate spherical distance.
Then the claim of the lemma is the triangle inequality for spherical distances: $\r(s_1,s_2) \le \r(s_1,s_3) + \r(s_2,s_3)$.
\end{proof}

%%%%%%%%%%%%%%%%%%%%%%%%%%%%%%%%%%%%%%%%%%%%%%%%%%%%%%%%%%%%%%%%%

\section{Domes and pyramids}
\label{domes}

One of our goals in this paper, achieved in Theorem~\ref{main}, is to show that if $Q$ can be obtained from $P$ by sculpting, then it can also be obtained from $P$ by tailoring.
The key step (Theorem~\ref{general_slice}) repeatedly slices off shapes we call g-domes.
Each g-dome slice can itself be achieved by slicing off pyramids, i.e., by suitable vertex truncation.
Lemma~\ref{VertexTruncation} will show that slicing off a pyramid can be achieved by tailoring, and thus leading to Theorem~\ref{main}.
In this section we establish that g-domes can be viewed as composed of stacked pyramids.

As usual, a \emph{pyramid} $P$ is the convex hull of a convex polygon \emph{base} $X$,
and one vertex $v$, the \emph{apex} of $P$, that does not lie in plane of $X$. The degree of $v$ is the number of vertices of $X$. 

A \emph{dome} is a convex polyhedron $G$ with a distinguished face $X$, the \emph{base}, and such that every other face of $G$ shares a (positive-length) edge with $X$.
Domes have been studied primarily for their combinatorial~\cite{Eppstein},~\cite{Epp-Loff} or unfolding~\cite{DOR} properties.
In~\cite{Eppstein} they are called ``treetopes" because removing the base edges from the $1$-skeleton leaves a tree, which the author calls the \emph{canopy}.\footnote{
These polyhedra are not named in~\cite{Epp-Loff}.}
Here we need a slight generalization.

A \emph{generalized-dome}, or \emph{g-dome} $G$, has a base $X$, with every other face of $G$ sharing either an edge or a vertex with $X$.
Every dome is a g-dome, and it is easy to obtain every g-dome as the limit of domes.
An example is shown in Fig.~\ref{figg-dome}, which also shows that removing base edges from the $1$-skeleton does not necessarily leave a tree: $(v_1,x_2,v_2)$ forms a cycle.
Let us define the \emph{top-canopy} $T$ of a g-dome $G$ as the graph that results by deleting from the $1$-skeleton of $G$ all base vertices and their incident edges. 
In Fig.~\ref{figg-dome} the top-canopy is $v_1 v_2$.

%See Fig.~\ref{figg-dome}
%==================Figure================================
\begin{figure}
\centering
 \includegraphics[width=0.5\textwidth]{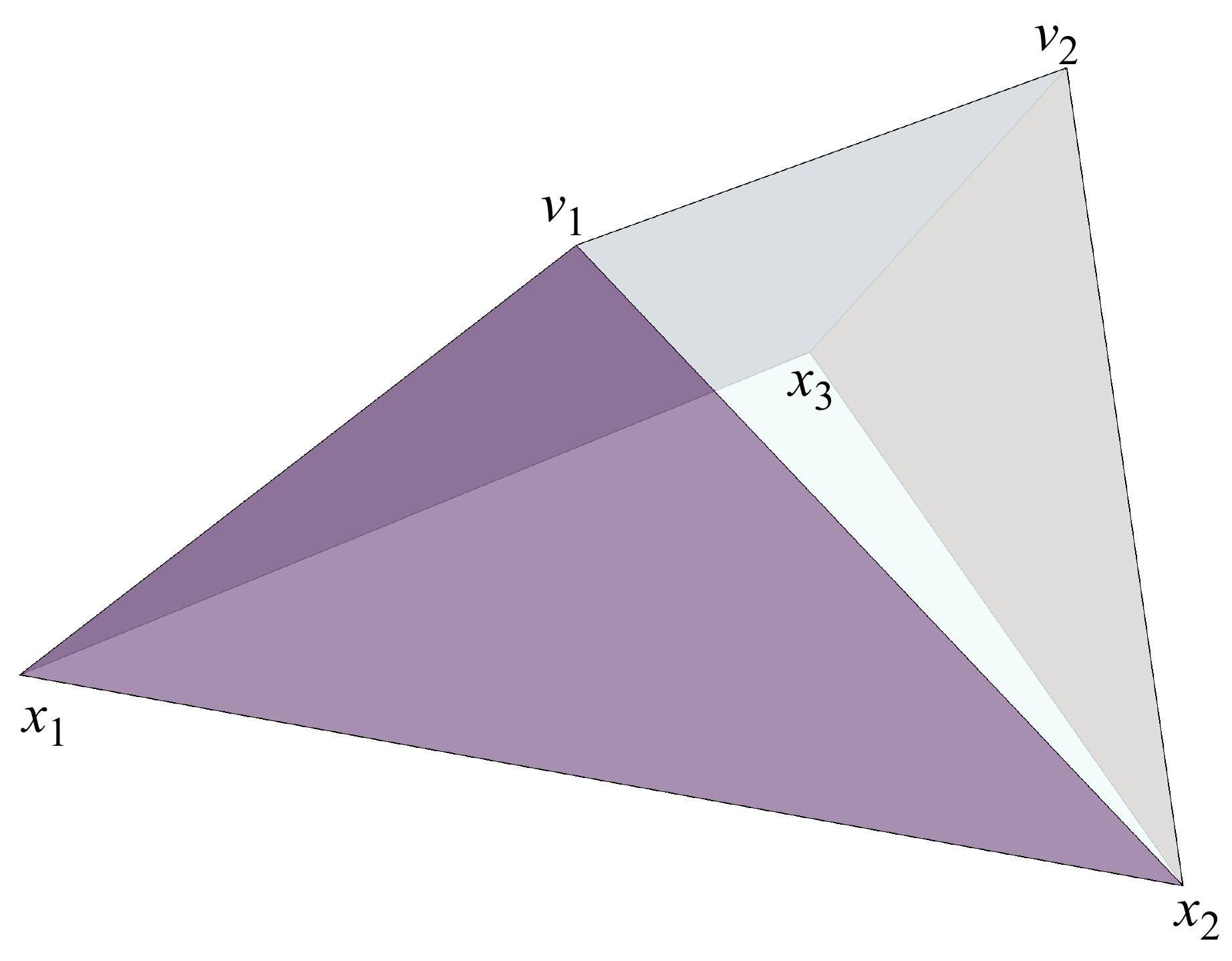}
\caption{A g-dome with base $x_1,x_2,x_3$ and top-canopy $v_1 v_2$.
}
\label{figg-dome}
\end{figure}
%==================Figure================================

\begin{lm}
\label{canopy}
The top-canopy $T$ of a g-dome $G$ is a tree.
\end{lm}
\begin{proof}
If $G$ is a dome, the claim follows, because even including the edges incident to the base $X$ results in a tree, and removing those edges leaves a smaller tree.

If $G$ is not a dome, then slice it with a plane parallel to, and at small distance above, the base. 
The result is a dome, and we can apply the previous reasoning.
\end{proof}

%==================Figure================================
\begin{figure}
\centering
 \includegraphics[width=1.0\textwidth]{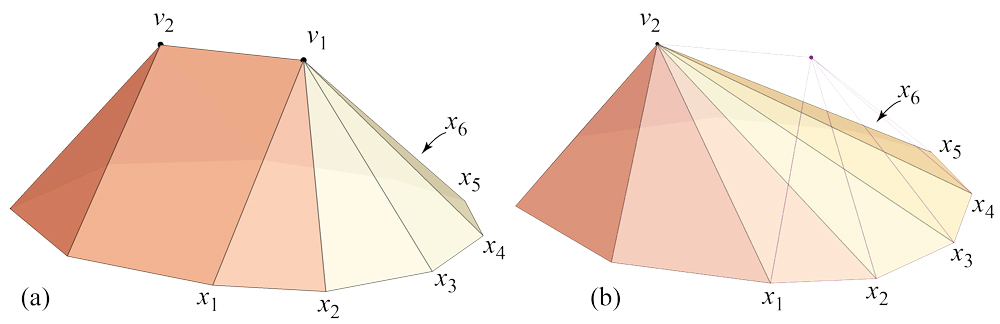}
\caption{
(a)~A dome $G$. Vertex $v_1$ of degree $k_1=6$ is adjacent to $X$, i.e.,
there are edges $v_1 x_i$. 
(b)~$G'$ after removal of $v_1$.}
\label{figDomeSlice_begend}
\end{figure}
%==================Figure================================

\begin{thm}
\label{pyra-dome}
\label{thmDomePyr}
Every g-dome $G$ of base $X$ can be partitioned into a finite sequence of $n$ pyramids $P_i$ with the following properties:
\begin{itemize}
\squeezelist
\item Each $P_i$ has a common edge with $X$.
\item Each $G_j = G \setminus \bigcup_{i=1}^j P_i$ is a g-dome, for all $j=1,...,n$.
\item The last pyramid $P_n$ in the sequence has the same base $X$ as $G$.
\end{itemize}
\end{thm}

The proof is a double induction, and a bit intricate.
One induction simply removes one vertex $v_1$ from the top-canopy.
We will illustrate the proof with the example in Fig.~\ref{figDomeSlice_begend}.
The second induction, inside the first one, reduces the degree of $v_1$ to achieve removal of $v_1$, at the cost of increasing the degree of $v_2$.

%==================Figure================================
\begin{figure}
\centering
 \includegraphics[width=1.0\textwidth]{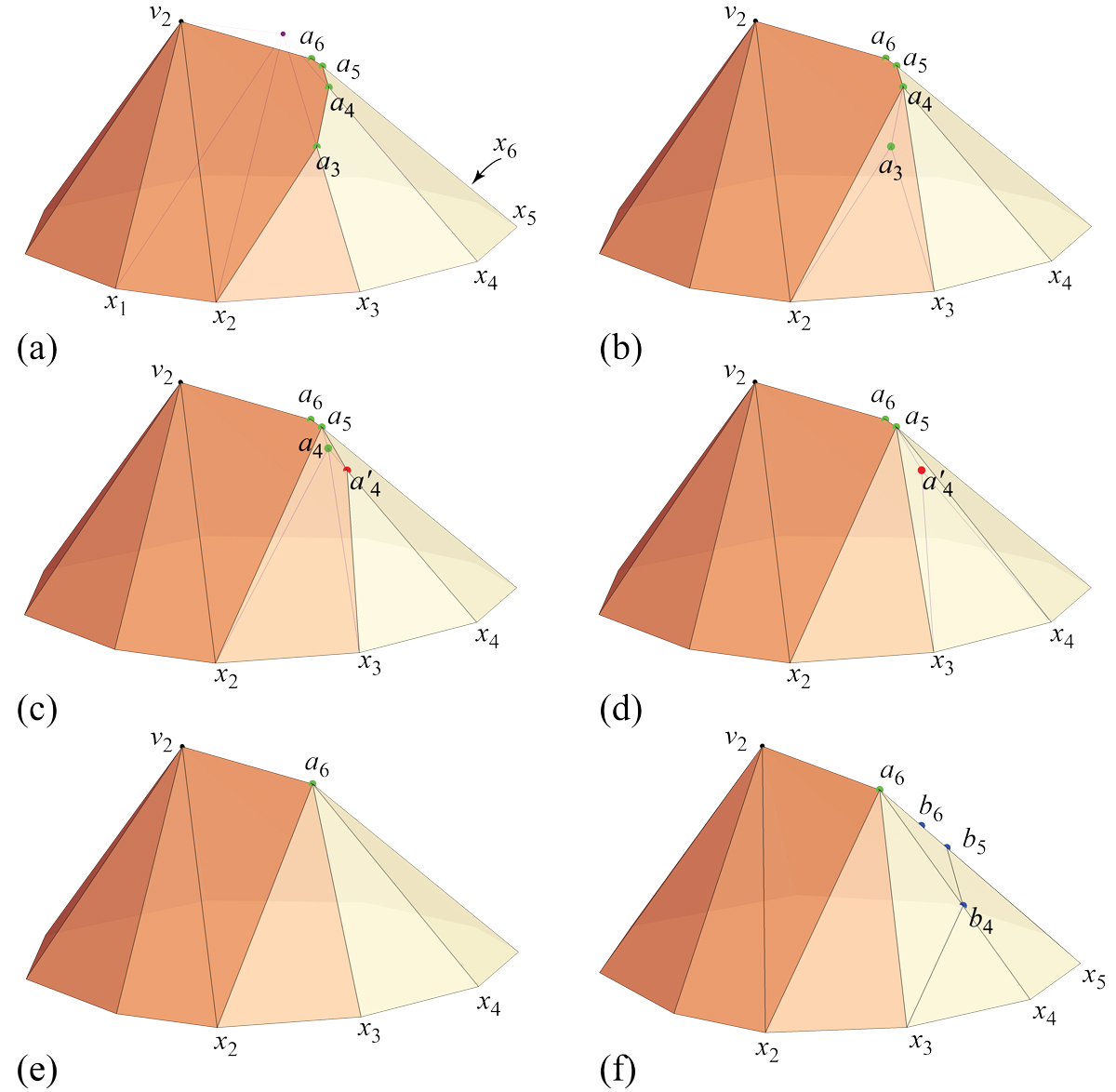}
\caption{
(a)~After slicing by the plane $\Pi_1=x_1x_2v_2$.
(b)~Removing $a_3$.
(c)~Removing $a_4$ creates $a'_4$.
(d)~Removing $a'_4$.
(e)~$a_6$ has replaced $v_1$ with one lower degree.
(f)~Intersection points with $\Pi_2=x_2 x_3v_2$.
}
\label{figSliceAway_123456}
\end{figure}
%==================Figure================================

\noindent
\begin{proof}
Let $m$ be the number of vertices in the top-canopy $T$ of $G$.
If $m=1$, $G$ is already a pyramid, and we are finished. So assume $G$'s top-canopy $T$ has at least two vertices.
Choose $v_1$ to be a leaf of $T$ given by Lemma~\ref{canopy}, and $v_2$ its unique parent.
Let $v_1$ be adjacent to $k$ vertices of $X$.
If $G$ is a dome, $v_1$ has degree $k+1$; if $G$ is a g-dome, then possibly $v_1$ has degree $k+2$.
Since the later case changes nothing in the proof, we assume for the simplicity of exposition that $G$ is a dome.
Our goal is to remove $v_1$ through a series of pyramid subtractions.

Let the vertices of $X$ adjacent to $v_1$ be $x_1,x_2,\ldots,x_k$. 
Let $\Pi_1$ be the plane $x_1x_2v_2$.
This plane cuts into $G$ under $v_1$, and intersects the edges $v_1 x_i$, $i \ge 3$, in points $a_i$.
In Fig.~\ref{figSliceAway_123456}(a), those points are $a_3,a_4,a_5,a_6$.
Remove the pyramid whose apex is $v_1$ and whose base (in our example) is $x_1,x_2,a_3,a_4,a_5,a_6,v_2$.

We now proceed to reduce the chain of new vertices $a_3,\ldots,a_k$ one-by-one until only $a_k$ remains.

First, with the plane $x_2x_3a_4$, we slice off the tetrahedron whose degree-$3$ apex is $a_3$; Fig.~\ref{figSliceAway_123456}(b).
Next, with the plane $x_2x_3a_5$, we slice off the pyramid with apex $a_4$. 
Unfortunately, because $a_4$ has degree-$4$, this introduces a new vertex $a'_4$;  Fig.~\ref{figSliceAway_123456}(c).
So next we slice with $x_3x_4a_5$ to remove the tetrahedron whose degree-$3$ apex is $a'_4$; Fig.~\ref{figSliceAway_123456}(d).
Continuing in this manner, alternately removing a tetrahedron followed by a pyramid with a degree-$4$ apex, we reach Fig.~\ref{figSliceAway_123456}(e).

Note that $v_1$ was connected to $k=6$ vertices of $X$, but $a_6$ is only connected to $5$:
the connection of $v_1$ to $x_1$ has in a sense been transferred to $v_2$. 
In general, $a_k$ has degree one less than $v_1$'s degree, and the degree of $v_2$ has increased.

Now we repeat the process, starting by slicing with $\Pi_2 = x_2x_3v_2$, which intersects the $a_k x_i$ edges at $b_4,\ldots,b_k$.
We remove the pyramid apexed at $a_6$ with base (in our example) of $x_2,x_3,b_4,b_5,b_6,v_2$; Fig.~\ref{figSliceAway_123456}(f).
The same methodical technique will remove all but the last new vertex $b_k$, which replaces $a_k$ but has degree one smaller.

Continuing the process, slicing with $\Pi_i = x_ix_{i+1}v_2$, up to $i=k-1$, will lead to the complete removal of $v_1$, as previously illustrated 
in Fig.~\ref{figDomeSlice_begend}(b), completing the inner induction.
Induction on the number of vertices of the g-dome then completes the proof.

With $G_0=G$, each $G_j$ is the intersection of $G_{k-1}$ with a closed half-space containing a base edge, so it is convex for all $j=1,...,n$. 
Indeed each $G_j$ is a g-dome, because all untouched faces continue to meet $X$ in either an edge or a vertex, and new faces always share an edge with $X$.
\end{proof}

\begin{rmk}
The partition of a g-dome into pyramids given by Theorem~\ref{pyra-dome} is not unique. For our example in Fig.~\ref{figDomeSlice_begend}(a), 
we finally get the pyramid apexed at $v_2$ in Fig.~\ref{figDomeSlice_begend}(b), but we could as well have ended with a pyramid apexed at $v_1$.
\end{rmk}

\begin{rmk}
The partition of a g-dome into pyramids given by Theorem~\ref{pyra-dome} has special properties, such as: every slice plane containing a base edge, and
every intermediate shape being a g-dome.
Without those properties, %claiming
proving that every g-dome may be partitioned into pyramids, would be easier.
\end{rmk}

%\noindent\blue{Partitioning is a rich research topic in convexity. Should we mention this here, with a very few references, if possible about polyhedra?
%It could bring more weight/motivation to this section.}
%\red{I'm hesitant, because of the specialized nature of the theorem.
%We needed special properties, such as: every intermediate shape being 
%a g-dome. Without those, for example claiming that
%every g-dome may be partitioned into pyramids, would be easier.
%Perhaps only the claim that every g-dome (and therefore every dome) may be partitioned into pyramids is a result one could imagine being used elsewhere.}

\noindent
We will see in Section~\ref{secTailor_1} that one g-dome of $O(n)$ vertices reduces to $O(n^2)$ pyramids of constant size,
and $O(n)$ pyramids each of size $O(n)$.

%%%%%%%%%%%%%%%%%%%%%%%%%%%%%%%%%%%%%%%%%%%%%%%%%%%%%%%%%%%%%%%%%

%\section{A partion algorithm}
%\label{Partition_1}

%\red{Dome partition algorithm + complexity?}
%\blue{[This is now elsewhere, isn't it? Absorbed into some other section.]}

%\red{[It could only be at (8.1), page 34, but there only one line mentions it. If I understand well, it is not included there.]}

%%%%%%%%%%%%%%%%%%%%%%%%%%%%%%%%%%%%%%%%%%%%%%%%%%%%%%%%%%%%%%%%%

\section{A rigidity result}
\label{secRigidity}
In this section we present a technical result for later use, which may be of independent interest.
The theorem says that two convex polyhedra that are isometric on all but the neighborhoods of one vertex are in fact congruent.
We also show this result cannot be strengthened: two convex polyhedra can differ in the neighborhoods of just two vertices.

\begin{thm}
\label{rigid}
Assume $P,Q$ are convex polyhedra with the same number of vertices, such that 
there are vertices $p\in P$ and $q\in Q$, and respective neighborhoods $N_p \subset P$, $N_q \subset Q$ not containing other vertices, 
and an isometry $\iota : P\setminus N_p \to Q \setminus N_q$.
Then $P$ is congruent to $Q$.
\end{thm}

\begin{proof}
The existence of $\iota$ on all but neighborhoods of $p$ and $q$ yields, in particular, that the curvatures 
$\o_P(p)$ of $P$ at $p$ and  $\o_Q(q)$ of $Q$ at $q$ are equal, to satisfy the curvature sum of $4\pi$ (by Gauss-Bonnet). 

Take a point $x \in P$ joined to each vertex of $P$ by precisely one geodesic segment.
Such an $x$ is easily found, because it is equivalent to the fact that no vertex of $P$ is interior to $C(x)$.
Moreover, we may choose $x$ such that $\iota(x)$ has the same property on $Q$.

Denote by $u$ the ramification point of $C(x)$ neighboring $p$ in $C(x)$, i.e., the ramification point of degree $\ge 3$ closest to $p$.
Let $v$ be the similar ramification point of $C(\iota(x))$ neighboring $q$ in $C(\iota(x))$.
Since $N_p$ and $N_q$ are small, we may assume they are disjoint from 
$u$ and $v$ and all the segments described above.

Star unfold $P$ with respect to $x$, and $Q$ with respect to $\iota(x)$, and denote by $\bar{P}$ and  $\bar{Q}$ the resulting planar polygons.
We'll continue to use the symbols $p$ and $q$, $u$ and $v$ to refer to the corresponding points in $\bar{P}$ and $\bar{Q}$ respectively.
Let $x_i$, $i=1,2$ be the images of $x$ surrounding $p$ in $\bar{P}$, and $\iota(x_i)$ the similar images in $\bar{Q}$.
See Fig.~\ref{figTent_star_Cx}(a,b).

By hypothesis, we have respective neighborhoods $\bar{N}_p \subset \bar{P}$ and $\bar{N}_q \subset \bar{Q}$ 
and an isometry $\bar{\iota}$ induced by $\iota$, with $\bar{\iota} : \bar{P} \setminus \bar{N_p} \to \bar{Q} \setminus \bar{N_q}$. 
Thus in Fig.~\ref{figTent_star_Cx}(b), all of $\bar{P}$ outside of the wedge $(x_1,u,x_2)$ is identical in $\bar{Q}$.
Therefore the triangles  ${x}_1{u}{x}_2$ and ${x}'_1 {v} {x}'_2$ are congruent.
Moreover, ${p}$ lies on the bisector of the angle $\angle {x}_1{u}{x}_2$, and ${q}$ lies on the bisector of the angle $\angle {x}'_1 {v} {x}'_2$.
Since $\angle  {x}_1 {p} {u} = \angle  {x}'_1 {q} {v} = \pi - \frac{1}{2}\o(p)= \pi - \frac{1}{2}\o(q)$, $p$ and $q$ are uniquely determined.
Consequently,  $\bar{P}$ and $\bar{Q}$ coincide, and refolding according to the same gluing identifications leads to congruent $P$ and $Q$.
\end{proof}

%See Fig.~\ref{figTent_star_Cx}
%==================Figure================================
\begin{figure}
\centering
 \includegraphics[width=1.0\textwidth]{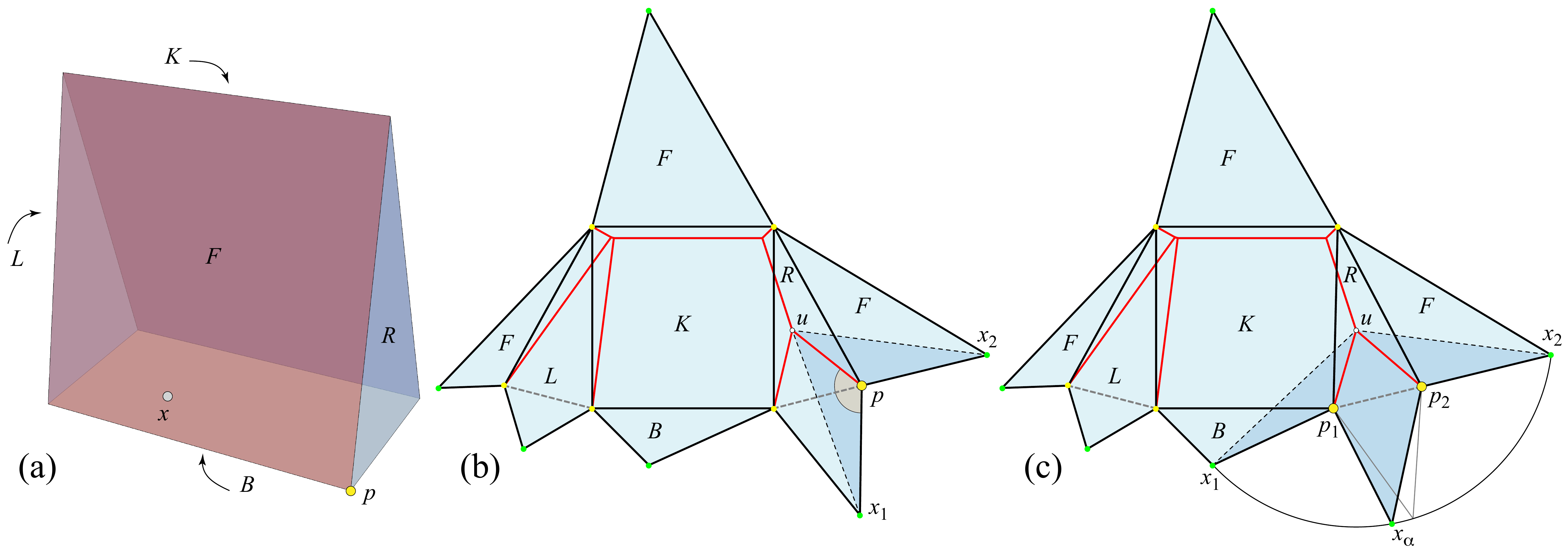}
\caption{
(a)~A $6$-vertex polyhedron $P$. The $F$ and $K$ faces are unit squares; 
$B$ is a $1 \times \frac{1}{2}$ rectangle, with $x \in B$.
(b)~Star-unfolding $\bar{P}$ of $P$. $\angle  {x}_1 {p} {u}$ is marked.
(c)~Moving $x_\alpha$ on the circle arc moves the bisectors
incident to $u$, and so moves $p_1$ and $p_2$.
Refolding results in a polyhedron incongruent to~(a).}
\label{figTent_star_Cx}
\end{figure}
%==================Figure================================

\begin{rmk}
\label{no_2}
It is perhaps surprising that the above result cannot be extended to 
claim that isometries excluding neighborhoods of two vertices
always imply congruence.
%the respective existence of  two vertices  $p_1, p_2\in P$ and $q_1, q_2 \in Q$.
\end{rmk}

\begin{proof}
If the points $p_1, p_2\in P$ and $q_1, q_2 \in Q$ do not have a common neighbor in $C(x)$ and $C(\iota(x))$ respectively, the above proof establishes rigidity.

Next we focus on $P$, and try to find positions for $p_1, p_2\in P$ determined by the hypotheses.
Assume, in the following, that  $p_1, p_2\in P$  have a common degree-$3$ ramification neighbor $u$ in $C(x)$.

Star unfold $P$ with respect to some $x \in P$,  to $\bar{P}$. 
The region of $\bar{P}$ exterior to the wedge $(x_1, u, x_2)$ is uniquely determined and identical in $\bar{Q}$.
See Fig.~\ref{figTent_star_Cx}(c).

Take a point $x_\alpha$ on the circle of center ${u}$ and radius $|x_1 u|=|x_2 u|$.
We now argue that positions of $x_\alpha$ on this circle allow $p_1$ and $p_2$ to vary while maintaining all outside of the $(x_1, u, x_2)$ wedge fixed.

Let $\angle {x}_\alpha{u}{x}_1 = 2 \alpha$.
On the bisector of that angle incident to $u$, one can uniquely determine a point ${p}_1$ such that $\angle  {x}_1 {{p_1}} {u} = \pi - \frac{1}{2}\o(p_1)$.
Similarly, one can uniquely determine a point ${p}_2$ on the bisector of that angle  $\angle {x}_\alpha{u}{x}_2$, such that $\angle  {x}_2 {{p_2}} {u} = \pi - \frac{1}{2}\o(p_2)$.

Thus we have identified a continuous $1$-parameter family of star unfoldings, and consequently of convex polyhedra, verifying the hypotheses.
\end{proof}

%%%%%%%%%%%%%%%%%%%%%%%%%%%%%%%%%%%%%%%%%%%%%%%%%%%%%%%%%%%%%%%%%

\section{Tailoring and sculpting}
\label{tas}

Having established in Theorem~\ref{pyra-dome} that g-domes can be partitioned into pyramids,
the goal of this section is to prove that removal of a pyramid, i.e., a vertex truncation, can be achieved by (digon-)tailoring.
We reach this in Lemma~\ref{VertexTruncation}: 
a degree-$k$ pyramid can be removed by $k-1$ tailoring steps, each step excising one vertex by removal of and then sealing a digon.
We start with Lemma~\ref{vertex_small} which claims the result but only under the assumption that the slice plane is close to the removed vertex.
Although this lemma is eventually superseded, it establishes the notation and the main idea.
Following that, Lemma~\ref{pyramid} removes the ``sufficiently small" assumption of Lemma~\ref{vertex_small}, but in the special case of $P$ a pyramid. 
Finally we reach the main claim in Lemma~\ref{VertexTruncation},
which shows this special case encompasses the general case.

In the following, we use $\partial S$ to indicate the $1$-dimensional boundary of a $2$-dimensional surface patch $S$.

%%%%%%%%%%%%%%%%%%%%%%%%%%%%%%%%%%

\subsection{Small volume slices}

\begin{lm}
\label{vertex_small}
Let $P$ be a convex polyhedron, and $Q$ the result obtained by slicing $P$ with a plane $\Pi$ at sufficiently small distance to a vertex $v$ of $P$, and removing precisely that vertex.
Then $Q$ can be obtained from $P$ by $k-1$ tailoring steps.
\end{lm}

\begin{proof}
Let the vertex $v$ to be removed have degree $k$ in the $1$-skeleton of $P$.
Let $e_i$, $i=1,\ldots,k$, be the edges incident to $v$, and $x_i$ the intersection of the slicing plane $\Pi$ with those edges:
$\{x_i\} = \Pi \cap e_i$.

We will illustrate the argument with the right triangular prism shown in Fig.~\ref{figSculpting}, 
where $k=3$ and $\Pi = x_1 x_2 x_3$.
Note that we do not exclude the case when some (or all) of the $x_i$ are vertices of $P$.

%See Fig.~\ref{figSculpting}
%==================Figure================================
\begin{figure}
\centering
 \includegraphics[width=0.5\textwidth]{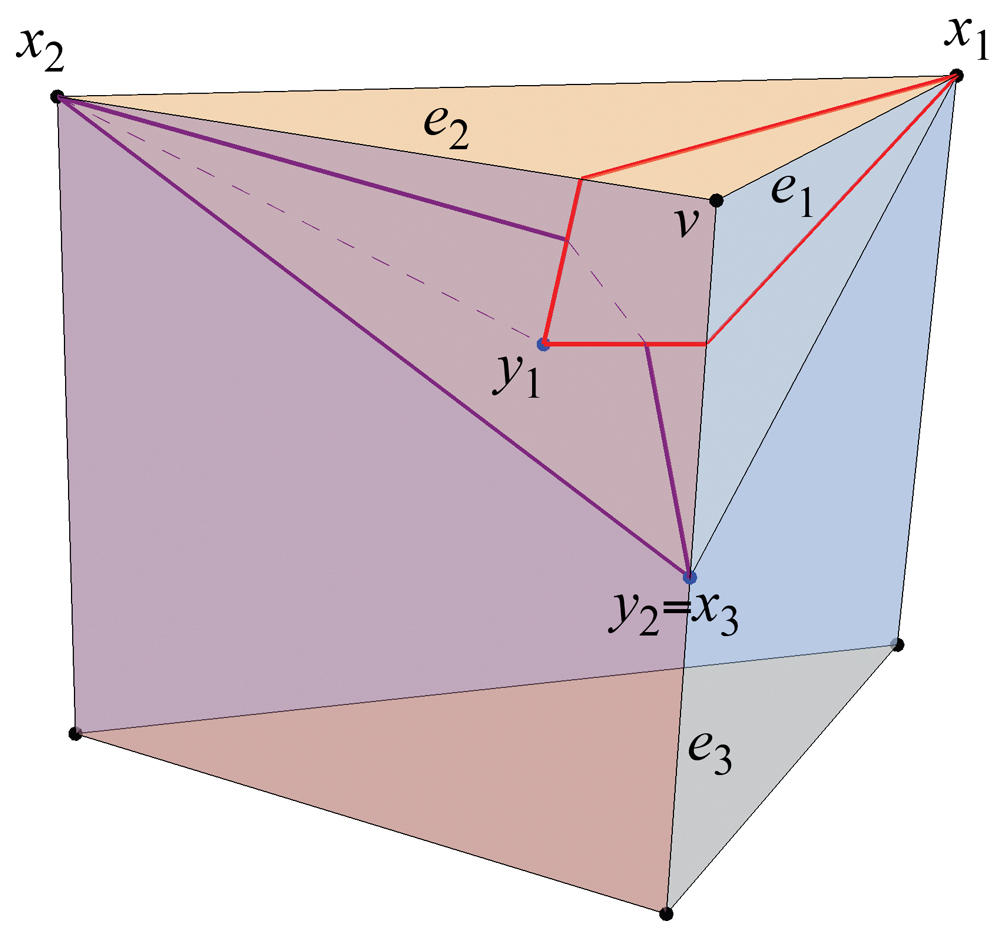}
\caption{$P$ is a prism whose top face is an isosceles right triangle.
The truncated vertex $v$ has degree $k{=}3$.
$x_3$ is the midpoint of $e_3$.
Digon $D_1 \supset \{v\}$ is shown red; $D_2 \supset \{y_1\}$ is purple.
$D_1$ is sutured closed before $D_2$ is excised.
The last replacement vertex $y_{k-1}$ must be identical to $x_k$.
}
\label{figSculpting}
\end{figure}
%==================Figure================================

Denote by $\o_P(x_i)$ $i=1,\ldots,k$ the curvatures of $P$ at $x_i$, and by $\o_Q(x_i)$ the corresponding curvatures of $Q$.
The curvature $\o_P(v)$ will be distributed to the $x_i$.

In the figure, $\o_P(v)=90^\circ$, the curvatures of the three $x_i$ are $135^\circ,135^\circ,0^\circ$ in $P$,
and approximately $156^\circ,156^\circ,48^\circ$ in $Q$.
Indeed the increases sum to $90^\circ$: $21^\circ+21^\circ+48^\circ$.

The goal now is to excise $k-1$ digons with one end at $x_1,x_2,\ldots,x_{k-1}$,
removing precisely the surface angle needed to increase $\o_P(x_i)$ to $\o_Q(x_i)$.
After digon removals at $x_1,\ldots,x_i$, we call the resulting polyhedron $P_i$.

Let a digon with endpoints $x_i$ and $y_i$ be denoted $D_i=(x_i,y_i)$.
Cut out from $P$ the digon $D_1=(x_1,y_1)$ containing only the vertex $v$ in its interior, of angle at $x_1$ equal to $\o_Q(x_1) - \o_P(x_1)$.
By the assumption that the slice plane $\Pi$ is sufficiently close to $v$, the curvature difference is small enough so that $D_1$ includes only $v$.
Again by the sufficiently-close assumption, we may assume the digon endpoint $y_1$ lies on the edge of $C(x_1)$ incident to $v$, prior to the first ramification point of $C(x_1)$.
After suturing closed the digon geodesics, $y_1$ becomes a vertex of curvature $\o_P(v)-( \o_Q(x_1) - \o_P(x_1) )$.
In the figure, $y_1$ has curvature $90^\circ - 21^\circ \approx 69^\circ$. In a sense, $y_1$ ``replaces" $v$.

Next cut out a digon $D_2=(x_2,y_2)$ containing only the vertex $y_1$ in its interior, of angle at $x_2$ equal to $\o_Q(x_2) - \o_P(x_2)$.
The newly created vertex $y_2$  ``replaces" $y_1$.
Continue cutting out digons $D_i=(x_i,y_i)$ up to $i=k-1$, each $D_i$ surrounding $y_{i-1}$, and replacing $y_{i-1}$ with $y_i$.

Because these tailorings have sharpened the curvatures  $\o_P(x_i)$ to match the after-slice curvatures $\o_Q(x_i)$,
it must be that the curvature at the last replacement vertex $y_{i-1}$
is the same as the curvature at $x_k$: $\o_P(y_{k-1}) = \o_Q(x_k)$ (to satisfy Gauss-Bonnet).
So now the tailored $P_{k-1}$ matches $Q$ in both the positions of the vertices $x_i$, $i=1,\ldots,k-1$, and their curvatures; 
the only possible difference is the location of $y_{i-1}$ compared to $x_k$.
But the rigidity result, Theorem~\ref{rigid}, implies that $y_{k-1}=x_k$, and $P_{k-1}$ and $Q$ are now congruent.
\end{proof}

The ``sufficiently-small" assumption in the preceding proof allowed us to assume that the digon $D_i=(x_1,y_1)$ endpoint $y_1$ lay 
on the segment of $C(x_i)$ incident to $v$ prior to the first ramification point $a_1$ of $C(x_1)$. 
Recall that $\o(x_1)+\o(y_i)=\o(v)$, and the further along the segment $va$ that $y_1$ lies, the larger the digon angle at $x_1$.
The procedure would be problematic if the digon angle at $x_1$ were not large enough even with $y_1$ at that ramification point $a_1$.
The next lemma removes the sufficiently-small assumption in the special case when $P$ is itself a pyramid, and the vertex truncation reduces $P$ its base, doubly covered.
Following this, we will show that the case when $P$ is a pyramid is the ``worst case," and so the general case follows.

%%%%%%%%%%%%%%%%%%%%%%%%%%%%%%%%%%%

\subsection{Pyramid case}

\begin{lm}
\label{pyramid}
Let $P$ be a pyramid over base $X$.
Then one can tailor $P$ to reduce it to $X$ doubly covered, using $k-1$ digon removal steps.
\end{lm}

\begin{proof}
We continue to use the notation in the previous lemma, and introduce further notation needed here.
Let $L= P \setminus X$ be the lateral sides of the pyramid $P$; so $P= L \cup X$.
After each digon $D_i=(x_i,y_i)$ is removed and sutured closed, the convex polyhedron guaranteed by
Alexandrov's Gluing Theorem will be denoted by $P_i$.
We continue to view $P_i$ as $P_i= L_i \cup X_i$, even though already $P_1$, is in general no longer a pyramid. 
We will see that all the digon excisions occur on $L_i$, while $X_i$ remains isometric to the original base $X$, but no longer (in general) planar.

We will use $C(x_i, P_j)$ to mean the cut locus of $x_i$ on $P_j$. 
%We will need
%in this proof to sometimes consider $j \neq i$.
Regardless of which $P_j$ is under consideration, we will denote by $a_i$
the first ramification point of $C(x_i)$ immediately beyond the
vertex $y_{i-1}$ surrounded by the digon $D_i=(x_i,y_i)$.

We need to establish two claims:
\begin{description}
\item [Claim~(1):] The cut locus $C(x_{i+1},P_i)$ is wholly contained in $L_i$.
\item [Claim~(2):] The digon angle $\a_{i+1}$ at $x_{i+1}$ to $a_{i+1}$ is large
enough to reduce the $L$-angle at $x_{i+1}$ to its $X$-angle on the base.
\end{description}

Before addressing the general case of these claims, we illustrate the situation
for $x_1$,  referencing Fig.~\ref{figPent_unf}.
%==================Figure================================
\begin{figure}
\centering
 \includegraphics[width=\textwidth]{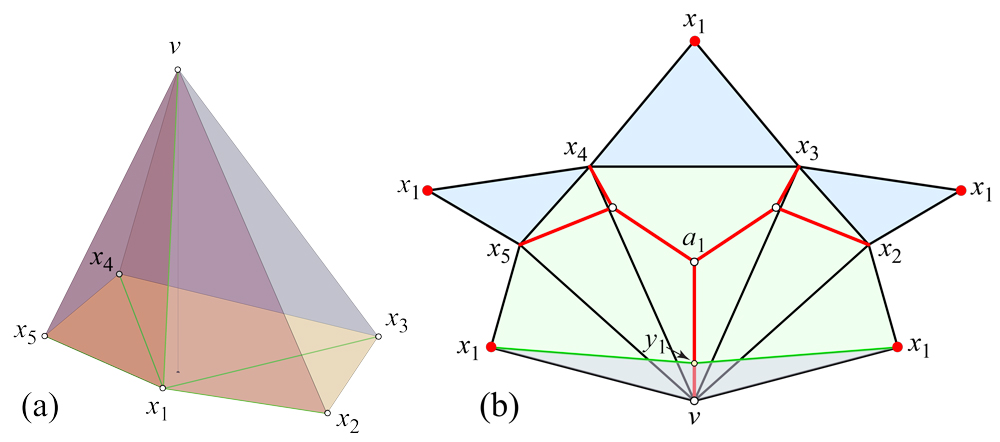}
\caption{(a)~A pyramid with pentagonal base $X$. Shortest paths from $x_1$ to all vertices
are marked green.
(b)~The star-unfolding with respect to $x_1$. 
The triangles from $X$ are blue; those from $L$ are green.
$C(x_1) \subset L$ is red.
The digon $D_1=(x_1,y_1)$ is shaded.
}
\label{figPent_unf}
\end{figure}
%==================Figure================================
The digon $D_1(x_1,y_1)$ surrounding $v$ places $y_1$ on the $v a_1$ segment
of $C(x_1,P)=C(x_1)$. If one imagines $y_1$ sliding along $v a_1$ from $v$ to $a_1$, the digon angle at
$x_1$, call it $\d_1$, increases.
To show that $y_1$ can be placed so that $\d_1$ is large enough to reduce the angle
at $x_1$ to its angle in $X$ will require $a_1$ to lie in $L$ (rather than in $X$).

It turns out that $C(x_1) \subset L$ follows from a lemma in~\cite{aaos-supa-97}.\footnote{
Lem.~3.3: the cut locus is contained in the ``kernel" of the star-unfolding which in our case
is a subset of $L$.}
However, after removing $D_1$ and invoking Alexandrov's Gluing Theorem,
we can no longer apply this lemma.
With this background, we now proceed to the general case.

\paragraph{Claim~(1): $C(x_{i+1},P_i) \subset L$.}
Assume we have removed digons at $x_1,\ldots,x_i$, so that
$P_i = X_i \cup L_i$, and $L_i$ contains one vertex $y_i$,
the endpoint of the last digon $D_i$ removed, and $X_i$ contains no vertices.
Assume to the contrary of Claim~(1) that $C(x_{i+1},P_i)=C(x_{i+1})$ includes a point $z$
strictly interior to $X_i$. 
%We'll abbreviate $C(x_{i+1},P_i)$ by $C(x_{i+1})$.
Because $z \in C(x_{i+1})$, there are two geodesic segments from $x_{i+1}$
to $z$, call them $\g^z_1$ and $\g^z_2$.
Because $X_i$ contains no vertices, it cannot be that both 
$\g^z_1$ and $\g^z_2$ are in $X_i$.
Say that $\g^z_1$ crosses $L_i$.
Let $p \in \partial X$ be the first point at which $\g^z_1$ enters $X_i$,
and let $\g_1 \subset \g^z_1$ be the portion from $x_{i+1}$ to $p$.
See Fig.~\ref{Claim1Abstract}.
%==================Figure================================
\begin{figure}
\centering
 \includegraphics[width=0.75\textwidth]{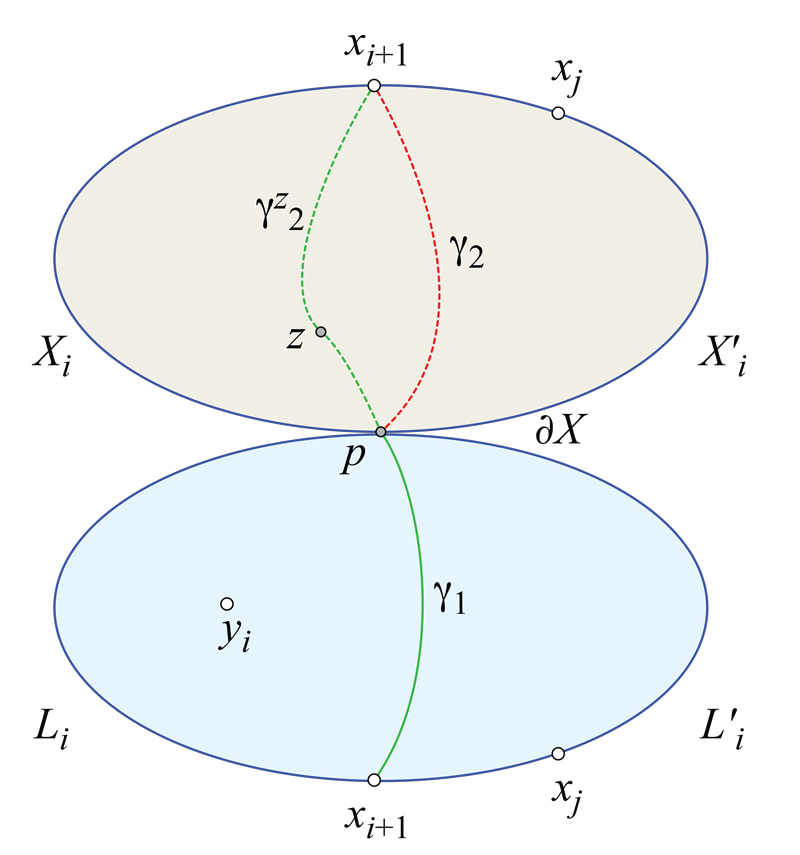}
\caption{$X_i$ is flipped above $L_i$ in this abstract illustration.
Cauchy's Arm Lemma ultimately shows that $|\g_1| \ge |\g_2|$.}
\label{Claim1Abstract}
\end{figure}
%==================Figure================================

The geodesic segment $\g_1$ divides $L_i$ into two parts; let $L'_i$ be the part that does not contain the vertex $y_i$.
Join $x_{i+1}$ to $p$ with a geodesic $\g_2$ lying in $X_i$.
$\g_2$ was a shortest path to $p$ in $X$, but may no longer be shortest in $X_i$.
$\g_2$ also divides $X_i$ into two parts; let $X'_i$ be the part sharing a portion of $\partial X$ with $L'_i$.

Now we will argue that $|\g_1| \ge |\g_2|$.
This will yield a contradiction, for the following reasons.
$\g_1$ is a shortest geodesic to $p$, because it extends to $\g^z_1$, which
is a shortest geodesic to $z$. So $\g_2$ cannot be strictly shorter than $\g_1$.
Therefore we must have $|\g_1| = |\g_2|$, which implies that $p \in C(x_{i+1})$.
But then $\g_1$ cannot continue to $\g^z_1$ beyond $p$,
as $p$ is a cut point.

To reach $|\g_1| \ge |\g_2|$, we will use an extension of Cauchy's Arm Lemma.
Let $\q^X_j$ be the angle at $x_j$ in $X_i$, and $\q^L_j$ be the angle at $x_j$ in $L_i$.
For $j > i+1$, we know that $\q^L_j > \q^X_j$ because $P$ is a pyramid (see Lemma~\ref{lemAngles}) and digon removal has not yet reached $x_j$.
We also know that $\q^X_j \le \pi$ because $X$ is convex.
However, $\q^L_j$ could be nearly as large as $2 \pi - \q^X_j$ if the pyramid $P$'s apex $v$ projects outside the base $X$.

Let $c_X$ be the planar convex chain in $\partial X$ that corresponds to $X'_i$; assume $c_X$ is
$x_{i+1}, x_{i+2}, \ldots, x_j, \ldots, p$ with angles $\q^X_j$. 
(The case where $c_X$ is includes the other part of $\partial X$, $x_{i+1}, x_i, \ldots, p$ can be treated analoguously.)
Then $|\g_2|$ is the length of the chord between $c_X$'s endpoints.
In order to apply Cauchy's lemma, we rephrase the angles at $x_j$ as turn angles $\t_j = \pi - \q^X_j$.
Cauchy's lemma then says that if the chain angles are modified so that the turn angles lie within the range $[0,\t_j]$, then the endpoints chord length
cannot decrease. Roughly, opening (straightening) the angles stretches the chord.
The extension of Cauchy's lemma reaches the same conclusion if the chain angles are modified so that the turn angles lie within $[-\t_j,\t_j]$.

Define $c_L$ as the planar (possibly nonconvex) chain composed of the same vertices $x_j$ that define $c_X$, but with angles $\q^L_j$.
Because $\q^L_j \le 2 \pi - \q^X_j$, the turn angles $\pi - \q^L_j$ in $c_L$ satisfy
$$
\pi - \q^L_j \ge \pi - (2 \pi - \q^X_j) = -(\pi - \q^X_j) = - \t_j \;.
$$
Also, because $\q^L_j > \q^X_j$, 
$$
\pi - \q^L_j \le \pi - \q^X_j = \t_j \;.
$$
So the $c_L$ turn angles are in $[-\t_j,\t_j]$, and we can conclude from Theorem~1 in~\cite{o-ecala-01}
that the $c_L$ endpoints chord length $|\g_1|$ is at least $|\g_2|$, the $c_X$ endpoints chord length.

We have now reached $|\g_1| \ge |\g_2|$, whose contradiction described earlier shows that indeed $C(x_{i+1}) \subset L_i$.

\paragraph{Claim~(2): $\a_{i+1} \ge \q^L_{i+1} - \q^X_{i+1}$.}
Recall that $\a_{i+1}$ is the angle at $x_{i+1}$ of the digon from $x_{i+1}$
to $a_{i+1}$, the first ramification point of $C(x_{i+1})$ beyond the vertex $y_i$.
The claim is that $\a_{i+1}$ is large enough to reduce $ \q^L_{i+1}$ to $\q^X_{i+1}$.
We establish this by removing a path from $C(x_{i+1})$ and tracking angles,
as follows.

From Claim~(1), $C(x_{i+1}) \subset L_i$.
Let $\r_{i,i+2}$ be the path in the tree $C(x_{i+1})$ from $x_i$ to $x_{i+2}$.
See Fig.~\ref{Star_CutLocus_Digons}.
%==================Figure================================
\begin{figure}
\centering
 \includegraphics[width=0.85\textwidth]{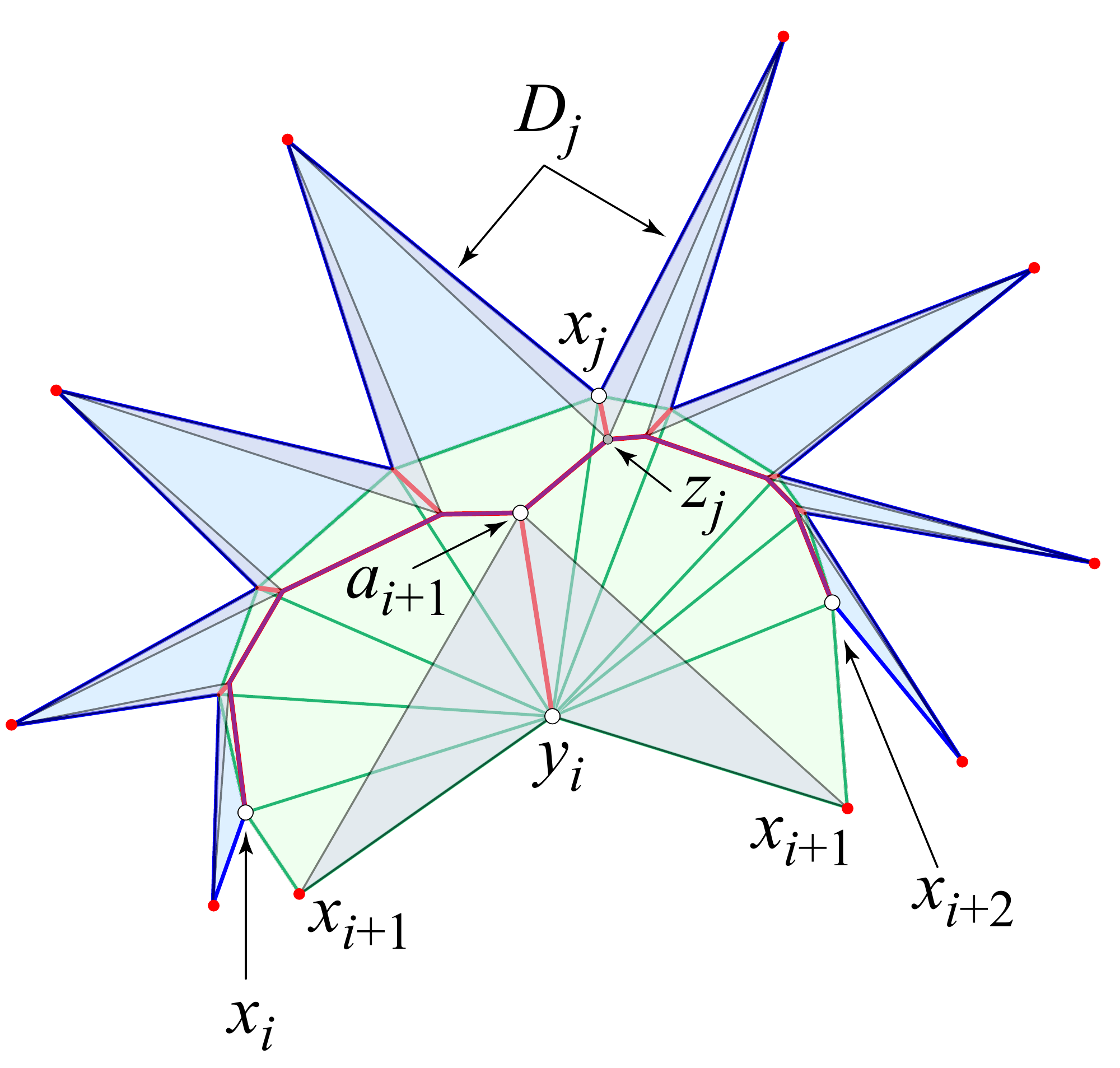}
\caption{
Star-unfolding of a pyramid with respect to $x_{i+1}$.
The triangles from $X_i$ are blue; those from $L_i$ are green.
Red points are images of $x_{i+1}$.
Digons are shaded. $\r_{i,i+2}$ is purple; remainder of $C(x_{i+1})$ is red.
}
\label{Star_CutLocus_Digons}
\end{figure}
%==================Figure================================
Removal of $\r_{i,i+2}$ from $C(x_{i+1})$
disconnects $C(x_{i+1})$ into the edge $y_i a_{i+1}$, and a series
of subtrees $T_j$. Each $T_j$ shares a point $z_j$ with $C(x_{i+1})$.
Let $D_j$ be the digon from $x_{i+1}$ to $z_j$, and let $\d_j$ be the
angle of $D_j$ at $x_{i+1}$.
Finally, let $\D_j = \sum_j \d_j$.

Note that all the $\d_j$ angles are in $X_i$.
In contrast, the angle at $x_{i+1}$ in the digon $D_{i+1} = D(x_{i+1},a_{i+1})$ is in $L_i$, as illustrated in Fig.~\ref{Star_CutLocus_Digons}.
We defer justifying this claim to later.

Cut off all $D_j$, and also cut off $D_{i+1}$.
Suture the surface closed; call it $P^*=X^* \cup L^*$.
By Lemma~\ref{Truncation}, the cut locus $C(x_{i+1},P^*)$ is precisely the path $\r_{i,i+2}$.
Therefore, by Lemma~\ref{path}, $P^*$ is a doubly covered convex polygon, so all angles at $x_j$ are equal above in $L^*$ and below in $X^*$.
In particular, $\q^{X^*}_{i+1} = \q^{L^*}_{i+1}$.
Now, because $\D$ angle was removed from $\q^{X}_{i+1}$, $\q^{X^*}_{i+1} = \q^{X}_{i+1} - \D$.
Because $\a_{i+1}$ was removed from $\q^L_{i+1}$, $\q^{L^*}_{i+1} = \q^{L}_{i+1} - \a_{i+1}$.
Therefore, 
\begin{eqnarray}
\q^{L}_{i+1} - \a_{i+1} &=& \q^{X}_{i+1} - \D \; \le \; \q^{X}_{i+1} \\
\a_{i+1} &\ge& \q^{L}_{i+1} - \q^{X}_{i+1}
\end{eqnarray}
which is Claim~(2).

It remains to show that $D_{i+1}$ is in $L_i$ rather than in $X_i$.
Suppose to the contrary that all the angle removal was in $X_i$. Then $\q^{L}_{i+1} = \q^{X}_{i+1} - \D - \a_{i+1}$.
Then $\q^{L}_{i+1} < \q^{X}_{i+1}$, which is not possible for $P$ a pyramid.
This completes the proof of Claim~(2) and the lemma.
\end{proof}

%%%%%%%%%%%%%%%%%%%%%%%%%%%%%%%

\subsection{General case}

\begin{lm}
\label{VertexTruncation}
Let $Q$ be obtained from $P$ by truncating vertex $v$.
Then, if $v$ has degree-$k$, $Q$ may be obtained from $P$ by $k-1$ tailoring steps, each the excision of a digon surrounding one vertex.
\end{lm}

\begin{proof}
Here we argue that the general case is in some sense no different than
the special case of $P$ a pyramid just established in Lemma~\ref{pyramid}.
In fact, the exact same digon excisions suffice to tailor $P$ to $Q$.

First we establish additional notation.
Let $\Pi$ be the plane slicing off $v$ above $\Pi$, and let $X = \Pi \cap P$.
Let the ``bottom" part of $P$ be $Q'$, with the final polyhedron $Q = Q' \cup X$.
We continue to use $L$ to denote the portion of $P$ above $\Pi$, so $P = Q' \cup L$.
After removal of digons at $x_1,x_2,\ldots,x_i$, we have $P_i = Q'_i \cup L_i$.

Below it will be important to distinguish between the three-dimensional extrinsic shape of $Q'_i$ and its intrinsic structure determined by the gluings that satisfy Alexandrov's Gluing Theorem.
We will use $\bar{Q}'_i$ for the embedding in $\R^3$ and $Q'_i$ for the intrinsic surface, and we will similarly distinguish between $\bar{X}_i$ and $X_i$.
Note that we can no longer assume that $C(x_{i+1}, P_i) \subset L_i$, for the cut locus could extend into $Q'$
(whereas it could not extend into $X$ in Lemma~\ref{pyramid}).

\medskip

It suffices to show by induction that, on $P_i$, the following statements hold:
\\ (a) The shortest path $\g_{i+1}$ joining $x_{i+1}$ to $y_i$ is included in $L_i$.
\\ (b) The ramification point $a_{i+1}$ is still on $L_i$.
\\ (c) The $x_{i+1}$ angle $\a_{i+1}$ of the digon $D_{i+1} = D(x_{i+1},a_{i+1})$ is larger than or equal to $\o_Q(x_{i+1}) - \o_P(x_{i+1})$ (and so sufficient to reduce the curvature to $\o_Q(x_{i+1})$).

\medskip

To see (a), assume, on the contrary, that $\g_{i+1}$ intersects $Q'_i$.
Assume, for the simplicity of the exposition, that $\g_{i+1}$ enters $Q'_i$ only once, at $x_{i+1}$, and exits $Q'_i$ at $p \in \partial X$.
Let $\g'_{i+1}$ denote the part of $\g_{i+1}$ between $x_{i+1}$ and $p$.

We now check~(a) for $i=0$.
$Q=Q' \cup X$ and $X$ is planar, hence, because the orthogonal projection of any rectifiable curve onto a plane shortens or leaves its length the same,
$\g'_1$ is longer than or has the same length as its projection $\g''_1$ onto $X$.
So $p$ is a cut point of $x_1$ along $\g_1$, contradicting the extension of $\g_1$ as a geodesic segment beyond $p$.

By the induction assumption, all the digon excisions occur on $L_i$; $Q'_i$ is unchanged.
Nevertheless, as part of $P_i$, neither $\bar{Q}'_i$ nor $\bar{X}_i$ is (in general) congruent to the original $\bar{Q}$ and $\bar{X}$.
However, if we consider $Q'_i$ and $X_i$  separate from $P_i$, we can reshape them so that $\bar{Q}'_i = \bar{Q}$
and $\bar{X}_i = \bar{X}$, precisely because they have not changed.
Then $\bar{X}$ is planar and the projection argument used for $i=0$ works for all $i$.

\medskip

Next we check (b) and (c) for $i=0$.
Consider, as in Lemma~\ref{pyramid}, the digon $D_1=(x_1,y_1)$ with $y_1 \in C(x_1)$,
with again $a_1$ the first ramification point of $C(x_1)$ beyond $v$.
The direction at $v$ of the edge $v a_1$ is only determined by the geodesic segment from $x_1$ to $v$, 
and hence is not influenced at all by $Q'$, because, by $i=0$ in~(a), that segment lies in $L$.

The ramification point $a_1$ is joined to $x_1$ by three geodesic segments, two of 
them---say $\g_1$ and $\g_2$---included in $L$.
The third geodesic $\g_3$ starts from $x_1$ towards $Q'$ and finally enters $L$ to connect to $a_1$. 
See Fig.~\ref{GeodesicsAbstract}.
Because these three geodesics have the same length, the longer $\g_3$ is, the longer are $\g_1$ and $\g_2$, and therefore more distant is $a_1$ to $v$.
So $a_1$ is closest to $v$, and the segment $v a_1$ shortest, when $Q' = X$ and $P$ is a pyramid.
It is when $v a_1$ is shortest that there is the least ``room" for $y_1$ on $v a_1$ to achieve the needed 
digon angle at $x_1$, for that angle is largest when $y_1$ approaches $a_1$.
Therefore, the case when $Q' = X$ and $P$ is a pyramid is the worst case, already settled in  Lemma~\ref{pyramid}.

%==================Figure================================
\begin{figure}
\centering
 \includegraphics[width=0.6\textwidth]{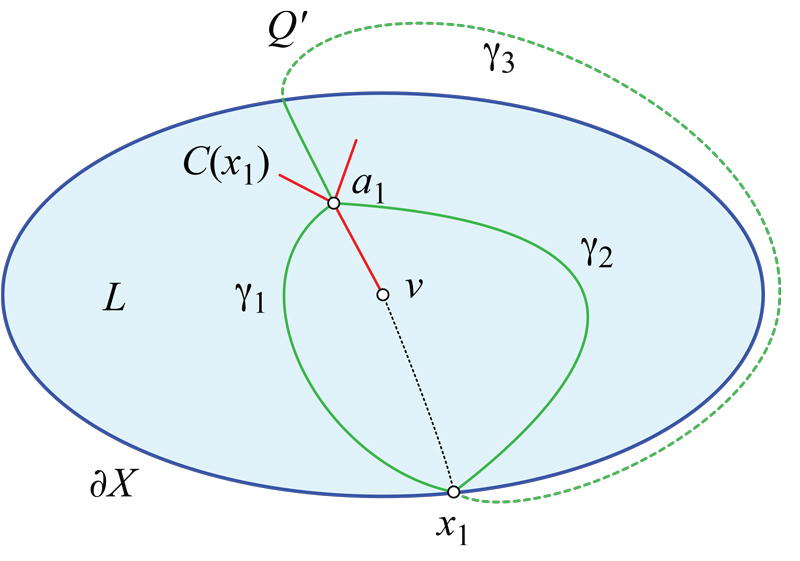}
\caption{Three geodesics to ramification point $a_1$ on $C(x_1)$. Dashed $\g_3$ partially in $Q'$.
The same situation holds under the changes:
$x_1 \to x_{i+1}$, $a_1 \to a_{i+1}$, $v \to y_i$, $Q' \to Q'_i$, $L \to L_i$.
}
\label{GeodesicsAbstract}
\end{figure}
%==================Figure================================

Now we treat the general case for (b) and (c).
Again by the induction assumption, all changes to $P_i$ were made on its ``upper part'' $L_i$.

Because we ultimately need to reduce $L$ to $X$, the angle $\o_Q(x_{i+1}) - \o_P(x_{i+1})$ necessary to be excised at $x_{i+1}$, does not depend on $Q'_i$, only on $L_i$.
Thus the argument used for $i=0$ carries through.
The situation depicted in Fig.~\ref{GeodesicsAbstract} remains the same, with $x_1$ replaced by $x_{i+1}$, $v$ replaced by $y_i$, and $a_1$ replaced by $a_{i+1}$.
The ramification point $a_{i+1}$ is closest to $y_i$, and the segment $y_i a_{i+1}$ shortest, when $Q' = X$ and $P$ is a pyramid.
It is when $y_i a_{i+1}$ is shortest that there is the least ``room" for $y_{i+1}$ on $y_i a_{i+1}$ to achieve the needed angle excision at $x_{i+1}$, 
for that angle is largest when $y_{i+1}$ approaches $a_{i+1}$.
Therefore, the case when $Q' = X$ and $P$ is a pyramid is the worst case, already settled in  Lemma~\ref{pyramid}.
\end{proof}

\noindent
Note that, in the end, the digon removals in Lemma~\ref{vertex_small}, and then in  Lemma~\ref{pyramid},
also work in the general case, Lemma~\ref{VertexTruncation}.

\bigskip

%\noindent\blue{[Somewhere earlier in this section:]
%$P = X \cup L$ and $P_i = X_i \cup L_i$.
%We use $\partial X = \partial L$ and $\partial X_i = \partial L_i$
%to indicate the geodesic polygon $x_1, x_2, \ldots, x_k$.}

%%%%%%%%%%%%%%%%%%%%%%%%%%%%%%%

\subsection{Remarks about tailoring}
\label{remarks_tailoring}

We now detail an example following Lemma~\ref{pyramid} to tailor a pyramid $P$ to its base $X$.
We continue to employ the notation used in the lemmas above.
The example is shown in Fig.~\ref{Hex3D}. $X$ is a regular hexagon, and $L$ consists of $k=6$ congruent, $70^\circ{-}70^\circ{-}40^\circ$ isosceles triangles. 
The curvature at the apex $v$ is $360^\circ - 6 \cdot 40 = 120^\circ$.
The angle at each $x_i$ in $X$ is $120^\circ$ whereas the angle in $L$ is $140^\circ$.
So each digon excision must remove $20^\circ$ from $x_i$.
As in the lemmas, we excise the digons in circular order around $\partial X$.

%See Fig.~\ref{Hex3D}
%==================Figure================================
\begin{figure}
\centering
 \includegraphics[width=0.5\textwidth]{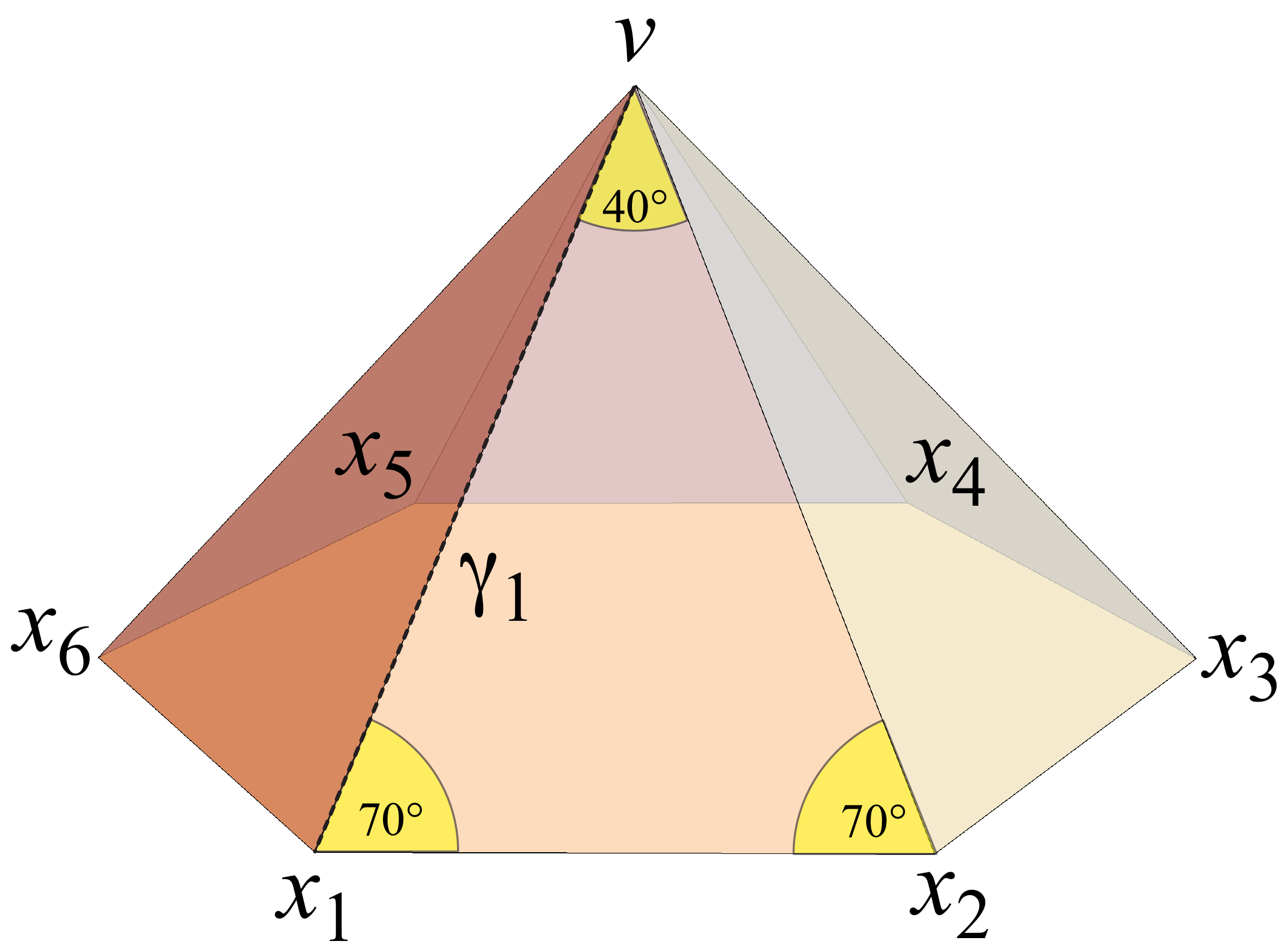}
\caption{Pyramid with regular hexagon base $X$;
all lateral faces congruent triangles.
}
\label{Hex3D}
\end{figure}
%==================Figure================================

We display the progress of the excisions on the layout of $L$ in Fig.~\ref{DigonsHexFlatX}(a).
Let $v=y_0$ for ease of notation.
$D_1=(x_1,y_1)$ includes the geodesic $\g_1$ from $x_1$ to $y_0$,
and locates $y_1$ on the cut locus segment as described in the lemmas.
The digon boundary geodesics each remove $10^\circ$ from the
left and right neighborhood of $x_1$, and meet at $y_1$ at an angle of $100^\circ$,
which is then the curvature at the new vertex: $\o(y_1)=100^\circ$.
Notice that the digon angles $20^\circ + 100^\circ$ match the curvature
$\o(v)=120^\circ$ removed, as they must to satisfy Gauss-Bonnet.
%See Fig.~\ref{DigonsHexFlatX}
%==================Figure================================
\begin{figure}
\centering
 \includegraphics[width=1.0\textwidth]{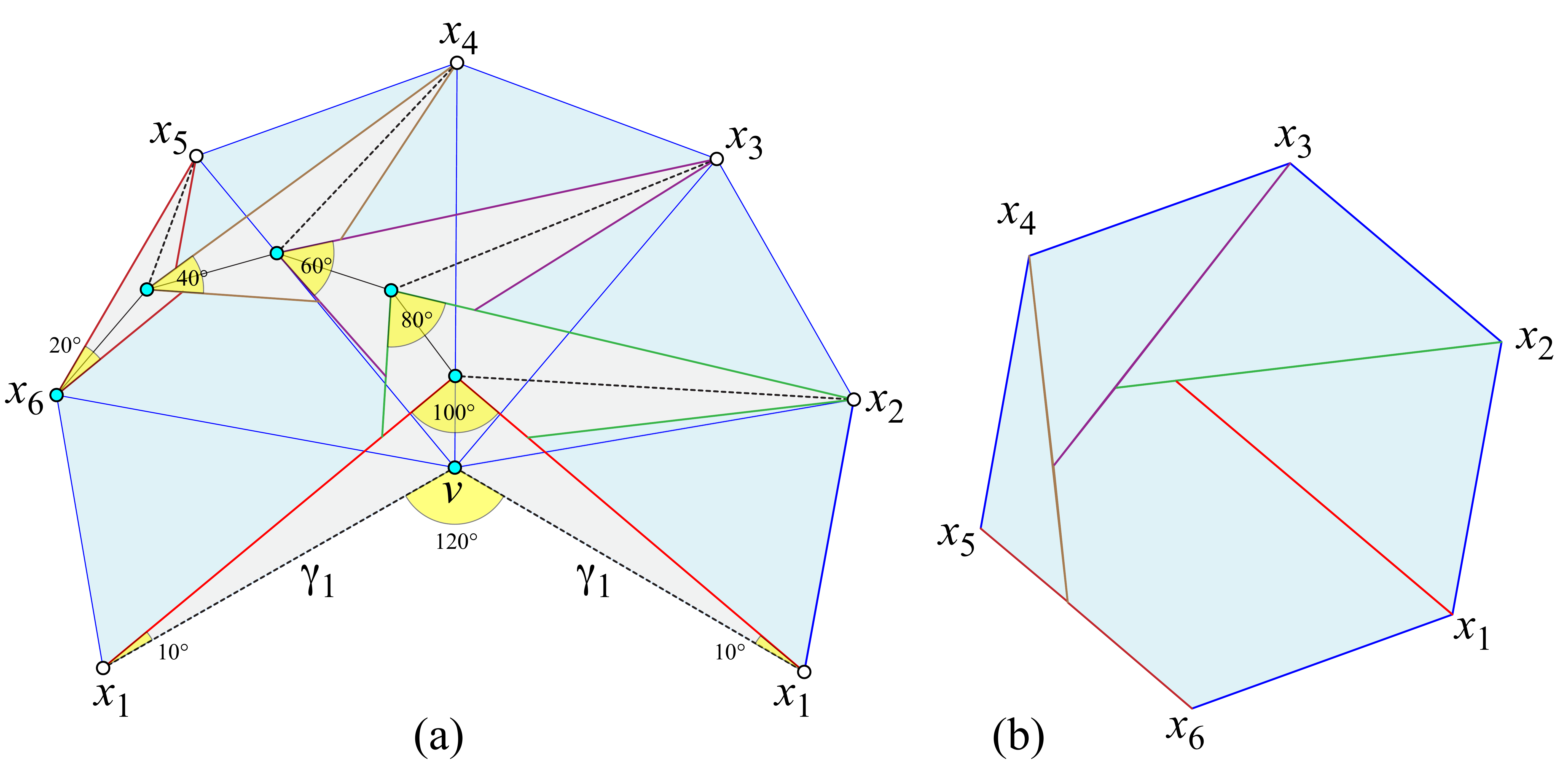}
\caption{(a)~Cone $L$ flattened; $\o(v)=120^\circ$.
Digons $D_i=(x_i,y_{i-1})$ shaded. Dashed lines are geodesics $\g_i$ from 
$x_i$ to the vertex $y_{i-1}$ (with $y_0=v$).
(b)~After excising all digons $D_1,\ldots,D_5$, $L_5$ is isometric to $X$.
Seals $ \bar{\s}_i$ are marked.
}
\label{DigonsHexFlatX}
\end{figure}
%==================Figure================================

One should imagine that $D_1$ is sutured closed in Fig.~\ref{DigonsHexFlatX}(a), producing $L_1$, before constructing $D_2(x_2,y_2)$. 
Let $\s_i$ be the geodesic on $L_i$ that results from sealing $D_i$ closed; $\s_i$ is like a ``scar" from the excision.
Notice that one of the geodesics bounding $D_2$ crosses $\s_1$.

This pattern continues as all $k{-}1=5$ digons are removed, each time replacing vertex $y_{i-1}$ with $y_i$, flattening the curvature $\o(y_i)=\o(y_{i-1})-20^\circ$.
Finally, after $D_5(x_5,y_5)$ is removed, $y_5$ is coincident with $x_6$.
No further digon removal is needed, because $D_5$ removed $20^\circ$ from $x_6$.
So now each angle in $L_k=L_5$ at all $k{=}6$ vertices is $120^\circ$, and $L_5$ is isometric to a flat regular hexagon, i.e., to $X$.

This final hexagon is shown in Fig.~\ref{DigonsHexFlatX}(b).
The images of the seals---call them $ \bar{\s}_i$---are in general clipped versions of $\s_i$ on $L_i$, clipped by subsequent digon removals.
The particular circular order of digon removal followed in this example and the lemmas results in a spiral pattern formed by $\bar{\s}_i$.
Other excision orderings, which ultimately would result in the same flat $L_{k-1}$
(effectively proved in Lemmas~\ref{vertex_small}--\ref{VertexTruncation}) would create different seal patterns.

Let $\S=\S^L_X$ denote the {\it seal graph} obtained as the union of all $\bar{\s}_i$.
It could be of some interest to characterize the patterns achievable as seal graphs, for example, is every seal graph a tree?
We only show here, with the next result, what a  degree-$4$ vertex in $\S$ looks like.
The inverse process, of reconstructing $P$ from $\S$ and $X$, will be treated more generally in 
Section~\ref{enlarging}.

\begin{lm}
If a seal graph $\S=\S^L_X$ has a degree-$4$ vertex $z$, then there exist $i<j<k$ such that $\bar{\s}_i$ and $\bar{\s}_j$ end at $z$, and $\bar{\s}_k$ passes beyond $z$.
Moreover, the digon excision order is $D_k$ just after $D_j$ just after $D_i$.
\end{lm}

\begin{proof}
Consider a common point $z$ of $\bar{\s}_i$ and $\bar{\s}_j$, with $i < j$.
We may assume $j > i+1$ , since otherwise $\deg z = 3$ in $\S$.

Assume first that no other $\bar{\s}_k$ passes through $z$.
Assume that $\deg z = 4$ in $\S$.
This implies that the digon $D_j$ crosses $\s_i$ and, since $\s_i$ remains a geodesic after the excision of $D_j$, 
$\s_i$ must be orthogonal to both geodesics bounding $D_j$.
Therefore, $\s_i$ creates with those two geodesics two geodesic triangles, both of positive curvature.
So each such triangle contains a vertex inside, contradicting that $D_j$ itself contains only one vertex.

Assume now that $\bar{\s}_i$ ends at $z$, as does $\bar{\s}_k$ for some $k \not\in \{i,j\}$.
Notice that we cannot have four $\bar{\s}$s ending at $z$, because for the last one 
arriving---say $\bar{\s}_k$---$\s_k$ would create a vertex at $z$, 
which will be excised by the digon $D_{k+1}$, breaking that degree-$4$ configuration at $z$.

So we may assume that $z$ belongs only to  $\bar{\s}_i$, $\bar{\s}_j$, and $\bar{\s}_k$.
We may further assume, without loss of generality, that $i<j<k$.

Notice that both $\bar{\s}_i$ and $\bar{\s}_j$ end at $z$ and $\bar{\s}_k$ does not, 
because otherwise $\s_k$ would create a vertex on $\bar{\s}_i$ (or on $\bar{\s}_j$) which would be excised by the digon $D_{k+1}$, 
breaking that degree-$4$ configuration at $z$.

The digon excision order follows: only $D_j$ could surround $y_i$, so its excision was just after $D_i$, and similarly for $D_j$.
\end{proof}

\medskip

For a particular digon-removal ordering, consider  the inverse image of $\S^L_X$ on $L$, and denote it by $\G=\G_X^L$. 
$\G$ is a simple geodesic polygon surrounding $v$.
In Fig.~\ref{DigonsHexFlatX}(a), $\G$ is the boundary of the gray region, 
effectively the union of the digons
(but recall that the digons $D_i$ live on different surfaces $P_{i-1}$; hence ``effectively").
Clearly, excising the surface bounded by $\G$ from $L$ all at once achieves the same effect as excising the digons $D_i$ one-by-one.
We do not know if it is possible or not to determine $\G$ directly on $L$; i.e., otherwise than via digons.

The region of $L$ bound by a geodesic polygon $\G$ is a particular instances of what we call a \emph{crest}:
a subset of $L$ enclosing $v$ whose removal and suitable suturing via AGT will reduce $L$ to $X$. 
Note that we allow the boundary of a crest to include portions of $\partial L$, e.g., $\G$ in Fig.~\ref{DigonsHexFlatX}(a) includes the $x_i$ as well as the edge $x_5 x_6$.
In Section~\ref{secCrests} we will show that it is possible to construct crests directly on $L$ without deriving them from digon removals.

%%%%%%%%%%%%%%%%%%%%%%%%%%%%%%%%%%%%%%%%%%%%%%%%%%%%%%%%%%%%%%%%%

\section{Tailoring is finer than sculpting}
\label{finer}

In this section we reach one of our main results, Theorem~\ref{main}, which says, roughly, 
that any polyhedron $Q$ that can be obtained by sculpting $P$ can be obtained by tailoring $P$.
Moreover, Lemma~\ref{NotReachable} shows that polyhedra can be obtained by tailoring that cannot be obtained by sculpting. So, in a sense, tailoring is finer than sculpting.

\begin{lm}
\label{NotReachable}
There are shapes $P$ and sequences of tailorings of $P$ that result in polyhedra not achievable by sculpting.
\end{lm}

\begin{proof}
We first tailor a regular tetrahedron $R$ as in Example~\ref{tetra-deg}, resulting in the kite $K$ in Fig.~\ref{fig2}(b).
To show that $K$ cannot fit inside $R$, assume $R$ has edge-length $1$.
Then its extrinsic diameter is $1$ and its intrinsic diameter is $2/\sqrt{3}$ (see, e.g. Theorem 3.1 in~\cite{ro1}).
Moreover, the extrinsic diameter of $K$ is precisely the intrinsic diameter of $R$, and so it cannot fit inside $R$.

We construct now a non-degenerate example, a modification of the previous one.
Consider a non-degenerate pentahedron $F$ close enough to $K=o a c_d b$ in Fig.~\ref{fig2}(b).
For example, it could have two vertices close to the vertex $a$ of $K$.
Insert into $F$ the removed digon from $R$; this is not affected by the new vertex, because it does not interfere with the geodesic segment from $c_d$ to $o$. 
We arrive at some surface $R'$ close enough to the original tetrahedron $R$.
Therefore, the intrinsic and extrinsic diameters of $R'$ and $F$ are close enough to those of $K$ and $R$, respectively,
and the above inequality between the extrinsic diameters of $R'$ and $F$ still holds, because of ``close enough."
\end{proof}

\begin{thm}
\label{general_slice}
\label{thmSliceGDomes}
\label{main}
\label{thmMainTailoring}
Let $P$ be a convex polyhedron, and $Q$ the result obtained by repeated slicing of $P$ with planes. Then $Q$ can also be obtained from $P$ by tailoring.

Consequently, for any given convex polyhedra $P$ and $Q$, one can tailor $P$ ``with sculpting'' to obtain any homothetic copy of $Q$ inside $P$.
\end{thm}

\begin{proof}
We first prove the result for slicing with one plane $\Pi$. The result for arbitrary slicing then follows immediately.

Assume $\Pi$ is horizontal, with $Q$ the portion below $\Pi$ and $P'$ the portion above.
Denote by $x_i$ the vertices of $Q$ in $\Pi$, $i=1,\ldots,m$; call the top face of $Q$ with these vertices $F$.
Let $e$ be any edge of $F$, say $e=x_1 x_2$, and let $F'$ be the face of $Q$ sharing $e$ with $F$.
Call the plane lying on $F'$ $\Pi_0$.

Now imagine rotating $\Pi_0$ about $e$ toward $P'$, noting as it passes through each vertex $v_1,v_2,\ldots$ in that order.
For perhaps several consecutive vertices, the portion of $P'$ between the previous and the current plane is a g-dome, but rotating further
takes it beyond a g-dome.
More precisely, let $\Pi_1$ through $v_{j_1}$ be the plane such that, in the sequence
$$v_1,v_2,\ldots,v_{j_1},v_{j_1+1},\ldots,v_{j_2}, v_{j_2+1}\ldots \;,$$
the portion $G_1$ of $P'$ between $\Pi_0$ and $\Pi_1$, including the vertices $v_1,v_2,\ldots,v_{j_1}$, is a g-dome, but
the portion rotating further to include one more vertex, $v_{j_1+1}$, is not a g-dome.
$\Pi_2$ through $v_{j_2}$ is defined similarly: the portion $G_2$ between $\Pi_1$ and $\Pi_2$ including $v_{j_1+1},\ldots,v_{j_2}$
is a g-dome, but including $v_{j_2+1}$ it ceases to be a g-dome.
Here for each $\Pi_i, \Pi_{i+1}$ pair we can imagine the base $X$ of the g-dome to lie in $\Pi_{i+1}$.

Fig.~\ref{figDodecaGdome} illustrates the process.
Here $\Pi_1$ is through $v_1=v_{j_1}$, and the last g-dome lies between $\Pi_3$ and $\Pi$.

So we have now partitioned $P'$ into g-domes $G_1, G_2, \ldots$.
Lastly, we invoke Lemma~\ref{VertexTruncation} to reduce each g-dome to its base by tailoring, in the order $G_1, G_2, \ldots$. 
This reduces  $P'$ to just the top face $F$ of $Q$.

Having established the claim of the theorem for one slice, it immediately follows that it holds for an arbitrary number of slices to form $Q$.

For the second part, shrink $Q$ by a dilation until it fits inside $P$, and then apply the first part.
 \end{proof}

%See Fig.~\ref{figDodecaGdome}
%==================Figure================================
\begin{figure}
\centering
 \includegraphics[width=0.5\textwidth]{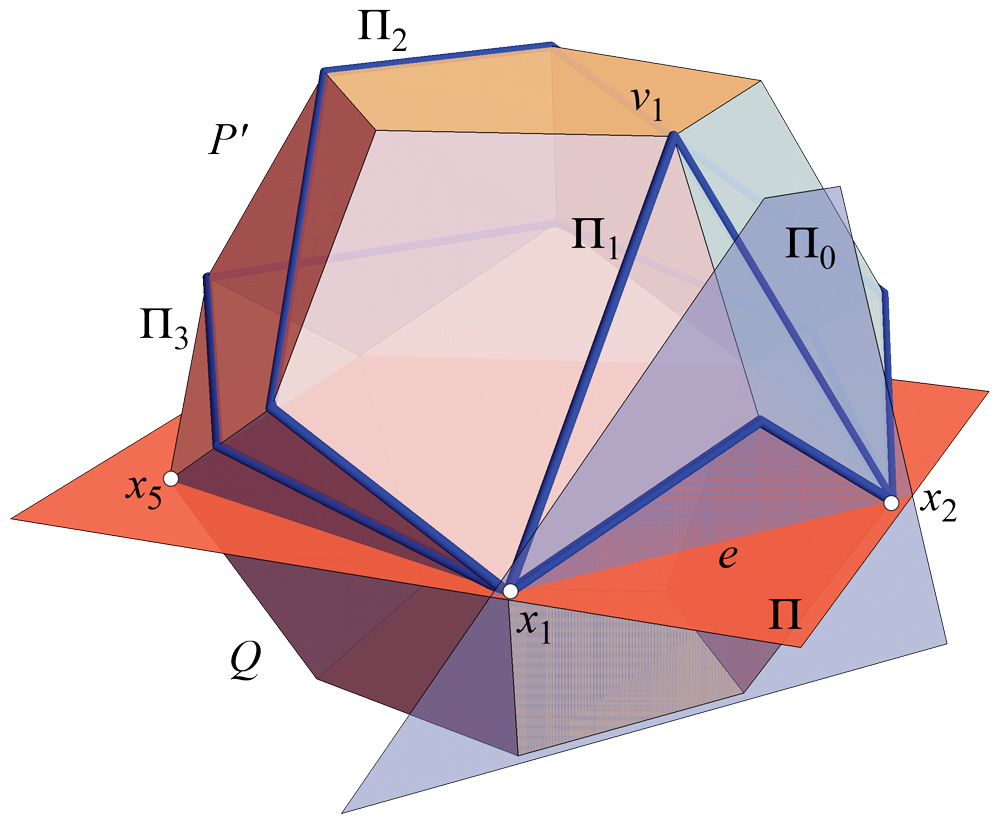}
\caption{Dodecahedron sliced by $\Pi$.
The polyhedron between each pair $\Pi_i, \Pi_{i+1}$ is a g-dome.
}
\label{figDodecaGdome}
\end{figure}
%==================Figure================================

\begin{co}
\label{approx_tailoring}
For any convex polyhedron $P$ and any convex surface $S$, one can tailor $P$ to approximate a homothetic copy of $S$.
\end{co}

\begin{proof}
Bring a homothetic copy of $S$ inside $P$.
Perform a series of slicings of $P$ with planes tangent to $S$.
Any degree of approximation desired can be achieved by
increasing the number of plane splicings. 
Call the result of these slicings $Q$.
Now apply Theorem~\ref{general_slice}.
\end{proof}

As we mentioned in the Abstract, an informal view of this corollary is that $P$ can be ``whittled" to e.g., a sphere $S$.

%%%%%%%%%%%%%%%%%%%%%%%%%%%%%%%%%%%%%%%%%%%%%%%%%%%%%%%%%%%%%%%%%

\section{A tailoring algorithm}
\label{Tailor_1}\label{secTailor_1}

In this section we follow Theorem~\ref{general_slice} to yield an algorithm for achieving the tailoring of $P$ to $Q$.
The steps will move from sculpting slices to g-domes to vertex truncations, i.e., pyramid removals.
Each pyramid removal is achieved by digon excisions and suturings.
%To repeat our description from Section~\ref{secIntroduction}, the steps are: 
%$$
%\textrm{plane slice} \;\to\; \textrm{g-domes} \;\to\; \textrm{pyramids}  \;\to\;  \textrm{digon removals} \;.
%$$
We analyze the complexity in three parts:
\begin{itemize}
\squeezelist
\item Algorithm~1: slice $\to$ g-domes.
\item Algorithm~2: g-dome $\to$ pyramids.
\item Algorithm~3: pyramid $\to$ digons.
\end{itemize}
The first two algorithms operate on the extrinsic $3$-dimensional structure of the
polyhedra.
The third algorithm instead processes its calculations on the intrinsic structure of the surface.

%Although the proofs require applying Alexandrov's Gluing Theorem after each digon removal, the algorithm never needs the extrinsic three-dimensional structure of the polyhedron.
%Instead, an intrinsic surface can be maintained to keep track of digon removals.
%
%\red{[I'm not sure I understand  ``complexity issues'' well, but... The above is indeed true only for vertex slicing (Lm.~\ref{VertexTruncation}). 
%For the  g-dome slicing (Thm.~\ref{thmSliceGDomes}), don't we need the 3D structure of intermediate polyhedra, in order to continue slicing?
%Moreover, don't we need the 3D structure of g-domes (Thm.~\ref{pyra-dome}), to partition them into pyramids?
%(This goes to the missing Sec. \ref{Partition_1}.)
%Those proofs are purely extrinsic. So apparently we also need several 3D structures of intermediate polyhedra or g-domes.]}
%
%In the end, the algorithm outputs a surface defined intrinsically, but which we know is isometric to $Q$.

%%%%%%%%%%%%%%%%%%%%%%%%%%%%%%%%%%%%%%

\subsection{Slicing algorithm}

\begin{algorithm}[htbp]
\caption{Following the sculpting of $Q$, partition sliced-off portions of $P$ into g-domes.}
%caption creates algo #
\DontPrintSemicolon

    \SetKwInOut{Input}{Input}
    \SetKwInOut{Output}{Output}

    \Input{Convex polyhedra $P$ and target $Q \subset P$}
    \Output{$O(n)$ g-domes, each possibly $O(n)$ vertices.}
       
     \BlankLine

\For{each of $O(n)$ faces $F_i$ of $Q$}{

{Slice $P$ by $\Pi_i$, the plane containing $F_i$.}

{Remove the portion $P'$ above $\Pi_i$.}

\tcp{Following Theorem~\ref{general_slice}.}

{$P'$ is partitioned into g-domes, with a total of $O(n)$ vertices.}

}%ForSlice

    %\KwResult{$O(n)$ g-domes, each possible $O(n)$ vertices. }
     
\end{algorithm}

%%%%%%%%%%%%%%%%%%%%%%%%%%%%%%%%%%%%%%%%%%%

\paragraph{Algorithm~1 Complexity Analysis.}
Let $n=\max\{ |P|, |Q| \}$ be the number of vertices of the larger of $P$ or $Q$; so $|P|,|Q| = O(n)$.\footnote{
A finer analysis would treat the number of vertices of $Q$ and $P$ independently, say, $Q$ with $m$ vertices.
The algorithm steps would be the same, but the complexities would be apportioned differently.
}
We follow Theorem~\ref{general_slice} for the slicing to g-domes.
It is obvious that there will be $\O(n)$ slices, one per face of $Q$, and any particular slice might result in g-domes with a total of $\O(n)$ vertices.
What is unclear is if each of $\O(n)$ slices can result in $\O(n)$ vertices. 
The example in the next subsection illustrates this issue.

We leave it as an open problem whether there is an example that leads to complexity $\O(n)$ per $\O(n)$ slice for every possible slice ordering. 
The below example shows this is perhaps not a straightforward question.
In the absence of a resolution, we will assume only that the $O(n)$ slices have each complexity $O(n)$.

Returning to the top-level analysis, whether a slice results in many g-domes with $O(n)$ vertices, or one g-dome of $O(n)$ vertices, does not affect the subsequent complexity calculations.

%%%%%%%%%%%%%%%%%%%%%%%%%%%%%%%%%%%%%%%%%%%

\subsection{Complexity of sculpting}

\begin{ex}
Consider Fig.~\ref{Qslice_PyPy_48}, two nested pyramids $P,Q$ sharing a common regular polygon base.
The complexity of slicing $P$ with face planes of $Q$ is order sensitive.
\end{ex}

%==================Figure================================
\begin{figure}
\centering
 \includegraphics[width=1.0\textwidth]{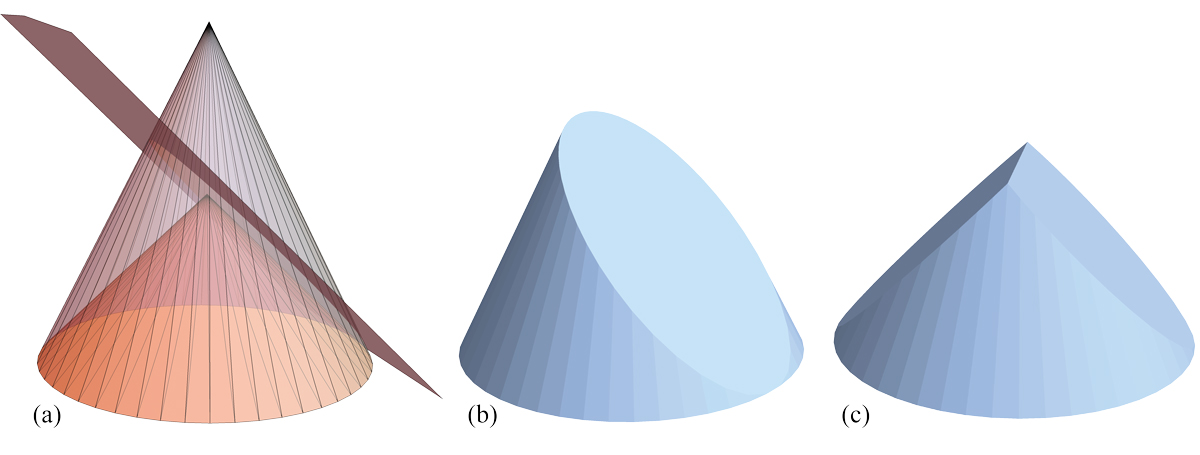}
\caption{$Q \subset P$. The shared base is a regular polygon of $n{=}48$ sides. 
(a)~One slice-plane $\Pi_0$ lying on a face of $Q$.
(b)~After slicing by $\Pi_0$.
(c)~A second, opposite slice $\Pi_{n/2}$.
}
\label{Qslice_PyPy_48}
\end{figure}
%==================Figure================================

\begin{proof}
Let the faces of $Q$ ordered in sequence around the base be $F_0,F_1,\ldots,F_{n-1}$, each $F_i$ determining a plane $\Pi_i$.
Let $P_i$ be the polyhedron after slicing $P$ with planes $\Pi_0,\Pi_1,\ldots,\Pi_i$.
$\Pi_0$ cuts $n-2$ edges of $P$, as shown in Fig.~\ref{Qslice_PyPy_48}(a,b), and effectively removes half the edges of $P$ from later slices.
If $P$ is sliced in the order $i=1,2,\ldots,n-2$, following adjacent faces around the base, each plane cuts a diminishing number of the remaining edges.
An explicit calculation shows that plane $\Pi_i$ cuts $\lfloor \frac{n+1-i}{2} \rfloor $ edges of $P_{i-1}$.
And because
$$
\sum_{i=1}^{n-2} \left\lfloor \frac{n+1-i}{2} \right\rfloor  = \Omega( n^2 ) \;,
$$
with this plane-slice ordering, $\O(n)$ slices each have $\O(n)$ vertices, for a total quadratic complexity, $\O(n^2)$.

\medskip

However, if one instead orders the slices in a binary-search pattern, then the total complexity is $\O(n \log n)$, as we now show.
Let $n=2^m$ be a power of $2$ without loss of generality. 
The pattern is slicing with planes lying on faces $F_i$ with indices in the order
$$
s = \left( 0, \frac{n}{2}, \frac{n}{4}, \frac{3n}{4}, \frac{n}{8}, \frac{3n}{8}, \frac{5n}{8}, \frac{7n}{8}, \ldots \right) \;.
$$
We partition this sequence $s$ of indices
into subsequences, $s= (0, s_1, s_2, \ldots, s_m)$, as follows:
\begin{eqnarray*}
s_1 &=& \left( \frac{n}{2} \right)\\
s_2 &=& \left( \frac{n}{2^2}, \frac{3n}{2^2} \right)\\
s_3 &=& \left( \frac{n}{2^3}, \frac{3n}{2^3}, \frac{5n}{2^3}, \frac{7n}{2^3} \right)\\
\mbox{} &\cdots& \mbox{} \\
s_k &=& \left( \frac{j \,n}{2^k} \right), \; \textrm{where} \; j = 1,3,5,\ldots, 2^k{-}1  \; .
\end{eqnarray*}
Notice that the number of indices in sequence $s_k$, $|s_k|=2^{k-1}$.
As a check, the total number of indices in $s$ is
$$
1+ \sum_{k=1}^m |s_k| = 1+ \sum_{k=1}^m 2^{k-1} = 2^m = n \;.
$$

One can calculate that the slice at $\displaystyle{i= \frac{j \, n}{2^k}}$ only cuts $\displaystyle{\frac{n}{2^k}}$ edges of $P_{i-1}$. 
Only $\displaystyle{\frac{n}{2^{k-1}}}$ edges are ``exposed" to $\Pi_i$: for example, for $k=2$ and $i=1/4$, $n/2$ edges
are possibly available for cutting by $\Pi_i$, as can be seen in Fig.~\ref{Qslice_PyPy_48}(c).
However, because of the slant of $\Pi_i$, only half of those, $n/4$, are in fact cut by $\Pi_i$.

Now we compute the total number of edges sliced by the planes following the sequence $s$.
Because $|s_k|=2^{k-1}$, and each slice in $s_k$ cuts $\displaystyle{\frac{n}{2^k}} $edges, the total number of cuts over all $k$ is
$$
\sum_{k=1}^m 2^{k-1} \frac{n}{2^k} = n\sum_{k=1}^m \frac{1}{2} = n \, \frac{m}{2} \;.
$$
And since $m=\log n$, the total complexity is $\O(n \log n)$,
or an average of $\O(\log n)$ for each of $\O(n)$ slices.
\end{proof}

%%%%%%%%%%%%%%%%%%%%%%%%%%%%%%%%%%%%%%%%%%%

\subsection{G-domes algorithm}

\begin{algorithm}[htbp]
\caption{Tailor one g-dome $G$ to its base $X$.}
%caption creates algo #
\DontPrintSemicolon

    \SetKwInOut{Input}{Input}
    \SetKwInOut{Output}{Output}

    \Input{A g-dome $G$ of $O(n)$ vertices}
    \Output{$O(n)$ pyramids, size $O(n)$; and $O(n^2)$ pyramids, size $O(1)$.}

    \BlankLine

\tcp{Following Theorem~\ref{pyra-dome}.}

%{Let g-dome $G$ have $O(n)$ vertices.}

%{Sum of vertex degrees is $2E = O(n)$.}

\For{each of $k=O(n)$ vertex-degree reductions}{
%\For{each of $k=O(n)$ $G$ top-canopy vertex-degree reductions}{
%\tcp{Possibly repeated for $O(n)$ leaves.}

{As in Theorem~\ref{pyra-dome}, slice with plane: $O(k)$.}

{Remove pyramid $P$ of $O(k)$ vertices.}

{``Clean-up" by removing $k$ pyramids each of size $O(1)$.}

}%forvi

    %\KwResult{The polyhedron $Q$ formed from $P$'s surface. }
     
\end{algorithm}

\paragraph{Algorithm~2 Complexity Analysis.}

We follow Theorem~\ref{pyra-dome} for partitioning each g-dome $G$ into pyramids.
Each vertex $v_i$ of the top-canopy of $G$ is removed,
as in Fig.~\ref{figDomeSlice_begend}, until only one remains.
Removal of each $v_i$ follows the degree-removal steps
illustrated in 
Fig.~\ref{figSliceAway_123456}.

Because the sum of the vertex degrees of a g-dome is $2E = O(n)$,
the asymptotic complexity of processing a g-dome with many vertices in its top-canopy
is no different than it is for just two vertices as in Fig.~\ref{figDomeSlice_begend}.
Moreover, we can assume that $v_2$ has degree-$3$ while
$v_1$ has degree-$k$, with $k=O(n)$.

First a plane slice results in a pyramid of $k$ vertices with apex $v_1$, which is removed
(and reduced by Algorithm~3).
Next follows a ``clean-up" phase that removes $O(k)$ pyramids
each of $4$ or $5$ vertices, so of constant size, $O(1)$.

This is then repeated for the new apex of degree $k-1$:
removal of a pyramid of $k-1$ vertices, and cleanup of $O(k-1)$ pyramids of constant size.
After iterating through $k,k-1,k-2,\ldots$, the algorithm has sliced off
$O(k^2)$ pyramids of constant size,
and $O(k)$ pyramids of size $O(k)$.
In the worst case $k=O(n)$.

\subsection{Pyramid algorithm}

\begin{algorithm}[htbp]
\caption{Tailor one pyramid $P$ to its base $X$.}
%caption creates algo #
\DontPrintSemicolon

    \SetKwInOut{Input}{Input}
    \SetKwInOut{Output}{Output}

    \Input{A pyramid $P$ of $O(n)$ vertices}
    \Output{$O(n)$ digons whose removal flattens $P$ to $X$. }

    \BlankLine

\tcp{Following Lemma~\ref{VertexTruncation}.}

\tcp{Assume apex degree-$k$, with $k=O(n)$.}

\For{each of $x_i$, $i=1,2,\ldots,k$}{

{Construct digon $D_i(x_i,y_i)$: Locate $y_i$.}

{Locate $y_i$ by tracing geodesics: $O(k)$.}

}%forxi

\end{algorithm}

\paragraph{Algorithm~3 Complexity Analysis.}
Lastly we concentrate on the cost of removing one pyramid $P$ of $O(n)$ vertices.
Following Lemma~\ref{VertexTruncation}, this requires $O(n)$ digon removals.
For each digon $D_i(x_i,y_i)$, we need to calculate the location of $y_i$
on $C(x_i)$; then $y_i$ becomes a vertex for the removal of the next
digon $D_{i+1}$.
Fortunately, there is no need to compute the cut locus $C(x_i)$.

Let us focus on locating $y_i$, after the removal of $D_{i-1}(x_{i-1},y_{i-1})$
the previous iteration.
Recall that $y_i$ is the only vertex on $L_i$,
so it is immediate to find the shortest path $\g$ from $x_i$ to $y_i$.
We know the angle $\q=\o_Q(x_i) - \o_P(x_i)$ needed to be removed
by $D_i$,
so we know that geodesics $\g_1$ and $\g_2$ at angles $\q/2$ left and right
of 
$\g = x_i y_i$ will meet at $y_i$.
Tracing $\g_1$ and $\g_2$ over the surface might
cross sealed digons $D_1,\ldots,D_{i-1}$, the ``seals" $\bar{\s}_i$ in Fig.~\ref{DigonsHexFlatX}(b).
%\blue{[JOR: This is a new thought. We'd have to represent the surface somehow, and
%such a representation will have to include the closed digons. I think
%there would have to be a transition across each such seal to determine the direction
%of the geodesic on the other side.]}
So the cost of computing $\g_1 \cap \g_2 = y_i$ is $O(n)$.

Thus the complexity of tailoring one pyramid of $O(n)$ vertices is $O(n^2)$.

Note that we do not need the extrinsic $3$-dimensional structure of the intermediate
polyhedra guaranteed by AGT to perform the calculations,
as is evident in the example described in
Fig.~\ref{DigonsHexFlatX}.

\subsection{Overall Tailoring Algorithm}
Putting the three algorithm complexities together, we have:
$$
O(n) \;\textrm{g-domes} \;\times\; O(n) \;\textrm{pyramids/g-dome} \;\times\; O(n^2) \;\textrm{per pyramid}
= O(n^4) \;.
$$

\noindent
We summarize in a theorem:

\begin{thm}
\label{SliceAlgorithm}
Given convex polyhedra $P$ and $Q$ of at most $n$ vertices each, and $Q \subset P$,
$P$ can be tailored to $Q$, following a sculpting of $Q$ from $P$ as in Theorem~\ref{general_slice}, in time $O(n^4)$.
\end{thm}

%%%%%%%%%%%%%%%%%%%%%%%%%%%%%%%%%%%%%%%%%%%%%%%%%%%%%%%%%%%%%%%%%

%\noindent\blue{[Note for Section~\ref{P-unfoldings} later: Although the algorithm complexity is $O(n^4)$,
%the number of pieces is $O(n^3)$ for digon-tailoring,
%and $O(n^2)$ for crest-tailoring. We spend $O(n)$ time identifying a digon, but when
%we are considering $P$-unfolding, we don't care about identifying, only the number of pieces.
%And crest-tailoring essentially takes $O(n)$ digons for one pyramid and makes them one piece.
%So $P$-unfolding is a bit more efficient in the number of pieces
%than we claimed in that section. Of course to find these pieces still is $O(n^4)$.]}
%
%\red{[I guess you count above the number of {\bf excised/inserted} pieces, which may well be what you claim.
%However, the seal graph in Subsec.~\ref{remarks_tailoring} may not be a tree (that's what I hoped/tried to prove at some moment),
%so it may dissect (possibly) every face of $Q$ into pieces. Up to now, we have not studied how many such pieces may result.
%A similar, yet more delicate, discussion for the crests. 
%The {\it crest graph} (to be defined) is clearly a tree on a face of $Q$; 
%but iterating the enlarging step would produce crossings of various (iterations) crest graphs, so the union of all crest graphs
%would no longer be a tree on $P$. And, again, we didn't study the resulting pattern, to estimate the number of pieces of $Q$ it produces.]}

%%%%%%%%%%%%%%%%%%%%%%%%%%%%%%%%%%%%%%%%%%%%%%%%%%%%%%%%%%%%%%%%%

\section{Tailoring without sculpting}
\label{no_sculpt}\label{secNo_sculpt}

In this section we prove a slightly weaker version of Theorem~\ref{main},
weaker in the sense that the homethet of $Q$ obtained could be arbitrarily small.
Nevertheless, this proof has its advantages, highlighted in Section~\ref{Open}~(2). 

Recall that an {\it isosceles tetrahedron} $T$ is a tetrahedron whose opposite edges are pairwise equal.
Therefore, for any such $T$, the total angle at each vertex is precisely $\pi$ and its faces are acute triangles.
Consequently, the star unfolding of $T$ with respect to any of its vertices results in an acute planar triangle.

\begin{lm}
\label{Iso_tetra}
Every convex polyhedron $Q$ has at least one pair of vertices admitting vertex merging, unless it is an isosceles tetrahedron or a doubly-covered triangle.
\end{lm}

\begin{proof}
We may assume that, for any two vertices of $Q$, their sum of curvatures is at least $2 \pi$, since otherwise vertex merging is possible
(see Subsection~\ref{secVertexMerging}).

In this case, it must be that $n \leq 4$. Indeed, since $\sum_{v \in Q} \o(v) = 4\pi$ (by the Gauss-Bonnet theorem),
if the sum of at least $5$ positive numbers is $4 \pi$ then the smallest two have sum $< 2\pi$.

If $n=3$, the cut locus of any vertex is a line-segment, by Lemma~\ref{basic}, so $Q$ is a  doubly-covered triangle (see Lemma~\ref{path}).

If $n=4$ then necessarily all vertex curvatures of $Q$ are $\pi$. 
Indeed, if the sum of $4$ positive numbers is $4 \pi$ then the smallest two have the sum $\leq 2\pi$, with equality if and only if all are $\pi$. 
So $Q$ is an isosceles tetrahedron.
\end{proof}

\begin{thm}
\label{main-}
\label{thmTailorNoSculpting}
For any given convex polyhedra $P$ and $Q$, one can tailor $P$ ``without sculpting'' until it becomes homothetic to $Q$.
\end{thm}

\begin{proof}
We proceed by induction over the number $n$ of vertices of $Q$.

\paragraph{Base of induction: $n=3$.}
In this case, $Q$ is a doubly-covered triangle.
We first show how to reduce $P$ to a doubly-covered convex polygon $Q'$.
It is then easy to tailor $Q'$ to $Q$.
We will use the cube example from Fig.~\ref{StarUnfCube} to illustrate the steps.

Let $x \in P$ be a point joined by unique geodesic segments to all vertices of $P$, 
and let $\rho$ be the unique path in $C(x)$ joining a pair of leaves of $C(x)$, i.e., joining the 
vertices $v_i$ and $v_j$ of $P$.

Then $C(x) \setminus \rho$ is a finite set of trees $T_i$.
Cut off from $P$ all $T_i$s by excising digons with one endpoint at $x$, and the other endpoint where $T_i$ joins $\rho$.
In Fig.~\ref{fig5}(a), $\rho$ connects $v_5$ and $v_7$, and separates four trees $T_i$.
After zipping each digon closed, we are left with a polyhedron $P_\textrm{flat}$ 
whose cut locus from the point corresponding to $x$ is precisely $\rho$ (by  Lemma~\ref{Truncation}).

\begin{figure}
\centering
 \includegraphics[width=0.9\textwidth]{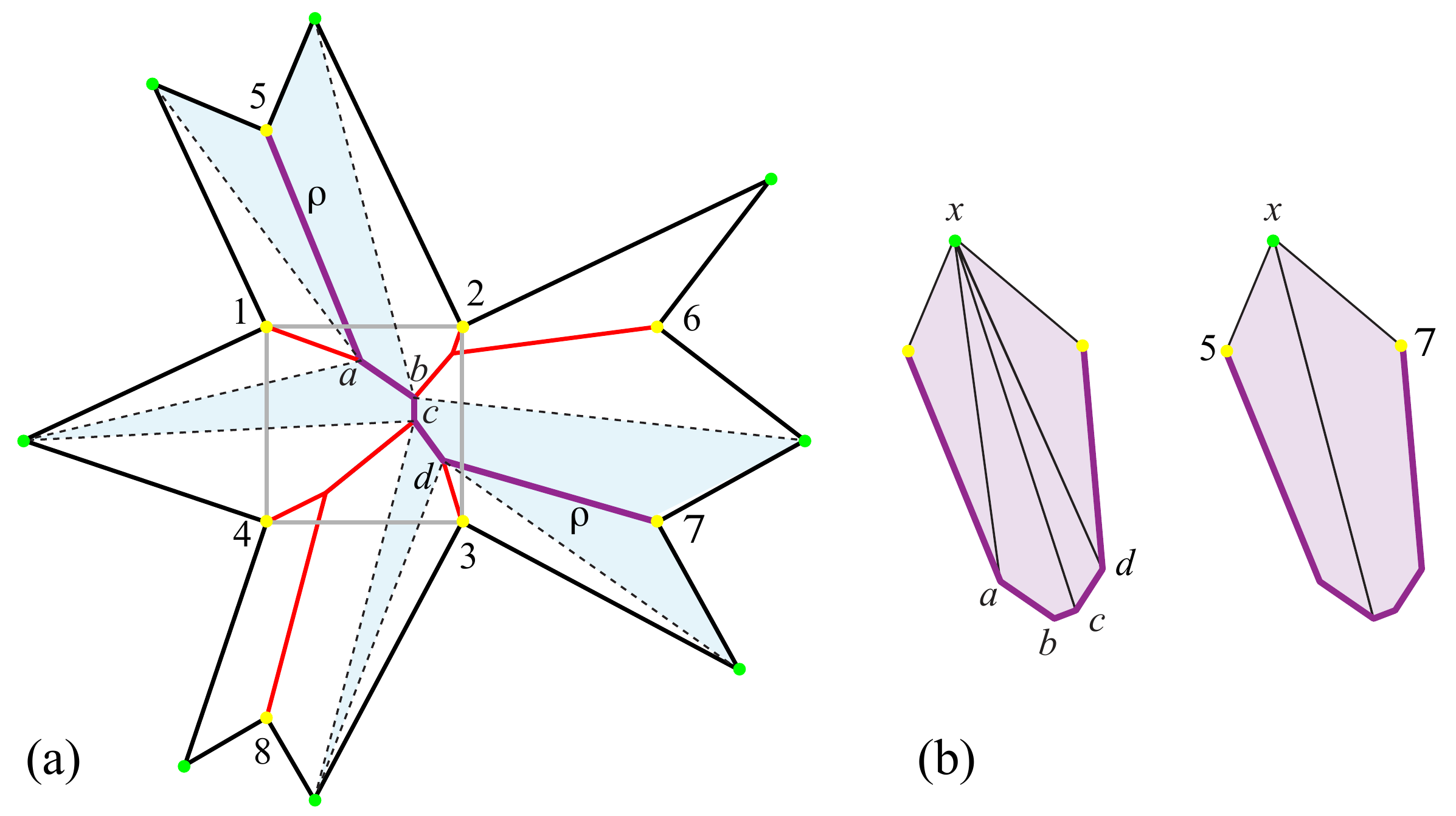}
\caption{Star unfolding of the cube in Fig.~\protect\ref{StarUnfCube}.
(a)~The path $\rho$ leaves four trees when removed from
$C(x)$. Excising four (white) digons leaves a surface (blue), which when
zipped closed folds to a doubly-covered $7$-gon, both sides of which
are shown in~(b).}
\label{fig5}
\end{figure}

Now we apply Lemma \ref{path}, which states that if $C(x)$ is a path, 
the polyhedron is a doubly-covered (flat) convex polygon, with $x$ on the rim. See Fig.~\ref{fig5}(b).
Now it is easy to tailor any doubly covered triangle $Q$ from $P_\textrm{flat}$.

\paragraph{General induction step.}
Assume we can tailor a homothetic copy of any polyhedron with $n-1$ vertices.
We want to prove the statement for $Q$ with $n$ vertices.

Assume first that $Q$ has two vertices whose sum of curvatures is strictly less than $2 \pi$, say $v_1$ and $v_2$.
Vertex merging $v_1$ and $v_2$  (by adding two back-to-back triangles $\triangle$ along the segment $\gamma$ connecting $v_1$ and $v_2$)
yields a convex polyhedron $Q'$ of $n-1$ vertices.
By the induction hypothesis, a homothetic copy of $Q'$ can be tailored from $P$.
Clearly, now $Q$ can be tailored from $Q'$ by cutting off $\triangle$s that we added to merge $v_1$ and $v_2$.

Assume now that for any two vertices of $Q$, their sum of curvatures is at least $2 \pi$.
In this case, $n \leq 4$ by Lemma~\ref{Iso_tetra}.

If $n=3$ we are done (by the induction base).

If $n=4$, $Q$ is an isosceles tetrahedron, again by Lemma~\ref{Iso_tetra}.

Recall from Example~\ref{iso_tetra} and Fig.~\ref{fig44} that such a $Q$ can be tailored to a doubly-covered rectangle $Q_\textrm{flat}$.
But it can also be tailored \emph{from} $Q_\textrm{flat}$, as that example shows. 
So now tailor $P$ to a doubly-covered polygon $P_\textrm{flat}$ as in the induction base case,
then tailor that flat polygon to a rectangle $Q_\textrm{flat}$, and finally tailor $Q_\textrm{flat}$ to $Q$, and we are done.
\end{proof}

\noindent
If the path $\rho$ in the above proof is chosen to be as long as possible,
then the flat $Q'$ has larger surface area than if $\rho$ is short.
See Open Problem~3 in Sec.~\ref{Open}.

\begin{rmk}
\label{tailor-}
The result $Q^T$ obtained by Theorem~\ref{main-} may be arbitrarily small compared to $Q$.

To see this, we return to the proof idea, rephrased for our purpose.
We excise digons from $P$, repeatedly until achieving a doubly-covered polygon $P_\textrm{flat}$.
We vertex-merge on $Q$, repeatedly until achieving a doubly-covered triangle or rectangle $Q_\textrm{flat}$.
We take a smaller copy $Q^s_\textrm{flat}$ of $Q_\textrm{flat}$ inside $P_\textrm{flat}$ and reverse the procedure, by tailoring out corresponding regions of what we added before.

\begin{enumerate}
\item The example of a regular pyramid shows that the area of $P_\textrm{flat}$ may be as small as $2/n$ of the original area of $P$.

\item The ratio between the area of $Q_\textrm{flat}$ and the area of $Q$ can be arbitrarily large.

\rm{To see this, consider an isosceles trapezoid $Z$ of base lengths $1$ and $1+ 2 \varepsilon$, and height $h$.
Its area is $(1+\varepsilon)h$.
Also consider the isosceles triangle $T$ obtained from $Z$ by extending its non-parallel sides until intersecting.
An elementary geometry argument provides the height of $T$, $(1+ 2 \varepsilon)h/2 \varepsilon$, and the area of $T$, $(1+ 2 \varepsilon)^2 h/4 \varepsilon$.

The ratio between the area of the doubles $Q_\textrm{flat}$ of $T$, and $Q$ of $Z$, is therefore 
$$
\frac{(1+ 2 \varepsilon)^2}{4  \varepsilon (1+ \varepsilon)} = 1+ \frac1{4  \varepsilon (1+ \varepsilon)}
$$ 
and can be arbitrarily large for $\varepsilon$ arbitrarily small.}
\end{enumerate}
The combination of $P_\textrm{flat}$ small and $Q_\textrm{flat}$ large leads to the arbitrarily-small claim for $Q^T$ with respect to $Q$.
\end{rmk}

\bigskip

%\red{We are exploring ways that could reduce the time complexity to $O(n^3)$ in Theorem~\ref{main-}.}
%
%\red{[I think this complexity phrase shoud go to the next section... Or we could say something without complexity, but it seems too vague.]}
%
%\red{We are exploring alternative proofs for Theorem~\ref{main-}.}

%%%%%%%%%%%%%%%%%%%%%%%%%%%%%%%%%%%%%%%%%%%%%%%%%%%%%%%%%%%%%%%%%

\section{Another tailoring algorithm}
\label{Tailor_2}\label{secTailor_2}

%%%%%%%%%%%%%%%%%%%%%%%%%%%%%%%%%%%%%%%%%%%%%%%%%%%%%%%%%%%%%%%%%

%\newpage
\begin{algorithm}[htbp]
\caption{Tailor $P$ to $Q$.}
%caption creates algo #
\DontPrintSemicolon

    \SetKwInOut{Input}{Input}
    \SetKwInOut{Output}{Output}

    \Input{Convex polyhedra $P$ and target $Q$}
    \Output{A tailored version of $P$ homothetic to $Q$}
    
    \BlankLine
    \tcp{Throughout "flat" = "doubly-covered"}
    %\BlankLine
   
    \tcp{Vertex merge on $Q$ repeatedly} 
    %\tcp{either flat triangle or isosceles tetrahedron $T_\textrm{isos}$.}
    \BlankLine

     \While{$Q \neq T_\textrm{isos}$ and $|Q| > 3$}{
    
      {Identify two vertices $v_i$ and $v_j$ such that $\o_i+\o_j < 2 \pi$.}
      
      {Vertex merge $v_i$ and $v_j$, reducing $Q$ by one vertex to $Q'$.}
      
      {$Q \leftarrow Q'$.}
        
      }%While
      \BlankLine
      \tcp{$Q$ now has $3$ or $4$ vertices.}
      \BlankLine
      
      \eIf(\tcp*[h]{$Q$ isosceles tetrahedron}){$Q = T_\textrm{isos}$}{
      {Special tailor $Q$ to flat rectangle $Q_\textrm{flat}$.}
      }%then
      {
      {$Q$ is already a flat triangle $Q_\textrm{flat}$.}
      }%else
       
     \BlankLine
     {Tailor $P$ to a flat polygon $P_\textrm{flat}$} \tcp{as in the proof base case.}
     
     {Tailor $P_\textrm{flat}$ to $Q_\textrm{flat}$.}
     
     {Scale $Q_\textrm{flat}$ to $Q^s_\textrm{flat}$ to fit on $P_\textrm{flat}$.}

     \BlankLine
    
     \tcp{Reverse the steps to reduce $Q$, each applied to $P$ starting
     with $Q^s_\textrm{flat}$.}
     
    \If(\tcp*[h]{so now a rectangle $R$}){$Q$ was an isos. tetrahedron $T_\textrm{isos}$}
    {Tailor $R$ back to $T_\textrm{isos}$.}
    
    \BlankLine
    \ForEach{vertex merging step that was applied to $Q$}
    {Reverse the step by cutting off the merge vertex.}
    
    \KwResult{A 3D polyhedron $Q^T$ homothetic to $Q$.}
     
\end{algorithm}
%\newpage

\noindent
In this section, we follow the proof of Theorem~\ref{main-} and convert it to a polynomial-time algorithm.
At a high level, $Q$ is reduced to a flat polygon $Q_\textrm{flat}$, $P$ is reduced to a flat polygon $P_\textrm{flat}$ and tailored to match $Q_\textrm{flat}$. 
Finally, the steps used to reduce $Q$ are reversed and applied to the flat remnant of $P$.

We illustrate the algorithm by tracing the steps to convert
a cube $P$ to the pentahedron $Q$ shown in Fig.~\ref{figPent}(a).
First, we reduce $Q$ as much as possible via vertex merging.
With just one vertex merge, we reach a regular tetrahedron, Fig.~\ref{figPent}(b,c), which is of course isosceles.
This then tailors to a rectangle $Q_\textrm{flat}$; see Fig.~\ref{figRegTetraRect}.

Now we reduce the cube $P$ to $P_\textrm{flat}$.
As we showed in Fig.~\ref{fig5}, this is a flat $7$-gon.
Now we tailor this to match the rectangle $Q_\textrm{flat}$.

After scaling $Q_\textrm{flat}$ to $Q^s_\textrm{flat}$ to fit on $P_\textrm{flat}$,
the remainder of the procedure reverses the steps that reduced $Q$ to $Q_\textrm{flat}$, but applied to the rectangle cut from $P_\textrm{flat}$. 
First the rectangle $R$ is restored to the regular tetrahedron by tailoring along one edge of $R$.
Then the vertex merging that produced the regular tetrahedron is reversed by cutting off vertex $d$ along $xy$.
We have now reached a homothetic copy of $Q$, considerably smaller than $Q$ but still homothetic.
And its surface is entirely composed of portions of the surface of $P$.

\begin{figure}
\centering
 \includegraphics[width=0.4\textwidth]{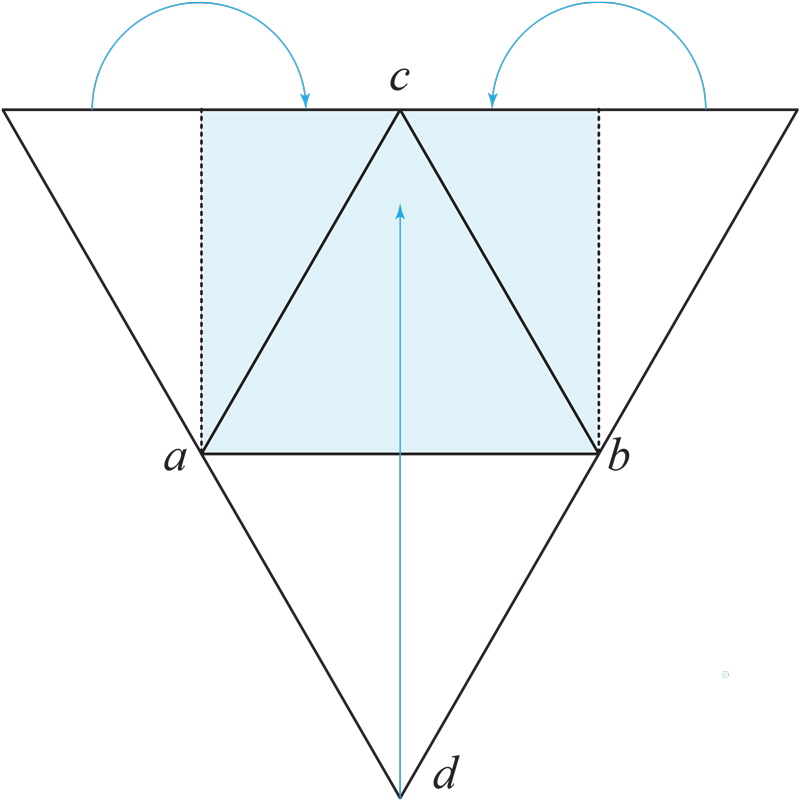}
\caption{Regular tetrahedron $abcd$ tailors to a doubly-covered rectangle $R$.
}
\label{figRegTetraRect}
\end{figure}

\medskip
We can now complete Theorem~\ref{main-}.
Let $|P|$ be the combinatorial size of the polyhedron $P$, i.e., the number of vertices.

\begin{thm}
\label{algo_main-}\label{thmAlgorithmNoSculpt}
For any given convex polyhedra $P$ and $Q$, one can tailor $P$ until it becomes homothetic to $Q$ in time $O(n^4)$, where $n=\max\{ |P|, |Q| \}$.
\end{thm}

\begin{proof}
The first step is to repeatedly apply vertex-merging to $Q$ until it is
reduced to $Q_\textrm{flat}$, when $|Q_\textrm{flat}| \in \{3,4\}$.
Identifying two vertices $v_i$ and $v_j$ such that $\o_i+\o_j < 2 \pi$
can be achieved in $O(n \log n)$ time just by sorting the curvatures
$\o_i$ and selecting the two smallest. 
From the initial sorting onward, only $O(\log n)$ would be needed to update
the list, but we'll see this efficiency is not necessary.

With $v_i$ and $v_j$ selected, the shortest path $\gamma$ between
them needs to be computed. Although there is a very complicated
optimal $O(n \log n)$ algorithm for computing shortest paths on a convex
polyhedron~\cite{SchreiberSharir}, that algorithm exploits the three-dimensional
structure of the polyhedron, which will not be available to us after the 
first vertex-merge. Recall that Alexandrov's gluing theorem guarantees that
the result of a vertex merge is a convex polyhedron, there as yet is no effective
procedure to construct the polyhedron.
However, we know the intrinsic structure of the polyhedron: its vertices, their
curvatures, a triangulation.
The algorithm of Chen and Han~\cite{ChenHan1,ChenHan2} can compute
shortest paths from this intrinsic data in $O(n^2)$ time.
Repeating this $n$ times to reach $Q_\textrm{flat}$ then
can be achieved in $O(n^3)$ time.

The next step is to tailor $P$ to $P_\textrm{flat}$ using the cut locus $C(x)$ from a ``generic point" $x$, i.e., 
one with a unique shortest path to each vertex of $P$.
As described in the proof of Theorem~\ref{main-} and Fig.~\ref{fig5}, from the star unfolding $\bar{P}_x$ of $P$ with respect to $x$, 
and $C(x) \subset \bar{P}_x$, the remaining steps to reach $P_\textrm{flat}$ can be achieved in linear time.

The star unfolding $\bar{P}_x$ can be computed in $O(n \log n)$ time using the complex Schreiber-Sharir algorithm~\cite{SchreiberSharir}, 
or in $O(n^2)$ time with the Chen-Han algorithm~\cite{ChenHan1,ChenHan2}.
$\bar{P}_x$ only needs to be computed once.
However, we know of no way to find a generic $x$ short of computing all the ``ridge-free'' regions on $P$, which takes $O(n^4)$ time~\cite{aaos-supa-97}.
See Open Problem~(3).

Reversing the $Q$ vertex-merging steps as tailorings cutting of the merged vertices on $P$, could easily be accomplished in $O(n \log n)$ time.

So the whole algorithm time-complexity is dominated by the $O(n^4)$ cost of finding a guaranteed generic $x$.

Keeping track of the considered vertices and employed digons gives in the end a correspondence between $Q^T$ and $Q$, and thus the 3D structure of $Q^T$.
\end{proof}

\noindent Because the ridge-free regions are determined by overlaying $n$ cut loci, the regions are delimited by a one-dimensional network of segments. 
Thus choosing a random point $x$ on $P$ is generic with probability $1$.
That still leaves the algorithm requiring $O(n^3)$ time.
We have little doubt this time complexity could be improved, perhaps to $O(n^2)$. 
We are currently exploring ways to reduce the complexity of the algorithm.

%\begin{rmk}
%\label{main-+}
%{\color{magenta}
%{\rm {To be included in the algorithm}}}
%The result of our algorithm is a polyhedron $Q^T$ homothetic to $Q$.
%After scaling $Q_\textrm{flat}$ to fit on $P_\textrm{flat}$, we can retain the correspondence between $Q_i$ and $P_i$
%with each $v_i$ ``vertex-unmerging"/tailoring step.
%So in the end we have the three-dimensional structure of $Q^T$, and it is homethetic to $Q$.
%\end{rmk}
%\blue{[You are right that this was not clear. I tried to distribute this into the algorithm and discussion.
%Might not have hit every button. But if so, this remark can be removed.]}

%%%%%%%%%%%%%%%%%%%%%%%%%%%%%%%%%%%%%%%%%%%%%%%%%%%%%%%%%%%%%%%%%

\section{Crests}
\label{secCrests}

In this section, we derive a method for identifying a crest that does not rely on digon removals, but rather works directly on a pyramid.
This allows us to achieve in Theorem~\ref{thmCrestTailoring} reshaping of $P$ to $Q$ by the removal of crests to flatten pyramids.
We call this process \emph{crest-tailoring}, in contrast to the \emph{digon-tailoring} explored in Sections~\ref{Tailor_1}-\ref{Tailor_2}.
We first illustrate the process of identifying a crest on a pyramid before proving that it always works.

\subsection{Examples}
As before, let $P = L \cup X$ be a pyramid with base $X$ and lateral sides $L$, with $\partial X = \partial L = (x_1, x_2, \ldots, x_k)$, for vertices $x_i$.
The apex is $v$, which projects orthogonally to $\bar{v}$.

Recall from Section~\ref{tas} that a \emph{crest} is a subset of $L$ enclosing $v$ whose removal and suitable suturing via AGT will reduce $L$ to $X$. 
We will describe the procedure for identifying a crest first for $\bar{v} \in X$ and then for $\bar{v} \not\in X$. Although the cases initially feel different,
the proofs will show that they are nearly the same.

Fig.~\ref{PyHexV_1} illustrates $\bar{v} \in X$.
Let $T_i = x_i x_{i+1} v$ and $\bar{T}_i=x_i x_{i+1} \bar{v}$.
We proved in Lemma~\ref{lemAngles} that the angle $\bar{\q}_i$ at $x_i$ in $X$ is strictly smaller than than the angle $\q_i$ on $L$, 
the sum of the angles in $T_{i-1}$ and $T_i$ incident to $x_i$ (as long as $|v \bar{v}| > 0$).
%See Fig.~\ref{PyHexV_1}
%==================Figure================================
\begin{figure}
\centering
 \includegraphics[width=1.0\textwidth]{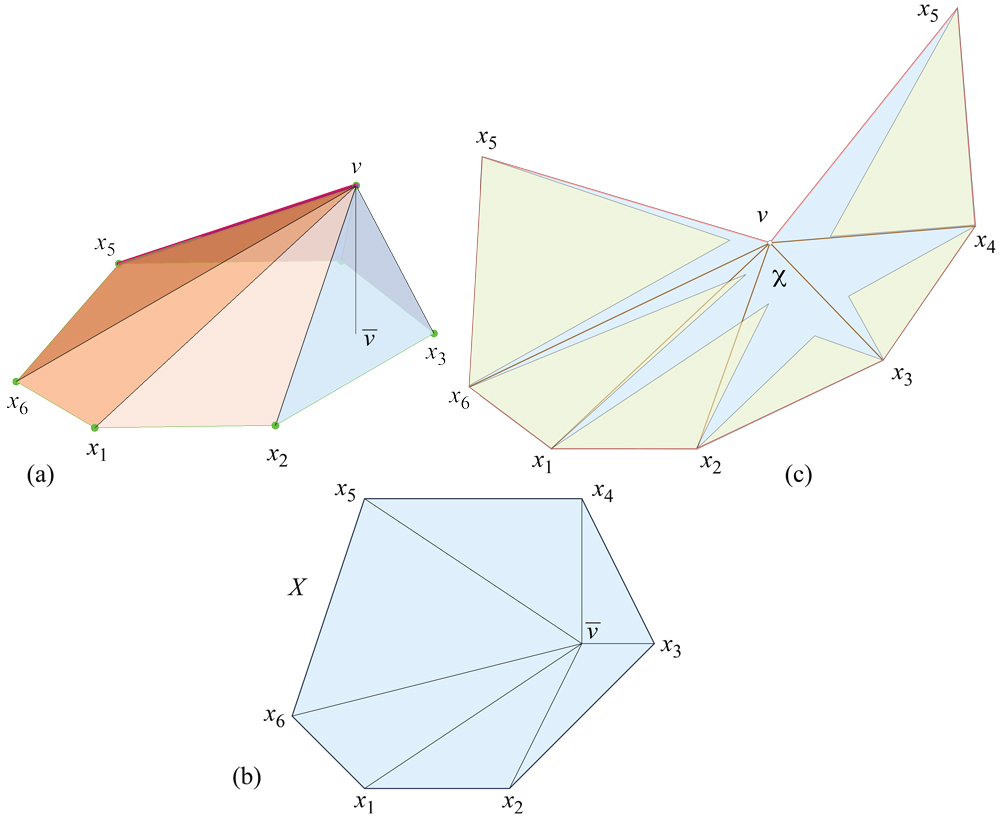}
\caption{(a)~Pyramid.
(b)~$X$ with projected triangles $\bar{T}_i=x_i x_{i+1} \bar{v}$.
(c)~Flattening of $L$ to $\bar{L}$, in this case by cutting the edge $x_5 v$.
The lifted triangles $T^L_i$ are shown yellow. The crest $\chi$ is blue.
}
\label{PyHexV_1}
\end{figure}
%==================Figure================================

A key definition is the \emph{lift} of $\bar{T}_i$ onto $L$.
Let $\bar{\a}_i$ and $\bar{\b}_i$ be the base angles of $\bar{T}_i$, at $x_i$ and $x_{i+1}$ respectively.
On $L$, extend geodesic $\g_i$ from $x_i$ at angle $\bar{\a}_i$, and extend geodesic $\g_{i+1}$ from $x_{i+1}$ at angle $\bar{\b}_i$.
Let $v_i$ be the point on $L$ at which these geodesics meet; $v_i$ is the image of $\bar{v}$.
(We will not establish that indeed these geodesics meet on $L$ until Lemma~\ref{lemLifting}.)
Then $\operatorname{lift}(\bar{T}_i) = T^L_i$ is a geodesic triangle on $L$ isometric to $\bar{T}_i$.
Another way to view the lift of $\bar{T}_i$ is to imagine $\bar{T}_i$ rotating about $x_i x_{i+1}$ by the dihedral angle there and pasting it on the inside of $L$.

Yet another way to view the lift is as follows.
$L$ is isometric to a cone and can be flattened by cutting along a generator, i.e., a segment from $v$ to $\partial L$.
Let $\bar{L}$ be a particular flattening, with the cut generator not ``near" $x_i$ just for simplicity.
Then place a copy of $\bar{T}_i$ on $\bar{L}$ matching $x_i x_{i+1}$.
Then refold $\bar{L}$ to $L$.
We will continue to reason with a flattened $\bar{L}$ but remembering
that $\bar{L}$ is a representation of $L$, and so the cut edge is not relevant.

This last layout-viewpoint yields a method to construct a full crest, call it $\chi$.
The base $X$, partitioned into $\bar{T}_i$, can be modified by opening the angle at $x_i$ from $\bar{\q}_i$ to $\q_i$. 
After opening at all $x_i$, this figure can be superimposed on the flattened $\bar{L}$, matching the boundaries $x_1, \ldots, x_k$.
This is illustrated in Fig.~\ref{PyHexV_2}(c).
The crest is then the portions of $L$ not covered by the lifted $\bar{T}_i$.
It should be clear that cutting out $\chi$ and suturing closed the matching edges will reduce $L$ to $X$,
for the $\bar{T}_i$ remaining after removing $\chi$ exactly partition $X$.

We next illustrate the case when $\bar{v} \not\in X$;
see Fig.~\ref{PyHexV_2}.
We perform the exact same process of lifting triangles $\bar{T}_i$
to $L$, but we clip those triangles to $X$---i.e.,
form the polygon $\bar{T}_i \cap X$---as indicated in~(c) of the figure.
Notice that two triangles, $\bar{T}_3$ and $\bar{T}_4$, are removed by
the clipping intersection.
Again it should be clear that cutting out $\chi$ and suturing closed 
will reduce $L$ to $X$.
%See Fig.~\ref{PyHexV_2}
%==================Figure================================
\begin{figure}
\centering
 \includegraphics[width=1.0\textwidth]{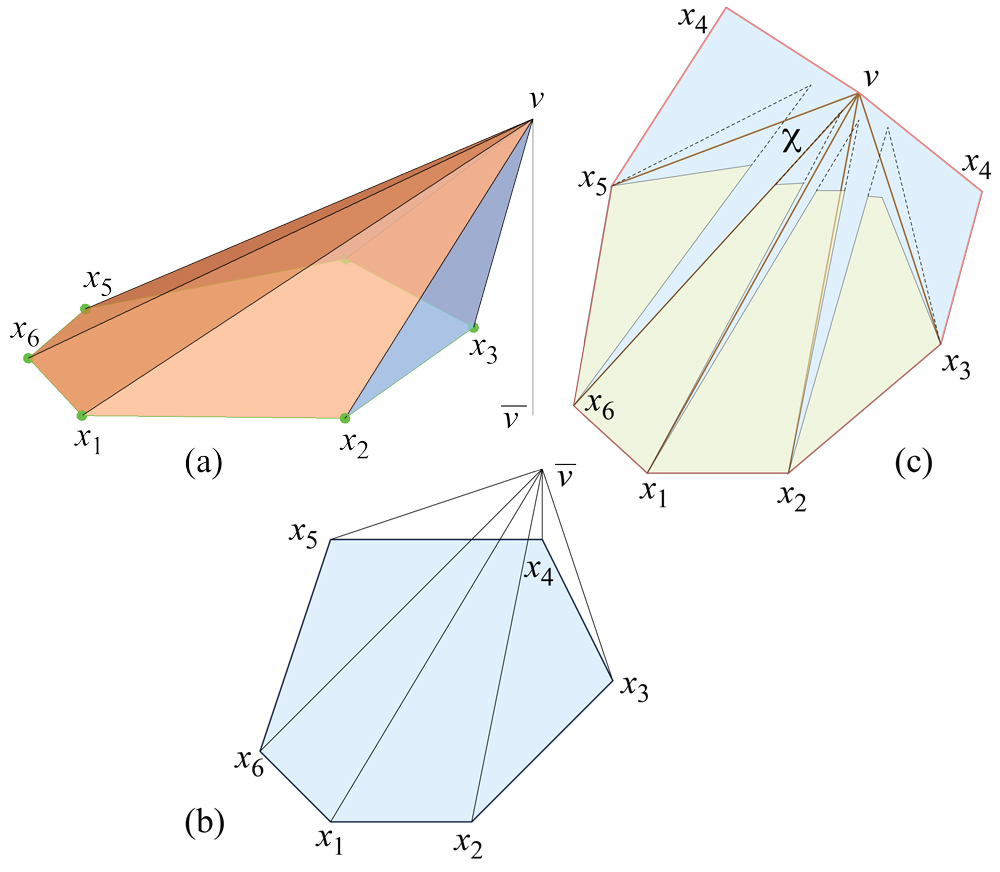}
\caption{Following the same conventions as in  Fig.~\protect\ref{PyHexV_1}:
(a)~Pyramid.
(b)~$X$; $\bar{v} \not\in X$.
(c)~$\bar{L}$, yellow clipped lifted triangles $T^L_i$, and blue crest $\chi$.}
\label{PyHexV_2}
\end{figure}
%==================Figure================================

%%%%%%%%%%%%%%%%%%%%%%%%%%%%%%%%%%%%%%%%%%%%%%%%%%%%%%%%

\subsection{Proofs}

We will need several geometric properties.

Let $\bar{\q}_i$ be the angle at $x_i$ in $X$, and  $\q_i$ the angle at $x_i$ on $L$, the sum of two triangle angles incident to $x_i$.
Then, by Lemma \ref{lemAngles}, $\q_i > \bar{\q}_i$.

\begin{lm}
\label{lemAltitude}
Let $T=a b c$ be a triangle in $\mathbb{R}^3$, with $a b$ on plane $\Pi$ and $c$ above that plane.
Let $\bar{c}$ be the orthogonal projection of $c$ onto $\Pi$, and $\bar{T}= a b \bar{c}$ the projected triangle.
Finally, let $T^r = a b c^r$ be the triangle $T$ flattened to $\Pi$ by rotating about $ab$.
Then $c^r$ and $\bar{c}$ lie on the altitude line perpendicular to the line containing $ab$.
\end{lm}

\begin{proof}
See Fig.~\ref{AltitudeLemma}.
The claim follows from the Theorem of the Three Perpendiculars.
Note that $T^r$ is congruent to $T$.
\end{proof}
%==================Figure================================
\begin{figure}
\centering
 \includegraphics[width=1.0\textwidth]{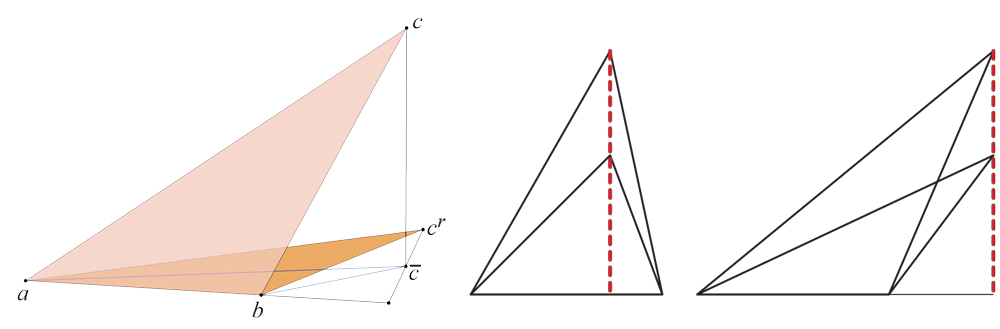}
\caption{
Rotated apex $c^r$ and projected apex $\bar{c}$ lie on same altitude.
}
\label{AltitudeLemma}
\end{figure}
%==================Figure================================
The consequence of Lemma~\ref{lemAltitude} is that, 
superimposing a lifted triangle $\bar{T}^L$ on a planar layout $\bar{L}$ of $L$,
the images of $v$ in $\bar{L}$ and $\bar{v}$ of $\bar{T}$, lie along the altitude of $T$.

%%%%%%%%%%%%%%%%%%%%%%%%%%%%%%%%%%%%%%%%%%%%%%%%%%%%%%%%%%%%%%%%%%%%%%

The following lemma assumes that $\bar{v} \in X$.
The case when $\bar{v} \not\in X$ will be treated separately.
%Let $\bar{v} \in X$.
Let $\bar{L}$ be a planar layout of $L$, say, cut open at edge $x_1 v$.
Let $\t_i= \pi-(\b_{i-1}+\a_i)=\pi - \q_i$ be the turn angle at vertex $x_i$ in the layout.
Because $\bar{v} \in X$, $\t_i > 0$, i.e., the planar image
of $\partial L$ in $\bar{L}$ is a convex chain,
and also convex wrapping around the cut edge $x_1 v$.

Let $a_i$ be the segment altitude of triangle $T_i = x_i v x_{i+1}$ in the layout $\bar{L}$.
Let $\nu = 2\pi-\o(v)$ be the surface angle of $P$ incident to $v$.

\begin{lm}
\label{lemAltTurns}
When  $\bar{v} \in X$ and consequentially $\bar{L}$ is a convex chain $x_1,\ldots,x_k$, the following hold (with $k+1 \equiv 1 \bmod{k}$):
\begin{enumerate}[label=(\alph*)]
\squeezelist
\item The sum of the turn angles, $\sum_i \t_i=\nu$, the surface angle incident to $v$.
\item The angle between $a_i$ and $a_{i+1}$ at $v$ on $L$ is exactly
equal to the turn angle $\t_i$ at vertex $x_i$.
\item The altitudes occur in order around $v$, in the sense that 
$a_{i+1}$ is counterclockwise of $a_i$ around $v$.
\end{enumerate}
\end{lm}

\begin{proof}
\begin{enumerate}[label=(\alph*)]
%\squeezelist
\item Viewing the entire layout $\bar{L}$ as a simple polygon, $\sum_i \t_i = 2 \pi$.
But we need to distinguish between $\t_1$ at the edge $x_1 v$ cut to flatten $L$, and $\t'_1$, the turns at the two images of $x_1$ in $\bar{L}$:
$$\t'_1 = 2\pi-(\a_1+\b_k) = \pi+ \t_1 \;.$$ 
The second anomalous turn in $\bar{L}$ is $\pi-\nu$ at $v$. So we have
\begin{eqnarray*}
\t'_1 + \sum_{i=2}^k \t_i + (\pi-\nu) &=& 2 \pi \;, \\
%\pi + \t_1 + \sum_{i=2}^k \t_i + (\pi-\nu) &=& 2 \pi \\
\sum_{i=1}^k \t_i &=& \nu \;.
\end{eqnarray*}

\item This can be seen by extending $T_i=x_i v x_{i+1}$ to a right triangle, with right angle at the foot of altitude $a_i$. 
See Fig.~\ref{AltitudeTurn}.

\item This follows directly from~(b). Note that here we rely on the turns $\t_i$ being positive, i.e., convex.
See Fig.~\ref{AltTurns}.
\end{enumerate}
\end{proof}
%
%See Fig.~\ref{AltitudeTurn}
%==================Figure================================
\begin{figure}
\centering
 \includegraphics[width=0.5\textwidth]{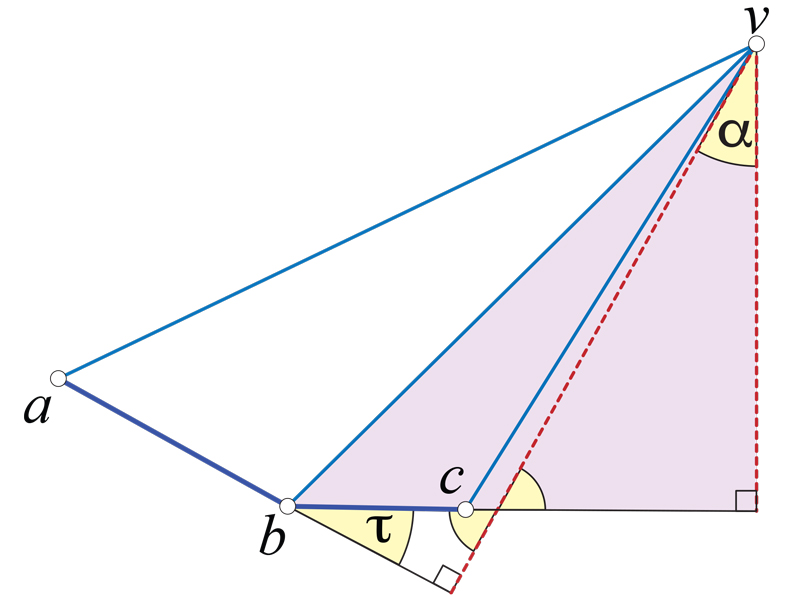}
\caption{
Triangles $avb$ and $bvc$ turn
$\t$ at $b$ of the convex chain $abc$. 
$\t$ is equal to the altitude turn $\a$;
angles of similar right triangles marked.
}
\label{AltitudeTurn}
\end{figure}
%==================Figure================================

%See Fig.~\ref{AltTurns}
%==================Figure================================
\begin{figure}
\centering
 \includegraphics[width=1.0\textwidth]{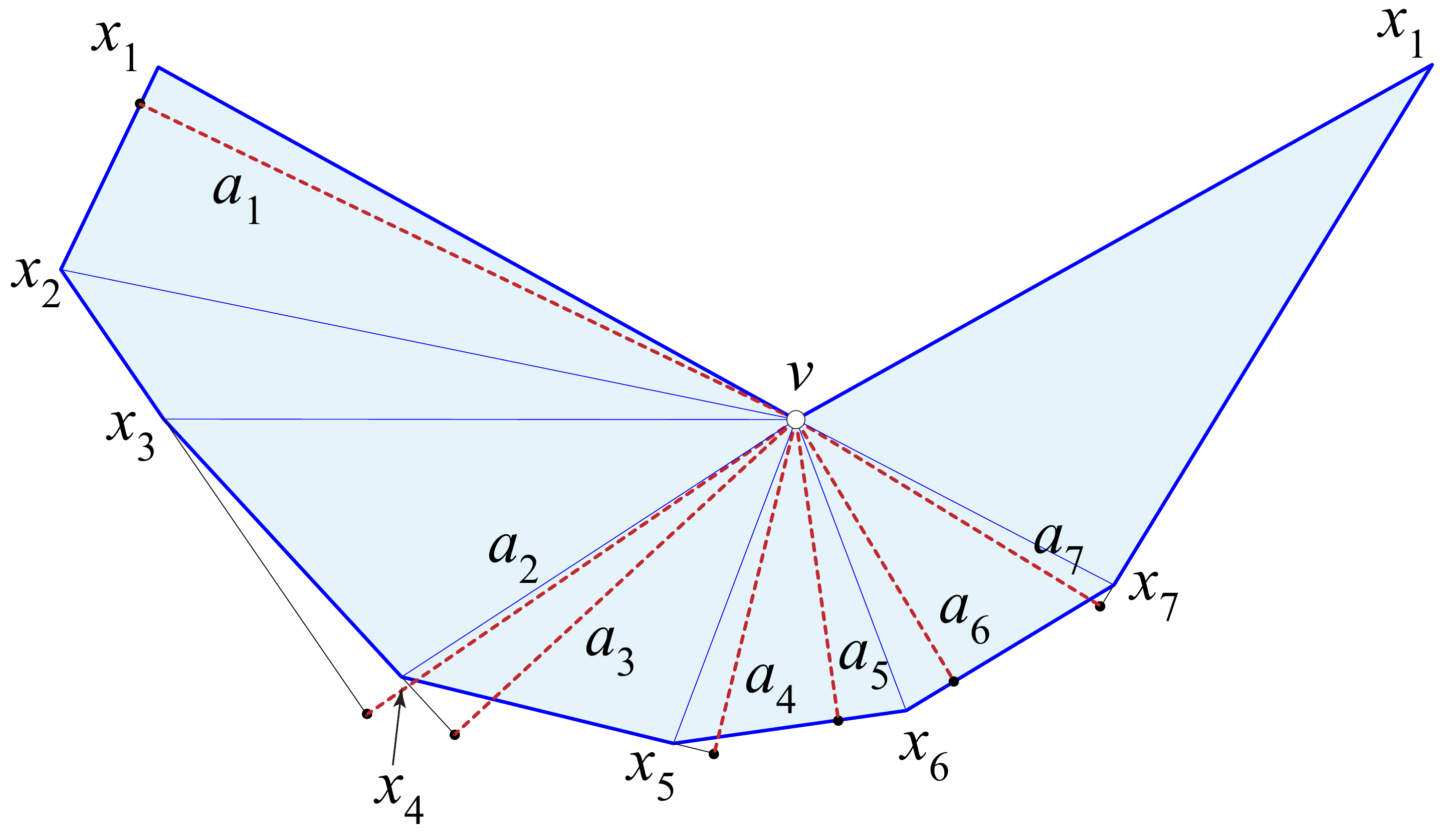}
\caption{
A convex chain $x_1,\ldots,x_7$ and the corresponding altitudes $a_1,\ldots,a_7$.
}
\label{AltTurns}
\end{figure}
%==================Figure================================

The consequence of Lemma~\ref{lemAltTurns} is that the surface angle $\nu$ around $v$
is partitioned by the altitudes $a_i$ in order, because $\sum_i \t_i = \nu$, and
the angle between $a_i$ and $a_{i+1}$ is $\t_i$.
Moreover, Lemma~\ref{lemAltitude} shows that the apexes of each lifted triangle
$\bar{T}^L_i$ lie on those altitudes, at some positive distance from $v$.
Consequently, we can connect those apexes to form a simple geodesic polygon
enclosing $v$ on $L$.
Because every turn angle $\t_i$ is strictly less than $\pi$,
connecting two adjacent apexes along $a_i$ and $a_{i+1}$ 
will keep $v$ to the same (counterclockwise) side. 
Call this polygon the \emph{moat} $M$ of $P$.\footnote{
We do not know whether $M$ is always convex, but we only need it to be simple.}
Fig.~\ref{MoatV_1} illustrates the moat for the 
example in Fig.~\ref{PyHexV_1}(c).
%See Fig.~\ref{MoatV_1}
%==================Figure================================
\begin{figure}
\centering
 \includegraphics[width=0.8\textwidth]{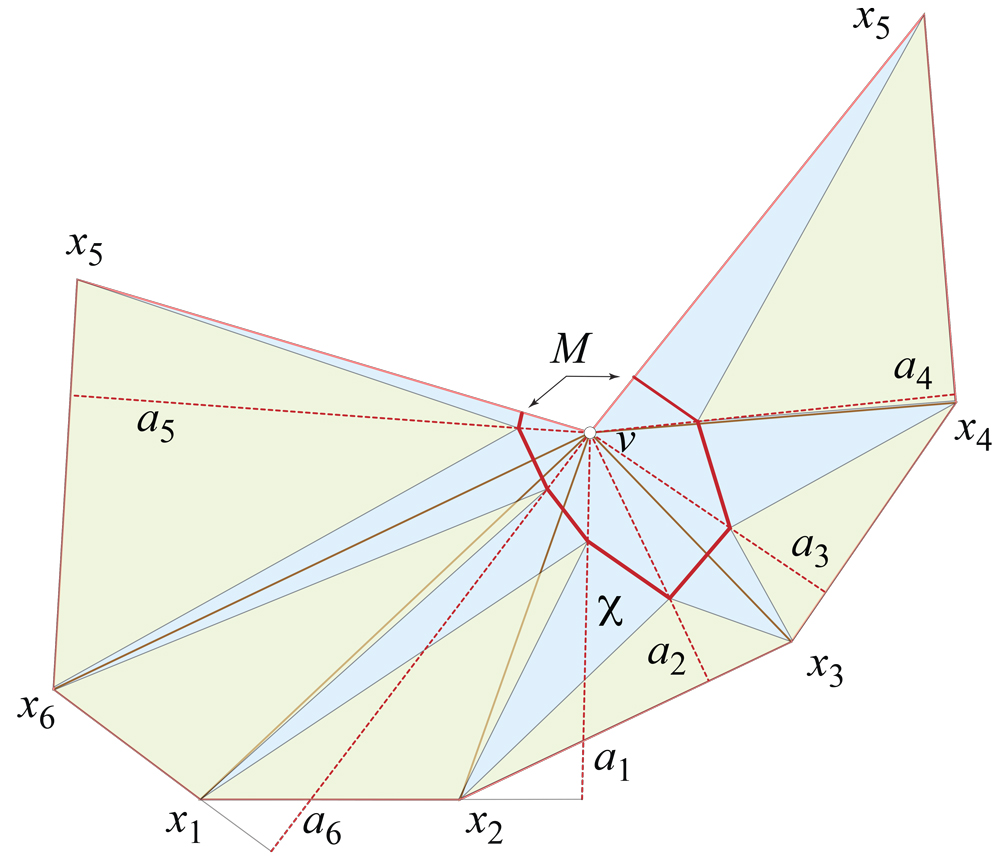}
\caption{The layout from Fig.~\protect\ref{PyHexV_1}(c)
shown with  moat $M$ and altitudes $a_i$ identified.
}
\label{MoatV_1}
\end{figure}
%==================Figure================================

\begin{lm}
\label{lemLifting}
For the case $\bar{v} \in X$, the lifting of all triangles $\bar{T}_i$ to $\bar{T}^L_i$ onto $L$ has the following properties,
(where we shorten ``geodesic triangle" to ``triangle"):
\begin{enumerate}[label=(\alph*)]
\item Each lifted triangle $\bar{T}^L_i$ fits on $L$: $\bar{T}^L_i \subset L$.
\item $v$ does not lie in any triangle $\bar{T}^L_i$.
\item No lifted triangle self-overlaps, and no pair of triangles overlap.
\end{enumerate}
\end{lm}

\begin{proof}
\begin{enumerate}[label=(\alph*)]
\item
Because the apex of the lifted $\bar{T}^L_i$ is on the moat $M$ which surrounds $v$, $\bar{T}^L_i$ remains on the portion of $L$ outside the moat.
\item Therefore no $\bar{T}^L_i$ includes $v$.
\item If we view the overlay of $\bar{L}$ with the opening of $\partial X$ by the angle $\q_i-\bar{\q}_i$ at each $x_i$ image, as in Fig.~\ref{PyHexV_1}(c),
then Cauchy's Arm Lemma shows that two lifted triangles cannot overlap.
Suppose $\bar{T}^L_i$ and $\bar{T}^L_j$ overlap, $i < j$.
Then we can identify two points $p_i \in \bar{T}^L_i$ and $p_j \in \bar{T}^L_y$ that coincide in the layout.
But in $X$, $p_i$ and $p_j$ were separated by a positive distance $d=|p_i p_j|$.
In $X$, draw a convex chain from $p_i$ to $\partial X$, around that boundary, to $p_j$.
The layout opens this chain by the positive angles $\q_i-\bar{\q}_i$, and so in the layout, $p_i$ and $p_j$ must be separated further than $d$, a contradiction.
\end{enumerate}
\end{proof}

\noindent
Lemma~\ref{lemLifting} shows that $\chi$, the region of $L$ not covered by the lifted triangles,
is indeed a crest.

\medskip

We now turn to the case $\bar{v} \not\in X$.
The difficulty here is that $\partial L$ in a layout $\bar{L}$ of $L$ may not be a convex chain,
and Lemma~\ref{lemAltTurns} relies on convexity for the altitudes to connect to $v$
in the same order as the vertices around $X$.
Indeed if $v$ were closer the plane of $X$ in the example in Fig.~\ref{PyHexV_2}(a),
then the angle at $x_3$ would be reflex. In general, a contiguous portion of $\partial L$
could be reflex. Lifting triangles incident to that reflex chain could lead to overlap,
violating~(c) of Lemma~\ref{lemLifting}.

However, as described earlier in Fig.~\ref{PyHexV_2}(c), the crest is formed by clipping the
triangles $\bar{T}_i$ to $X$.
Triangles $\bar{T}_3=x_3 \bar{v} x_4$ and $\bar{T}_4=x_4 \bar{v} x_5$ in Fig.~\ref{PyHexV_2}(b)
fall entirely outside $X$, and so play no role.
The convex portion of $\partial L$ still satisfies Lemma~\ref{lemAltTurns},
so the corresponding altitudes are incident to $v$ in the same order as the vertices
along the convex chain.
This allows us to define a partial moat $M$, and then close it off to a simple polygon by a geodesic  path surrounding $v$.
This is illustrated in Fig.~\ref{MoatV_2}.

%See Fig.~\ref{MoatV_2}
%==================Figure================================
\begin{figure}
\centering
 \includegraphics[width=0.6\textwidth]{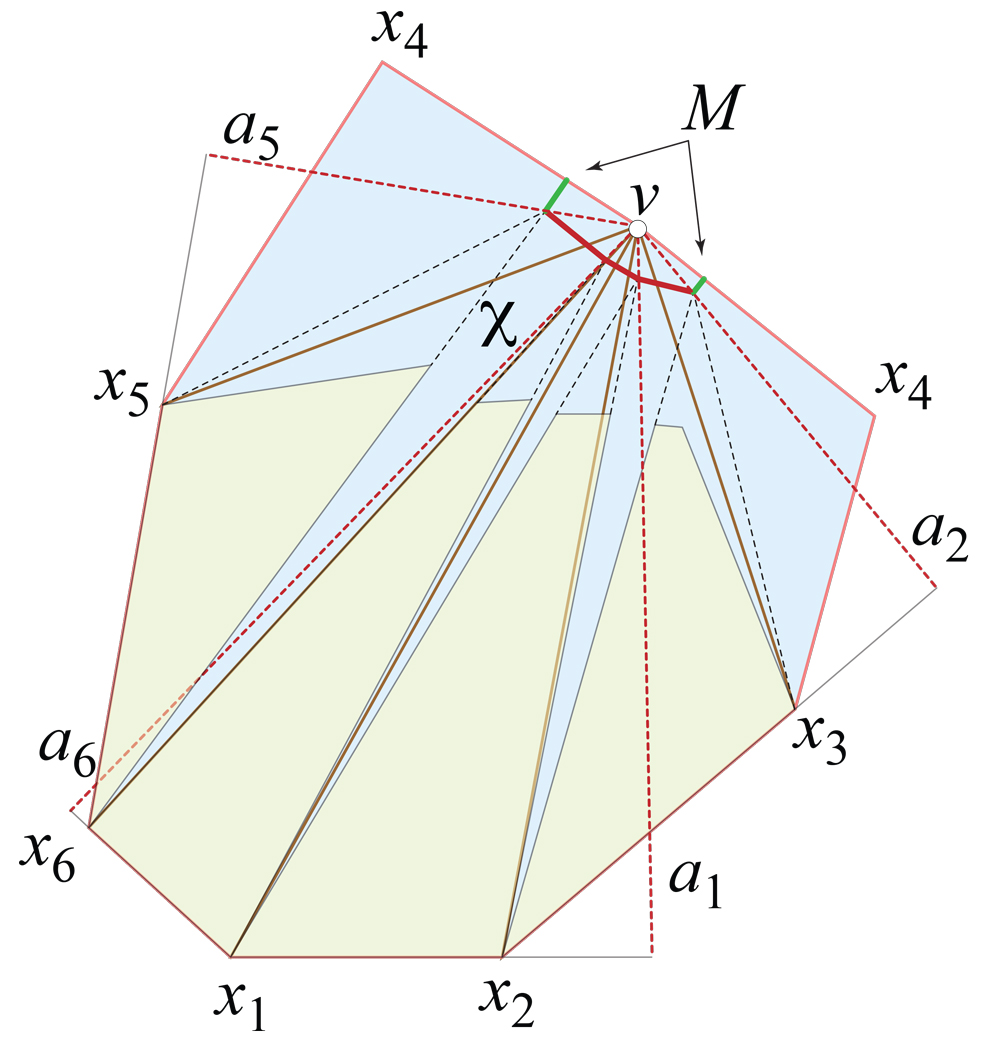}
\caption{The layout from Fig.~\protect\ref{PyHexV_2}(c)
shown with moat $M$ and altitudes $a_i$ marked.
Note $a_3$ and $a_4$ are missing because $T_3$ and $T_4$ are clipped
as outside $X$.
The green edges mark the closing of the partial moat around $v$.
}
\label{MoatV_2}
\end{figure}
%==================Figure================================

This renders Lemma~\ref{lemLifting} true for the lifted triangles along the convex chain of $\partial L$, which are the only ones not clipped entirely away.
We summarize in a theorem:

\begin{thm}
\label{thmCrest}
A crest $\chi$ can be constructed as the portion of $L$ not covered by lifted triangles in the case of $\bar{v} \in X$, 
and clipped lifted triangles in the case $\bar{v} \not\in X$, as described above.
\end{thm}

\noindent
We remark that the same procedure will work for other points $w \in X$ within some neighborhood of $\bar{v} \in X$, resulting in different crests.
However, a simple example shows that not every point $w \in X$ will produce a crest.
Consider $X=x_1 x_2 x_3$ an equilateral triangle of centre $o$, and $v$ close to $o$, with $vo$ orthogonal to $X$.
Take $w \in o x_3$ close to $x_3$. Then the isosceles triangle $w x_1 x_2$ is  larger than the isosceles triangle
$v x_1 x_2$, so no congruent copy of the former can fit inside $L$ without encompassing $v$ and so self-overlapping.

We have this as a counterpart to Theorem~\ref{main}:

\begin{thm}
\label{thmCrestTailoring}
For any convex polyhedra $P$ and $Q$, one can crest-tailor $P$ to any homothetic copy of $Q$ inside $P$, in time $O(n^4)$, where $n=\max\{ |P|, |Q| \}$.
\end{thm}

\begin{proof}
The lemmas leading to Theorem~\ref{main} 
established that ultimately we need to tailor single vertex truncations, i.e., tailor pyramids.
So the claim follows from Theorem~\ref{thmCrest}.
\end{proof}

\paragraph{Complexity Analysis.}
Let $v$ have degree-$k$ with $k=O(n)$.
The overlays of the planar layouts shown in Figs.~\ref{PyHexV_1} and~\ref{PyHexV_2} can be constructed in linear time.
To compute the points along the edges $x_i v$ at which the geodesic triangle edges cross will cost $O(n)$ per geodesic.
To clip the lifted triangle to $X$ naively costs $O(n)$ per triangle,
but with some care we believe this could be accomplished also in amortized linear time overall.
Nevertheless the total cost of computing the crest $\chi$ is $O(n^2)$,
and we believe there are examples with total combinatorial complexity 
(number of geodesic/edge intersections) of $\Omega(n^2)$.

Note that $O(n^2)$ is the same complexity for reducing $P$ to $X$ via digon-tailoring described in Section~\ref{Tailor_1}.

We are assuming that the combinatorial complexity of the crest $\chi$ on $L$
determines the time complexity of computing the crest.
However, a crest flattened to the plane, not overlayed on $\bar{L}$, has
combinatorial complexity $O(n)$---$k$ (possibly clipped) triangles---and can be constructed in $O(n)$ time.
There are circumstances where implicit representations of 
shortest paths suffice for subsequent uses, such as supporting queries. A prominent example is
the optimal algorithm for shortest paths on a convex polyhedron~\cite{SchreiberSharir}.
An implicit representation of $\chi$ can be constructed in linear time.
In contrast, digon-tailoring to flatten $L$ to $X$ appears to require quadratic time even for 
an implicit representation.

%%%%%%%%%%%%%%%%%%%%%%%%%%%%%%%%%%%%%%%%%%%%%%%%%%%%%%%%%%%%%%%%%

\section{Surface enlarging and reshaping}
\label{enlarging}

In this short section we consider first the problem of enlarging surfaces, 
the opposite in some sense to our goal in previous sections.
Formally, consider convex polyhedra $P$ and $Q \subset P$.
We show how to locally cut open $Q$ and insert surface pieces into the cuts to obtain $P$.

Previous sections have established three methods of tailoring $P$ to $Q$:
\begin{itemize}
\squeezelist
\item Theorem~\ref{main} digon-tailors according to a sculpting of $P$ to $Q$.
\item Theorem~\ref{main-} digon-tailors, without sculpting, $P$ to a (possibly small) homothet of $Q$.
\item Theorem~\ref{thmCrestTailoring} crest-tailors according to a sculpting of $P$ to $Q$.
\end{itemize}
For the two digon-tailorings,
we tailor $P$ to $Q$ and afterward reverse the process.
During the tailoring process, we keep track of the cuts and digons removed.
Then, starting with $Q$, we cut each geodesic and insert the corresponding digon,
in reverse order.
For crest-tailoring, the process is the same, except that we cut the boundary of crests
and insert the removed crests in reverse order.

These, together with the corresponding algorithm analyses, lead to %this
the next result.

\begin{thm}
\label{Enlarge_Algorithm1}
Given convex polyhedra $P$ and $Q$ of at most $n$ vertices each, and $Q \subset P$,
$Q$ can be enlarged to $P$, following any of the above three processes, in time $O(n^4)$.
\end{thm}

\noindent
We also obtain an analogue of Corollary \ref{approx_tailoring}:

\begin{co}
\label{approx_enlarging}
For any convex polyhedron $Q$ and any convex surface $S \supset Q$, one can enlarge $Q$ to approximate $S$.
\end{co}

We are now in the position to solve the general problem: for any convex polyhedra $P$ and $Q$, reshape $Q$ to $P$ by tailoring and enlarging.
The solution is, at this point, obvious: enlarge $Q$ sufficiently to include $P$, and afterward tailor it. For completeness, we state next the formal results.

\begin{thm}
\label{reshape}
Given convex polyhedra $P$ and $Q$ of at most $n$ vertices each,
one can reshape $Q$ to $P$, following any of the above three processes and their reverse processes, in time $O(n^4)$.
\end{thm}

\noindent
In particular, this result extends Thm.~\ref{main-} to ``any'' $Q$ inside $P$.

\begin{co}
\label{approx_reshaping}
For any convex polyhedron $Q$ and any convex surface $S$, one can reshape $Q$ as in Theorem~\ref{reshape} to approximate $S$.
\end{co}

%%%%%%%%%%%%%%%%%%%%%%%%%%%%%%%%%%%%%%%%%%%%%%%%%%%%%%%%%%%%%%%%%

\section{$P$-unfoldings}
\label{P-unfoldings}

As a direct consequence of Subsection~\ref{enlarging}, we can answers positively the following problem, apparently not previously explored.

Given convex polyhedra $P$ and $Q \subset P$, can one cut-up the surface $Q$ so that the pieces may be pasted onto $P$, non-overlapping,
and so form an isometric subset of $P$? 
Call this a \emph{$P$-unfolding} of $Q$, or an unfolding of $Q$ to the surface $P$.
The use of the word ``unfolding" here is intentionally suggestive, but note that planar unfoldings of polyhedra are generally connected.

\begin{thm}
\label{thmP-unf}
Given convex polyhedra $P$ and $Q$ of at most $n$ vertices each, and $Q \subset P$, a $P$-unfolding of $Q$ can be determined in time $O(n^4)$, following 
Theorem~\ref{Enlarge_Algorithm1}.
\end{thm}

\begin{proof}
Just enlarge $Q$ to $P$, and then remove all inserted digons or crests.
The result is a subset of $P$ isometric to the cut-up $Q$.
\end{proof}

To justify the use of ``unfolding," here is an example of a connected $P$-unfolding that is
not a planar \emph{net}, i.e., a non-overlapping simple, planar polygon.
Let $Q$ be the classical thin, nearly flat tetrahedron with an overlapping edge-unfolding.
See, e.g., \cite[Fig.~22.8, p.~314]{DOR}.
Take $P$ to be a slightly larger homothet of $Q$.
Then the same edge-cuts embed $Q$ onto $P$ without overlap.

As with the ``fewest nets" unfolding problem~\cite[Prob.~22.1, p.~308]{DOR}, there could be interest in minimizing the number of disconnected pieces of a $P$-unfolding.
We believe the $P$-unfolding question could also be answered by applying the Wallace-Bolyai-Gerwien dissection theorem.
However, this will result in a ``pseudopolynomial number of pieces"~\cite{abbott2012hinged}.\footnote{
``pseudopolynomial means polynomial in the combinatorial complexity ($n$) and the dimensions of an integer grid on which the input is drawn."}
A future task is exploring bounds on the number of pieces achieved by Theorem~\ref{thmP-unf}.

\medskip

Simple examples show that the $P$-unfolding of $Q$ produced by Theorem~\ref{thmP-unf} is not necessarily simply connected, i.e., without holes.
However, in general, that is indeed the case, as shown by the following result. The next paragraph explains the meaning of ``in general.''

Consider the space ${\cal S}$ of all convex surfaces, endowed with the topology induced by the usual Pompeiu-Hausdorff metric.
Fix some $P \in {\cal S}$.
Consider in ${\cal S}$ the subset ${\cal P}={\cal P}_P^n$ of all polyhedra $Q\subset P$ with %at most
precisely $n$ vertices, with the induced topology.
Two polyhedra in ${\cal P}$ are then close to each other if and only if they have close respective vertices.
``General'' refers to polyhedra $Q$ in an open and dense subset of ${\cal P}$.

\begin{thm}
\label{thmQPflat}
For any convex polyhedron $P$ and any $n \in \N$, there exists a subset ${\cal Q}={\cal Q}_P^n$ open and dense in ${\cal P}$, 
such that the $P$-unfolding $Q_P$ of each $Q \in {\cal Q}$ is flat (i.e., contains no internal vertices), and is simply connected.
\end{thm}

\begin{proof}
Assume we have some convex polyhedron $Q \subset P$ such that $Q_P$ contains an internal vertex $v$, and so the curvatures of $P$ and $Q$ at $v$ are equal: $\o_Q(v)=\o_P(v)$.
Slightly alter the position of the vertices of $Q$, to get $\o_Q(w) \neq \o_P(u)$, for any vertices $w \in Q$ and $u \in P$.
Of course, this remains valid in a small neighborhood of the new $Q$.

Assume now we have some convex polyhedron $Q \subset P$ such that $Q_P$ is not simply connected, 
i.e., $Q_P$ contains a noncontractible curve $\s \subset P \cap Q_P$.
The Gauss-Bonnet Theorem shows that the total curvature $\Omega_Q (\s)$ of $Q_P$ inside $\s$ equals the total curvature $\Omega_P (\s)$ of $P$ inside $\s$: $\Omega_Q (\s) = \Omega_P (\s)$. We next show that every such $Q$ that violates the theorem can be approximated with polyhedra that do satisfy the theorem.

Slightly alter the position of the vertices of $Q$, to get a new polyhedron $Q'$ on which the following property (V) is verified. 
(V): any partial sum of vertex curvatures is different from any partial sum of vertex curvatures in $P$. 
This implies that,  for any simple closed curve $\t$ on $Q'$, $\Omega_{Q'} (\t)$ cannot be written as the sum of vertex curvatures of $P$.
Therefore, $Q'_P$ has no curve in common with $P$, noncontractible in $Q'_P$. And so $Q'$ does satisfy the theorem.

Since the property (V) is valid on a neighborhood $N$ of $Q'$, it follows that all polyhedra in $N$ do satisfy the theorem.
\end{proof}

%%%%%%%%%%%%%%%%%%%%%%%%%%%%%%%%%%%%%%%%%%%%%%%%%%%%%%%%%%%%%%%%%

\section{Remarks and Open Problems}
\label{Open}

Our work leaves open several questions of various natures, some of which  have been mentioned in the text.
Others are presented in this section, accompanied with related remarks.

\medskip

{\bf 1.} 
Cutting $P$ with no restrictions and zipping closed to obtain a homothetic copy of $Q$ can be easily solved as follows.
Unfold $Q$ to the star-unfolding $S_Q$ with respect to some point $p$, shrink $S_Q$ to fit inside a face $F$ of $P$, cut out $S_Q$ from $F$, refold and glue.
(This bears some similarity with Theorem~\ref{thmTailorNoSculpting}.)
Possibly more efficient in terms of surface area is an interesting suggestion of Jin-ichi Itoh~\cite{I}: 
first star-unfold $P$ to $S_P$, shrink $S_Q$ to fit inside $S_P$, and
then cut out $S_Q$ from $S_P$.
Clearly the same strategy could be followed for any non-overlapping unfoldings of $P$ and $Q$, but {\it completing the details might be quite challenging}.

\medskip

{\bf 2.} 
We have presented three methods for tailoring, of very different flavours.
On one hand, the method of tailoring with sculpting given in Sections~\ref{domes}--\ref{finer} seems appropriate for local tailoring, and it can produce any $Q$ inscribed in $P$.
This is true for either digon-tailoring or crest-tailoring.

On the other hand, the method of tailoring without sculpting given in Section~\ref{no_sculpt} is purely intrinsic, in that it doesn't need the spatial structure of $P$ and $Q$ to work. 
The surfaces can be given, in this case, as a collection of polygons glued together as in AGT.
But it has the disadvantage of being ``non-economical,'' in the sense that it ``discards'' a lot of 
$P$'s surface area; see Problem~$3$ below. 

All three methods can be reversed to enlarge surfaces, where the results are the same, but not-requiring the spatial structure might be a clear advantage.

\medskip

{\bf 3.}
We do not know if any of our algorithm complexities are best possible. Perhaps several could be improved, or lower bounds established.
Especially notable is the problem of {\it finding a cut-locus generic point $x$ on a convex polyhedron of $n$ vertices in less than $O(n^4)$ time}.
See Theorem~\ref{thmAlgorithmNoSculpt}.

\medskip

{\bf 4.} 
In the class of convex surfaces, the lower bound for the ratio of area to diameter is $0$, because line-segments can be approximated by convex bodies.
Therefore, one could necessarily lose almost all surface area in the process of tailoring.
Also, in order to approximate by tailoring a sphere inscribed in a very long convex surface, one 
is forced to lose almost all surface area.

{\it What is the minimum amount of surface area one has to discard, in order to approximate a sphere of diameter $1$ inscribed in a convex body $K$ of diameter $d>1$, by tailoring $K$?}

\medskip

{\bf 5.} 
Consider a convex polyhedron $P$, mark a face $X$ of it as a base, and let $D$ be a fixed g-dome over $X$, interior to $P$.
In view of Theorem~\ref{thmDomePyr}, one could ask the following question.

{\it Is it possible to partition $P$ into pyramids and $D$, with planes through the edges of $X$?
After each sectioning we move away the sliced pyramid.}

Our procedure described in Theorem~\ref{thmSliceGDomes} uses many $X$, one per slice, and many g-domes per $X$. Here we are asking
whether a single $X$ and a single g-dome might suffice.

The restriction to domes is necessary, otherwise the answer is {\sc no}, as simple examples can show.
The restriction to planes through base sides is also necessary,  otherwise the answer is {\sc yes}, implied by Theorem~\ref{thmSliceGDomes} and the proof of Theorem~\ref{thmSliceGDomes}.

\medskip

{\bf 6.} 
We have shown, with different methods, that the unfolding of $Q \subset P$ onto $P$ exists. Clearly, the $P$-unfolding depends on the order of tailoring operations.
The $P$-unfolding does not seem to be necessarily a connected subset of $P$.

{\it Is there some method and/or orderings of operations that would render 
a $P$-unfolding connected?}

Notice that Theorem~\ref{thmQPflat} established simply connectedness in general.
But simply connected does not imply connected, e.g., the union of several disjoint disks is a disconnected set but simply connected.

\medskip

{\bf 7.} 
It would be interesting to study tailoring for other classes of convex surfaces.

\setlength{\leftskip}{8mm}

\medskip
\noindent{\bf 7.1.}
Despite Corollary~\ref{approx_tailoring}, which shows that any convex surface can be
approximated as closely as desired, it does not seem possible to start with a strictly convex surface and tailor it to a polyhedron.
However, one can sculpt a surface to a polyhedron.

\medskip
\noindent{\bf 7.2.}
Tailoring spheres is limited to digons between antipodal points, and the only possible results are constant curvature ``spindles,'' but one could continue with tailoring spindles. 
It seems that at least a part of the present work could  apply to $1$-polyhedra.
These are polyhedra whose faces are (congruent to) geodesic polygons on the unit sphere. 
They can approximate convex surfaces with curvature bounded below by $1$
(in the sense of A.~D. Alexandrov), 
just as convex polyhedra can approximate ordinary 
convex surfaces~\cite{AZ67},~\cite{itoh2015moderate}.

\medskip
\noindent{\bf 7.3.}
One could also define tailoring for general (i.e., not necessarily convex) polyhedra, of arbitrary topology.
Clearly, a necessary condition for $Q$ to be tailored from a homothetic copy of $P$ is to have the same topology as $P$. 
Our Theorem~\ref{thmMainTailoring} might suggest this is also sufficient, but that is not true.
By Alexandrov's Gluing Theorem (AGT), tailoring a convex polyhedron always produces a convex polyhedron, 
never a nonconvex polyhedron homeomorphic to the sphere but having negative curvature at some vertex.
There is as yet no counterpart to AGT for nonconvex polyhedra.
Therefore, in the general framework, tailoring could be a much subtler topic.

%%%%%%%%%%%%%%%%%%%%%%%%%%%%%%%%%%%%%%%%%%%%%%%%%%%%%%%%%%%%%%%%%

%%%%%%%%%%%%%%%%%%%%%%%%%%%%%%%%%%%%%%%%%%%%%%%%%%%%%%%%%%%%%%%%%


\begin{thebibliography}{99}

\bibitem{abbott2012hinged}
T.~G. Abbott, Z.~Abel, D.~Charlton, E.~D. Demaine, M.~L. Demaine, and S.~D.  Kominers.
\newblock {\em Hinged dissections exist.}
\newblock Discrete \& Comput. Geom., 47(1):150--186, 2012.

\bibitem{aaos-supa-97}
P.~K. Agarwal, B.~Aronov, J.~O'Rourke, and C.~A. Schevon.
\newblock {\em Star unfolding of a polytope with applications.}
\newblock SIAM J. Comput., 26:1689--1713, 1997.

\bibitem{code2} A.~D. Alexandrov, 
{\it Convex Polyhedra}, 
Springer-Verlag, Berlin, 2005. Monographs in Mathematics. 
Translation of the 1950 Russian edition by N. S. Dairbekov, S. S. Kutateladze, and A. B. Sossinsky.

\bibitem{Al_SW} A.~D. Alexandrov, {\sl Selected Works II, Intrinsic Geometry of Convex Surfaces}, 
Chapman $\&$ Hall/CRC, Taylor $\&$ Francis Group, 2006.

\bibitem{AZ67}
A.~D. Aleksandrov and V.~A. Zalgaller, 
\newblock {\em Intrinsic Geometry of Surfaces}, 
\newblock Transl. Math. Monographs, Providence, RI, Amer. Math. Soc., 1967.

\bibitem{ao}  B. Aronov and J. O'Rourke,
{\it Nonoverlap of the star unfolding},
Discrete Comput. Geom. 8 (1992), 219--250.

\bibitem{BI} A.~I. Bobenko and I. Izmestiev, 
{\it Alexandrov's theorem, weighted Delaunay triangulations, and mixed volumes}, 
Annales de l'Institut Fourier, Grenoble 58 (2008), 447--505.

\bibitem{ChenHan1}
J. Chen and Y. Han,
\newblock {\em Shortest paths on a polyhedron.}
\newblock In Proc. 6th Annual Symposium Comput. Geom., pages 360--369,
  1990.

\bibitem{ChenHan2}
J. Chen and Y. Han,
\newblock {\em Shortest paths on a polyhedron, Part {I}: {C}omputing shortest paths.}
\newblock Internat. J. Comput. Geom. \& Appl., 6(02):127--144, 1996.

\bibitem{DOR}
E.~D. Demaine and J. O'Rourke, 
{\it Geometric Folding Algorithms: Linkages, Origami, Polyhedra}, 
Cambridge University Press, 2007.

\bibitem{Eppstein}  D. Eppstein,
{\it Treetopes and their Graphs},
Discrete Comput. Geom. (2020), 1--31.

\bibitem{Epp-Loff} D. Eppstein and M. L\"offler, 
{\it Bounds on the Complexity of Halfspace Intersections when the Bounded Faces have Small Dimension},
Discrete Comput. Geom. 50 (2013), 1--21.

%\bibitem{Firsching} M. Firsching,
%{\it Computing Maximal Copies of Polyhedra Contained in a Polyhedron},
%Experimental Math. 24 (1), 2015, 98-105

\bibitem{I} J. Itoh, {\it private communication}, 2020.

\bibitem{INV}
J. Itoh, C. Nara and C. V\^\i lcu, {\it Continuous flattening of convex polyhedra}, 
in A. M\'arquez et al. (Eds.), {\sl Computational Geometry}, Springer LNCS vol. 7579 (2012), 85--97.

\bibitem{itoh2015moderate}
J. Itoh, J.~Rouyer, and C.~V{\^\i}lcu.
\newblock {\em Moderate smoothness of most Alexandrov surfaces.}
\newblock Internat. J. Math., 26(04):1540004, 2015.

\bibitem{KiazykLubiw}
S. Kiazyk and A. Lubiw,
\newblock {\em Star unfolding from a geodesic curve.}
\newblock Discrete \& Comput. Geom., 56(4):1018--1036, 2016.

\bibitem{o-ecala-01}
J.~O'Rourke.
\newblock {\em An extension of Cauchy's arm lemma with application to curve
  development.}
\newblock In Proc. 2000 Japan Conf. Discrete Comput. Geom., volume 2098
  of Lecture Notes Comput. Sci., pages 280--291. Springer-Verlag, 2001.
  
\bibitem{JOR07} J. O'Rourke, 
{\it Computational geometry column 49},
Internat. J. Comput. Geom. Appl., 38(2): 51--55, 2007. Also in SIGACT News,
38(2): 51-55(2007), Issue 143.

\bibitem{o-vtcp-2020}
J.~O'Rourke.
\newblock Vertex-transplants on a convex polyhedron.
\newblock In {\em Proc. 32st Canad. Conf. Comput. Geom.}, Aug. 2020.
%\newblock To appear.

\bibitem{OV}
J. O'Rourke and C. V\^\i lcu, {\it Development of Curves on Polyhedra via Conical Existence},
Comput. Geom. 47 (2014), 149--163.

\bibitem{ro1}  J. Rouyer, 
{\it Antipodes sur un t\'etra\`edre r\'egulier}, 
J. Geom. 77 (2003), 152--170.

\bibitem{SchreiberSharir}
Y. Schreiber and M. Sharir,
\newblock {\em An optimal-time algorithm for shortest paths on a convex polytope in
  three dimensions.}
\newblock Discrete \& Comput. Geom., pages 500--579, 2008.

\bibitem{SS}  
M. Sharir and A. Schorr, 
{\it On shortest paths in polyhedral spaces}, SIAM J. Comput. 15 (1986), 193--215.

%\bibitem{Wiki-surg} Surgery: \url{https://en.wikipedia.org/wiki/Surgery_theory}.

\bibitem{Zal} V. A. Zalgaller, {\it An isoperimetric problem for tetrahedra}.
J. Math. Sci. 140 (2007), 511--527.

\end{thebibliography}
\end{document}